\documentclass[pdf,sn-mathphys-num]{sn-jnl}

\usepackage{amsmath}
\usepackage{amssymb}
\usepackage{paralist}
\usepackage[misc]{ifsym}
\usepackage{epsfig} 
\usepackage{epstopdf} 
\usepackage{hyperref}

\theoremstyle{thmstyleone}%
\newtheorem{theorem}{Theorem}
\newtheorem{proposition}[theorem]{Proposition}%

\theoremstyle{thmstyletwo}%
\newtheorem{remark}{Remark}%
\newtheorem{lemma}{Lemma}
\newtheorem{corollary}{Corollary}
\theoremstyle{thmstylethree}%

\title[Stackelberg LQDG with follower's cheap control]{Suboptimal open-loop solution of a Stackelberg linear-quadratic differential game with cheap control of a follower: analytical/numerical study}

\author[1]{\fnm{Valery Y.} \sur{Glizer}}\email{valgl120@gmail.com}
\author[1]{\fnm{Vladimir} \sur{Turetsky}}\email{turetsky1@braude.ac.il}
\affil[1]{\orgdiv{Department of Mathematics}, \orgname{Braude College of Engineering},\\ \orgaddress{\street{Snunit 51, P.O.B. 78}, \city{Karmiel}, \postcode{2161002},  \country{Israel}}}

\begin{document}

\abstract{
A two-player finite horizon linear-quadratic Stackelberg differential game is considered. The feature of this game is that the control cost of a follower in the cost functionals of both players is small, which means that the game under consideration is a cheap control game. The open-loop solution of this game is studied. Using the game's solvability conditions, obtaining such a game's solution is reduced to the solution of a proper boundary-value problem. Due to the smallness of the follower's control cost, this boundary-value problem is singularly perturbed. The asymptotic behaviour of the solution to this problem is analysed. Based on this analysis, the asymptotic behaviour of the open-loop optimal players' controls and the optimal values of the cost functionals is studied. Using these results, asymptotically suboptimal players' controls are designed. An illustrative example of a supply chain problem with a small control cost of a retailer is presented.}

\keywords{Stackelberg linear-quadratic differential game, cheap control game, singularly perturbed boundary-value problem, asymptotic analysis, open-loop suboptimal players' controls}

\maketitle

\section{Introduction}

In this paper, we consider a two-player finite horizon linear-quadratic Stackelberg differential game. For this game, we study the case where the control cost of a follower in the cost functionals of both players is small, meaning that this game is a cheap control game. More generally, a cheap control problem is an extremal control problem in which a control cost of at least one decision maker is much smaller than the cost of the state variable in at least one cost functional of the problem. Cheap control problems are important in qualitative and quantitative analysis of many topics in optimal control, $H_{\infty}$-control, and differential game theories. Thus, such problems are important in: (1) qualitative study and analytical/numerical computation of singular controls and arcs (see, e.g, \cite{Omalley-1974,Bell-Jacobson-1975,Omalley-conf-1975,Omalley-1975,Omalley-1976,Omalley-1977,Kurina-1977,
Sannuti-Wason-1985,Saberi-Sannuti-1987,Smetannikova-Sobolev-2005,Glizer-SIAM-2012,Glizer-Optimization,Shinar-Glizer-Turetsky-IGTR,Glizer-Kelis-book});
(2) analysis of limiting forms and maximally achievable accuracy of optimal regulators and filters (see, e.g., \cite{Kwakernaak-Sivan-1972,Francis-1979,Saberi-Sannuti-1986,Lee-Bien-1987,Braslavsky-et-al-1999,Seron-et-al-1999,Glizer-Kelis-NACO-2018}); (3) inverse optimal control problems (see e.g. \cite{Moylan-Anderson-1973});
(4) optimal control problems with high gain control in dynamics (see, e.g., \cite{Young-Kokotovic-Utkin-1977,Kokotovic-Khalil-OReilly-1986}).

Cheap control differential games and closely related singular differential games appear in various applications. For example, such games appear in robust interception problems with uncertainties (see, e.g., \cite{Turetsky-Glizer-IGTR-2007}), in robust tracking problems with uncertainties (see, e.g., \cite{Turetsky-Glizer-Shinar-IJSS-2014,Turetsky-2016}), in pursuit-evasion problems (see, e.g., \cite{Shinar-Glizer-Turetsky-IGTR,Starr-Ho-1969,Turetsky-Shinar-Automatica,Turetsky-JOTA-2004,
Turetsky-Glizer-Axioms,Glizer-NoDEA-2000,Glizer-Axims-2021}), in robust investment
problems (see, e.g., \cite{Hu-Oksendal-Sulem}), in biology processes (see e.g. \cite{Hamelin-Lewis}), in supply chain problems (see \cite{Turetsky-Glizer-IMA-OR-2025}).

The smallness of the control cost yields a singularly perturbed boundary-value problem associated with a cheap control problem by solvability (control optimality) conditions.

As aforementioned, in this paper, we consider a cheap control differential game. Cheap control differential games were extensively studied in the literature. Thus, various zero-sum cheap control games were analyzed in \cite{Glizer-Optimization,Shinar-Glizer-Turetsky-IGTR,Glizer-Kelis-book,Turetsky-Glizer-IGTR-2007,Turetsky-Glizer-Shinar-IJSS-2014,
Turetsky-Shinar-Automatica,Turetsky-JOTA-2004,Turetsky-Glizer-Axioms,Glizer-NoDEA-2000,Petersen-1986}. In \cite{Glizer-Axims-2021,Glizer-PAFA-2022,Glizer-ASVAO}, different cheap control Nash equilibrium games were studied. However, to the best of our knowledge, a cheap control Stackelberg equilibrium differential game was considered only in two works \cite{Glizer-Turetsky-Axioms-2024,Turetsky-Glizer-IMA-OR-2025} where the case of the leader's cheap control was studied.

It should be noted the following. Similarly to a zero-sum differential game and a Nash equilibrium differential game, a Stackelberg equilibrium differential game is not a cooperative game. Furthermore, similarly to a Nash equilibrium differential game, a Stackelberg equilibrium differential game is a nonzero-sum differential game. However, a Stackelberg equilibrium differential game differs considerably from a zero-sum differential game and from a Nash equilibrium differential game. Namely, in a zero-sum differential game and in a Nash equilibrium differential game, players act simultaneously, while in a Stackelberg equilibrium differential game, they do not. A Stackelberg equilibrium differential game is a hierarchical game. In this game, one of the players (the leader) makes its best decision based on the set of possible best decisions of the other player (the follower). Thus, the leader has an advantage in comparison to the follower. More details on a Stackelberg equilibrium differential game can be found in the book \cite{Basar-Olsder}. There are different real-life applications of a Stackelberg equilibrium differential game. Thus, such a game is used in the analysis and optimization of supply and marketing channels (see, e.g.,  \cite{Turetsky-Glizer-IMA-OR-2025,He-Prasad-et-al-2007,Colombo-Labrecciosa-2019,Kanska-Wiszniewska-Matyszkiel-2022} and references therein). Application of a Stackelberg equilibrium differential game to a security (defenders-attackers) problem is presented in \cite{Solis-Clempner-2020} (see also references therein). In \cite{Wan-et-al-2024}, the problem of a transponder’s communication channel with multiple-user terminals is modeled by a Stackelberg equilibrium differential game. Application of a Stackelberg equilibrium differential game to analysis and optimization of power systems is proposed in \cite{Ming-et-al-2024}.

A small cost of the follower's control in its cost functional (a cheap control of the follower) can considerably decrease the aforementioned advantage of the leader. This circumstance can be important, for instance, in a supply chain problem in the case where the optimization of a state-dependent part of the follower's (the retailer's) cost functional is much more essential than the optimization of its control cost in this cost functional.

The main motivation of this paper is to study a novel nontrivial class of cheap control differential games. More precisely, in this paper, we consider a Stackelberg finite horizon linear-quadratic differential game with two players (a leader and a follower) and with a small cost of the follower's control in both cost functionals. We study the asymptotic behaviour of an open-loop solution to this game. For this purpose, we consider the boundary-value problem associated with the Stackelberg game by the solvability conditions. Due to the small cost of the follower's control, the aforementioned boundary-value problem is singularly perturbed in the conditionally stable case. An asymptotic solution of this problem is formally constructed and justified. Based on this solution, asymptotic expansions of the optimal controls of the leader and the follower are constructed. Using these results, suboptimal controls of the players are formally designed and justified. Asymptotic approximations of the optimal values of the cost functionals are also derived. The theoretical results are applied to an asymptotic solution of a supply chain problem with a single manufacturer and a single retailer.

It should be noted that a cheap control problem is closely related to a control problem for a singularly perturbed system. Namely, the former can be converted to the latter by a proper change of the control variable with the small (cheap) cost (see, e.g., \cite{Sannuti-1983,Kokotovic-1984,Popescu-Gajic-1999,Hoai-2017}). Linear-quadratic Stackelberg differential games for singularly perturbed systems were studied in the works \cite{Salman-Cruz-1979,
Khalil-Medanic-1980,Salman-Cruz-1983,Mizukami-Xu-1990,Mizukami-Suzumura-1993,Mukaidani-2007}. However, in these works (in contrast with the present paper), the infinite horizon version of the game was considered, and a suboptimal closed-loop (feedback) solution to the considered game was derived.

The paper is organized as follows. In the next section (Section 2), the Stackelberg cheap control game is rigorously formulated. Main notions and assumptions are introduced. Objectives of the paper are stated. In Section 3, the solvability conditions of the considered game are presented. Asymptotic solution of the singularly perturbed boundary-value problem, appearing in these solvability conditions, is formally constructed and justified in Section 4. Section 5 is devoted to obtaining the main results of the paper. In this section, the asymptotic expansion of the game's solution is formally derived and justified. Also, the suboptimal players' controls are formally designed and justified. In Section 6, an illustrative example of a supply chain problem with a small control
cost of a retailer is presented. Conclusions are given in Section 7.

The following main notations are applied in the paper.

\begin{enumerate}
\item	$\mathbb{R}^{n}$ denotes the $n$-dimensional real Euclidean space;
\item	$\|\cdot\|$ denotes the Euclidean norm either of a vector ($\|z\|$) or of a matrix ($\|A\|$);
\item   $\mathbb{C}[t_{1},t_{2}; \mathbb{R}^{q}]$ denotes the linear space of all functions $f(t): [t_{1},t_{2}] \rightarrow \mathbb{R}^{q}$ continuous in the interval $[t_{1},t_{2}]$;
\item	the superscript $"T"$ denotes the transposition either of a vector ($z^{T}$) or of a matrix ($A^{T}$);
\item $I_{n}$ denotes the identity matrix of dimension $n$;
\item ${\rm col}\big(x_{1},x_{2},...,x_{k}\big)$, where $x_{1} \in \mathbb{R}^{n_1}$, $x_{2} \in \mathbb{R}^{n_2}$,..., $x_{k} \in \mathbb{R}^{n_{k}}$, denotes the column block-vector of the dimension $n_{1} + n_{2} + ... +n_{k}$ with the upper block $x_{1}$, the next block $x_{2}$ and so on, and the lower block $x_{k}$;
\item ${\rm Re}(\mu)$ denotes the real part of a complex number $\mu$.
\end{enumerate}

\section{Problem statement}\label{sec2}
\subsection{Initial game formulation}\label{subsect-2.1}

The dynamics of the game is described by the following differential equation:
\begin{equation}
\frac{dZ(t)}{dt}={\mathcal{A}}(t)Z(t)+{\mathcal{B}}_{\rm u}(t)u(t)+{\mathcal{B}}_{\rm v}(t)v(t),\  \  \ t\in [0,t_{f}],\  \  \ Z(0)=Z_{0},  \label{or-eq}
\end{equation}
where $Z(t)\in \mathbb{R}^{n}$ is the state vector, $u(t)\in \mathbb{R}^{r}$, ($r \le n$),
$v(t)\in \mathbb{R}^{s}$, ($s < n$) are the players' controls; $t_{f}> 0$ is a given time-instant; ${\mathcal{A}}(t)$ and ${\mathcal{B}}_{i}(t)$, $(i = {\rm u,v})$, are given matrix-valued functions of corresponding dimensions continuous for $t \in [0,t_{f}]$; $Z_{0}\in \mathbb{R}^{n}$ is a given vector.

The cost functionals of the player with the control $u(t)$ (the leader) and the player
with the control $v(t)$ (the follower) are, respectively,
\begin{eqnarray}\label{perf-ind-1}{\mathcal{J}}_{\rm u}(u,v)=
\frac{1}{2}\int_{0}^{t_{f}}\big[Z^{T}(t){\mathcal{D}}_{{\rm u}}(t)Z(t) + u^{T}(t)G_{{\rm u, u}}(t)u(t)
+ \varepsilon^{2}v^{T}(t)G_{{\rm u, v}}(t)v(t)\big]dt,
\end{eqnarray}
and
\begin{eqnarray}\label{perf-ind-2}{\mathcal{J}}_{\rm v}(u,v)=
\frac{1}{2}\int_{0}^{t_{f}}\big[ Z^{T}(t){\mathcal{D}}_{{\rm v}}(t)Z(t) + \varepsilon^{2}v^{T}(t)G_{{\rm v,v}}(t)v(t)
+ u^{T}(t)G_{{\rm v,u}}(t)u(t)\big]dt,
\end{eqnarray}
where, for any $t \in [0,t_{f}]$, ${\mathcal{D}}_{i}(t)$, $G_{i,j}(t)$, $(i={\rm u,v};\ j={\rm u,v})$ are given symmetric
matrices of corresponding dimensions; the matrix-valued functions ${\mathcal{D}}_{i}(t)$, $G_{i,j}(t)$, $(i={\rm u,v};\ j={\rm u,v})$ are continuous for $t \in [0,t_{f}]$; $\varepsilon >0$ is a small parameter.

In what follows, we assume:
\newline \textbf{(A1)} For any $t \in [0,t_{f}]$, the matrix ${\mathcal{B}}_{\rm v}(t)$ has full column rank $s$.
\newline \textbf{(A2)} For any $t \in [0,t_{f}]$, the matrices
${\mathcal{D}}_{i}(t)$, $(i={\rm u,v})$ are positive semi-definite.
\newline \textbf{(A3)} For any $t \in [0,t_{f}]$,
the matrices $G_{{\rm u,u}}(t)$, $G_{{\rm v,v}}(t)$, $G_{{\rm u,v}}(t)$ and $G_{{\rm v,u}}(t)$ are positive definite.

\begin{remark}\label{mathcal-G_uu-G_vv=I}In what follows, for the sake of technical simplicity (but without loss of generality), we set $G_{\rm u,u}(t) \equiv I_{r}$, $G_{\rm v,v}(t) \equiv I_{s}$, $t \in [0,t_{f}]$.\end{remark}

We will solve the game, given by the differential equation (\ref{or-eq}) and the cost functionals (\ref{perf-ind-1}) and (\ref{perf-ind-2}), with respect to the open-loop Stackelberg (hierarchical) equilibrium (solution) $\big(u^{*}(t),v^{*}(t)\big)$. This equilibrium is defined as follows in the set of all pairs $\big(u(t),v(t)\big)$, $u(t) \in \mathbb{C}[0,t_{f}; \mathbb{R}^{r}]$, $v(t) \in \mathbb{C}[0,t_{f}; \mathbb{R}^{s}]$ (for more details see, e.g.,~\cite{Basar-Olsder,Dockner-2000}).

For any $u(t)\in \mathbb{C}[0,t_{f}; \mathbb{R}^{r}]$, let
\begin{equation}\label{v-0-definition}v^{0}\big(t;u(t)\big) \stackrel{\triangle}{=} {\rm argmin}_{v(t)\in \mathbb{C}[0,t_{f}; \mathbb{R}^{s}]}{\mathcal{J}}_{\rm v}\big(u(t),v(t)\big)
\end{equation}
along the trajectories of (\ref{or-eq}).

Furthermore, let
\begin{equation}\label{u*-definition}u^{*}(t) \stackrel{\triangle}{=} {\rm argmin}_{u(t)\in \mathbb{C}[0,t_{f}; \mathbb{R}^{r}]}{\mathcal{J}}_{\rm u}\Big(u(t),v^{0}\big(t;u(t)\big)\Big)
\end{equation}
along the trajectories of (\ref{or-eq}) with $v(t) = v^{0}\big(t;u(t)\big)$.

Moreover, let
\begin{equation}\label{v*-definition}v^{*}(t) \stackrel{\triangle}{=} v^{0}\big(t;u^{*}(t)\big).
\end{equation}

Finally, the pair $\big(u^{*}(t),v^{*}(t)\big)$ is the open-loop Stackelberg solution of the game (\ref{or-eq})--(\ref{perf-ind-2}).

We call the trajectory of the differential equation (\ref{or-eq}), generated by the pair of controls $\big(u(t) = u^{*}(t) , v(t) = v^{*}(t)\big)$, the Stackelberg optimal trajectory of the game (\ref{or-eq})--(\ref{perf-ind-2}). Also, we call the values ${\mathcal{J}}_{\rm u}\big(u^{*}(t),v^{*}(t)\big)$ and ${\mathcal{J}}_{\rm v}\big(u^{*}(t),v^{*}(t)\big)$ the Stackelberg optimal values of the cost functionals ${\mathcal{J}}_{\rm u}(u,v)$ and ${\mathcal{J}}_{\rm v}(u,v)$, respectively, in the game (\ref{or-eq})--(\ref{perf-ind-2}).

\subsection{Transformation of the game (\ref{or-eq})--(\ref{perf-ind-2})}

In what follows, we assume:\\
\newline \textbf{(A4)} $\det\Big({\mathcal{B}}_{\rm v}^{T}(t){\mathcal{D}}_{\rm v}(t){\mathcal{B}}_{\rm v}(t)\Big) \neq 0 $\  \ $\forall t \in [0,t_{f}]$.
\newline \textbf{(A5)} The matrix-valued functions ${\mathcal{A}}(t)$, ${\mathcal{B}}_{\rm u}(t)$, ${\mathcal{D}}_{\rm u}(t)$, $G_{{\rm u, v}}(t)$ and $G_{{\rm v, u}}(t)$ are continuously differentiable in the interval $[0,t_{f}]$.
\newline \textbf{(A6)} The matrix-valued functions ${\mathcal{B}}_{\rm v}(t)$ and ${\mathcal{D}}_{\rm v}(t)$ are twice continuously differentiable in the interval $[0,t_{f}]$.

By ${\mathcal B}_{c}(t)$, we denote a complement matrix to the matrix ${\mathcal B}_{\rm v}(t)$ for any $t \in [0,t_{f}]$. Thus, the dimension of the matrix ${\mathcal B}_{c}(t)$ is $n\times(n-s)$, and the block-form matrix $\big({\mathcal B}_{c}(t),{\mathcal B}_{\rm v}(t)\big)$ is invertible for all $t\in [0,t_{f}]$. Due to the assumption (A6) and the results of the book \cite{Glizer-Kelis-book} (Section 3.3), the matrix-valued function ${\mathcal B}_{c}(t)$ can be chosen as twice continuously differentiable in the interval $[0,t_{f}]$.

Consider the following matrix-valued functions of $t \in [0,t_{f}]$:
\begin{equation}\label{matr-L} {\mathcal L}_{\rm v}(t) = {\mathcal B}_{c}(t) - {\mathcal B}_{\rm v}(t)\big({\mathcal B}_{\rm v}^{T}(t){\mathcal D}_{\rm v}(t){\mathcal B}_{\rm v}(t)\big)^{- 1}{\mathcal B}_{\rm v}^{T}(t){\mathcal D}_{\rm v}(t){\mathcal B}_{c}(t),\ \ {\mathcal R}_{\rm v}(t) = \big({\mathcal L}_{\rm v}(t),{\mathcal B}_{\rm v}(t)\big).
\end{equation}

\begin{remark}\label{properties-mathcal-R} By virtue of the results of \cite{Glizer-Kelis-book} (Section 3.3), the matrix ${\mathcal R}_{\rm v}(t)$ is invertible for all $t \in [0,t_{f}]$. Moreover, the matrix-valued function ${\mathcal R}_{\rm v}(t)$ is twice continuously differentiable with respect to  $t \in [0,t_{f}]$.
\end{remark}

Using the matrix-valued function ${\mathcal{R}}_{\rm v}(t)$, we make the following transformation of the state variable in the game (\ref{or-eq})--(\ref{perf-ind-2}):
\begin{equation}\label{state-transform}
Z(t)={\mathcal R}_{\rm v}(t)z(t),\  \  \  \  t \in[0,t_f],
\end{equation}
where $z(t)\in \mathbb{R}^{n}$ is a new state variable.

Using  Remark \ref{mathcal-G_uu-G_vv=I}, we obtain (quite similarly to the results of \cite{Glizer-Kelis-book} (Section 3.3)) the following assertion.
\begin{proposition}\label{init-problem-transf}
Let the assumptions (A1)--(A6) be valid. Then,
the transformation (\ref{state-transform}) converts the system (\ref{or-eq}) and the functionals
(\ref{perf-ind-1}) and (\ref{perf-ind-2}) into the following new system and functionals:
\begin{equation}\label{new-eq}
\frac{dz(t)}{dt}=A(t)z(t)+B_{\rm u}(t)u(t)+B_{\rm v}(t)v(t),\  \  \ t\in [0,t_{f}],\  \  \ z(0)=z_{0},
\end{equation}
\begin{equation}\label{new-perf-ind-1}J_{\rm u}(u,v)=
\frac{1}{2}\int_{0}^{t_{f}}\big[
z^{T}(t)D_{\rm u}(t)z(t) + u^{T}(t)u(t) + \varepsilon^{2}v^{T}(t)G_{\rm u,v}(t)v(t)\big]dt,
\end{equation}
\begin{equation}\label{new-perf-ind-2}J_{\rm v}(u,v)=\frac{1}{2}
\int_{0}^{t_{f}}\big[
z^{T}(t)D_{\rm v}(t)z(t) + \varepsilon^{2}v^{T}(t)v(t) + u^{T}(t)G_{\rm v,u}(t)u(t)\big]dt,
\end{equation}
where
\begin{equation}\label{new-matr-A}
A(t)={\mathcal R}_{\rm v}^{-1}(t)\big[{\mathcal{A}}(t){\mathcal R}_{\rm v}(t)-d{\mathcal R}_{\rm v}(t)/dt\big],
\end{equation}
\begin{equation}\label{new-matr-B_uv}
B_{\rm u}(t)={\mathcal{R}}_{\rm v}^{-1}(t){\mathcal{B}}_{\rm u}(t),\  \  \  B_{\rm v}(t)={\mathcal{R}}_{\rm v}^{-1}(t){\mathcal{B}}_{\rm v}(t) =\left(\begin{array}{l} 0 \\
I_{s}\end{array}\right),
\end{equation}
\begin{equation}
D_{\rm u}(t)={\mathcal{R}}_{\rm v}^{T}(t){\mathcal{D}}_{\rm u}(t)\mathcal{R}_{\rm v}(t), \label{new-matr-D_u}
\end{equation}
\begin{equation}\begin{array}{llll}
D_{\rm v}(t)=\mathcal{R}_{\rm v}^{T}(t){\mathcal{D}}_{\rm v}(t)\mathcal{R}_{\rm v}(t)=\left(
\begin{array}{l}
D_{{\rm v},1}(t)\  \  \  \  \ 0 \\
0\  \  \  \  \  \  \  \  \  \  \  \ D_{{\rm v},2}(t)\end{array}\right), \\ \\
D_{{\rm v},1}(t)\ ={\mathcal{L}}_{\rm v}^{T}(t){\mathcal{D}}_{\rm v}(t){\mathcal{L}}_{\rm v}(t),\  \  \
D_{{\rm v},2}(t)={\mathcal{B}}_{\rm v}^{T}(t){\mathcal{D}}_{\rm v}(t){\mathcal{B}}_{\rm v}(t),\end{array}\label{new-matr-D_v}\end{equation}
\begin{equation}\label{z_0} z_{0}={\mathcal{R}}_{\rm v}^{-1}(0)Z_{0}.\end{equation}

The matrix $D_{\rm v,1}(t)$ is symmetric positive semi-definite,
while the matrix $D_{\rm v,2}(t)$ is symmetric positive definite for all $t\in[0,t_f]$.
Moreover, the matrix-valued functions $A(t)$, $B_{\rm u}(t)$, $D_{\rm u}(t)$, $G_{\rm u,v}(t)$, and $G_{\rm v,u}(t)$ are continuously differentiable in the interval $[0,t_f]$, while the matrix-valued function $D_{\rm v}(t)$ is twice continuously differentiable in the interval $[0,t_f]$.
\end{proposition}

\begin{remark}\label{equiv-original-transormed-games} It should be noted the following. The open-loop Stackelberg solution to the differential game \mbox{(\ref{new-eq})--(\ref{new-perf-ind-2})} is defined similarly to such a solution to the game (\ref{or-eq})--(\ref{perf-ind-2}). Furthermore, since the matrix ${\mathcal R}_{\rm v}(t)$ is invertible for all $t \in [0,t_{f}]$, then the transformation (\ref{state-transform}) is invertible. Hence, the games (\ref{or-eq})--(\ref{perf-ind-2}) and (\ref{new-eq})--(\ref{new-perf-ind-2}) have the same open-loop Stackelberg solution (if it exists), i.e., these games are equivalent to each other. Furthermore, the Stackelberg optimal values of the cost functionals of the leader and the follower in the game (\ref{new-eq})--(\ref{new-perf-ind-2}) are equal to such values in the game (\ref{or-eq})--(\ref{perf-ind-2}). Moreover, due to the form of the matrices $B_{\rm v}(t)$ and $D_{\rm v}(t)$, the new game (\ref{new-eq})--(\ref{new-perf-ind-2}) is much simpler than the initially formulated game (\ref{or-eq})--(\ref{perf-ind-2}). Therefore, in what follows in the paper, we deal with the game (\ref{new-eq})--(\ref{new-perf-ind-2}).
\end{remark}

\subsection{Main objectives of the paper}

The objectives of the paper are:

\begin{itemize}
\item[(i)] To construct and justify an asymptotic expansion with respect to $\varepsilon$ of the open-loop Stackelberg solution $\big(u^{*}(t,\varepsilon), v^{*}(t,\varepsilon)\big)$ to the game (\ref{new-eq})--(\ref{new-perf-ind-2}).
\item[(ii)] To derive an asymptotically suboptimal Stackelberg solution to the game (\ref{new-eq})--(\ref{new-perf-ind-2}), {i.e., the pair of the admissible controls $\big(\check{u}(t,\varepsilon), \check{v}(t,\varepsilon)\big)$, such that}
\begin{eqnarray}\label{definition-suboptimal}\lim_{\varepsilon \rightarrow + 0}\big|J_{\rm u}\big(\check{u}(t,\varepsilon), \check{v}(t,\varepsilon)\big) -  J_{\rm u}^{*}(\varepsilon)\big| = 0,\nonumber\\
\lim_{\varepsilon \rightarrow + 0}\big|J_{\rm v}\big(\check{u}(t,\varepsilon), \check{v}(t,\varepsilon)\big) -  J_{\rm v}^{*}(\varepsilon)\big| = 0.
\end{eqnarray}
\item[(iii)] To construct and justify $\varepsilon$-free approximations of the Stackelberg optimal values $J_{\rm u}^{*}(\varepsilon) = J_{\rm u}\big(u^{*}(t,\varepsilon), v^{*}(t,\varepsilon)\big)$, $J_{\rm v}^{*}(\varepsilon) = J_{\rm v}\big(u^{*}(t,\varepsilon), v^{*}(t,\varepsilon)\big)$ of the cost functionals in the game (\ref{new-eq})--(\ref{new-perf-ind-2}), valid for all sufficiently small $\varepsilon>0$.
\end{itemize}

\section{Solvability Conditions of the Stackelberg Game (\ref{new-eq})--(\ref{new-perf-ind-2})}\label{sec3}

For a given $\varepsilon > 0$, we consider the following boundary-value problem in the time-interval $[0,t_{f}]$:
\begin{eqnarray}\label{initial-bound-value-pr}
\frac{dz(t)}{dt} = A(t)z(t) - S_{\rm u}(t)\lambda_{\rm u}(t) - S_{\rm v}(t,\varepsilon)\lambda_{\rm v}(t),\  \  \  \ z(0) = z_{0},\nonumber\\
\frac{d\lambda_{\rm u}(t)}{dt} = - D_{\rm u}(t)z(t) - A^{T}(t)\lambda_{\rm u}(t) + D_{\rm v}(t)\mu(t),\  \  \  \ \lambda_{\rm u}(t_{f}) = 0,\nonumber\\
\frac{d\lambda_{\rm v}(t)}{dt} = - D_{\rm v}(t)z(t) - A^{T}(t)\lambda_{\rm v}(t),\  \  \  \ \lambda_{\rm v}(t_{f}) = 0,\nonumber\\
\frac{d\mu(t)}{dt} = S_{\rm v}(t,\varepsilon)\lambda_{\rm u}(t) - S_{\rm u,v}(t,\varepsilon)\lambda_{\rm v}(t) + A(t)\mu(t),\  \  \  \ \mu(0) = 0,\nonumber\\
\end{eqnarray}
where
\begin{eqnarray}\label{matr-S_u-S_v-Suv} S_{\rm u}(t
) = B_{\rm u}(t)B_{\rm u}^{T}(t),\nonumber\\
S_{\rm v}(t,\varepsilon) = \frac{1}{\varepsilon^{2}}B_{\rm v}(t)B_{\rm v}^{T}(t)  = \left(\begin{array}{l}0\  \  \ 0\\
0\  \  \ \frac{1}{\varepsilon^{2}}I_{s}\end{array}\right),\nonumber\\
S_{\rm u,v}(t,\varepsilon) = \varepsilon^{2}B_{\rm v}(t)G_{\rm u,v}(t)B_{\rm v}^{T}(t) = \left(\begin{array}{l}0\  \  \ 0\\
0\  \  \ \varepsilon^{2}G_{\rm u,v}(t)\end{array}\right).\nonumber\\
\end{eqnarray}

By virtue of the results of the book~\cite{Basar-Olsder} (Section 7.6), the following assertion is valid.
\begin{proposition}\label{Stackelberg-solution} Let the assumptions
(A1)--(A6) be satisfied. Then, for any given $\varepsilon > 0$, the boundary-value problem (\ref{initial-bound-value-pr}) has the unique solution
${\rm col}\big(z(t,\varepsilon),\lambda_{\rm u}(t,\varepsilon),\allowbreak\lambda_{\rm v}(t,\varepsilon),\mu(t,\varepsilon)\big)$. Moreover, the Stackelberg game (\ref{new-eq})--(\ref{new-perf-ind-2}) has the unique open-loop solution $\big(u^{*}(t,\varepsilon),v^{*}(t,\varepsilon)\big)$, where
\begin{equation}\label{u*-v*}
u^{*}(t,\varepsilon) = - B_{\rm u}^{T}(t)\lambda_{\rm u}(t,\varepsilon),\  \  \  \ v^{*}(t,\varepsilon) = - \frac{1}{\varepsilon^{2}}B_{\rm v}^{T}(t)\lambda_{\rm v}(t,\varepsilon),\  \  \  \ t \in [0,t_{f}].
\end{equation}

The component $z(t,\varepsilon)$ of the solution to the boundary-value problem (\ref{initial-bound-value-pr}) is the Stackelberg optimal trajectory of the game (\ref{new-eq})--(\ref{new-perf-ind-2}).
\end{proposition}

\section{Asymptotic solution of the
boundary-value problem (\ref{initial-bound-value-pr})}\label{sec4}

\subsection{Transformation of the problem (\ref{initial-bound-value-pr})}

For all $t\in [0,t_{f}],$ let us partition the matrices $A(t)$, $D_{\rm u}(t)$ and $S_{\rm u}(t)$ into blocks as

\begin{eqnarray}\label{matr-A-D_u-S_u-blocks} A(t) = \left(\begin{array}{l}A_{1}(t)\  \  \ A_{2}(t)\\ A_{3}(t)\  \  \ A_{4}(t)\end{array}\right),\ \ \ D_{\rm u}(t) = \left(\begin{array}{l}D_{\rm u,1}(t)\  \  \ D_{\rm u,2}(t)\\
D_{\rm u,2}^{T}(t)\  \  \ D_{\rm u,3}(t)\end{array}\right),\nonumber\\ S_{\rm u}(t) = \left(\begin{array}{l}S_{\rm u,1}(t)\  \  \ S_{\rm u,2}(t)\\
S_{\rm u,2}^{T}(t)\  \  \ S_{\rm u,3}(t)\end{array}\right),
\end{eqnarray}
where the matrices $A_{1}(t)$, $A_{2}(t)$, $A_{3}(t)$, and $A_{4}(t)$ are of the dimensions $(n - s)\times(n - s)$, $(n - s)\times s$, $s\times(n - s)$ and $s\times s$, respectively; the matrices $D_{\rm u,1}(t)$, $D_{\rm u,2}(t)$ and $D_{\rm u,3}(t)$ are of the dimensions $(n - s)\times (n - s)$, $(n - s)\times s$ and $s\times s$, respectively; the matrices $S_{\rm u,1}(t)$, $S_{\rm u,2}(t)$ and $S_{\rm u,3}(t)$ are of the dimensions $(n - s)\times (n - s)$, $(n - s)\times s$ and $s\times s$, respectively; for any $t \in [0,t_{f}]$, the matrices $D_{\rm u,1}(t)$, $D_{\rm u,3}(t)$, $S_{\rm u,1}(t)$ and $S_{\rm u,3}(t)$ are symmetric.

Furthermore, we partition the vector $z_{0}$ into blocks as
\begin{equation}\label{z_0-blocks}z_{0} = \left(\begin{array}{c}z_{01}\\ z_{02}\end{array}\right),\  \  \  \ z_{01} \in \mathbb{R}^{n-s},\  \  \  \ z_{02} \in \mathbb{R}^{s}.
\end{equation}

Due to the expression of $S_{\rm v}(t,\varepsilon)$ (see the equation (\ref{matr-S_u-S_v-Suv})), the right-hand sides of the first and the fourth differential equations in (\ref{initial-bound-value-pr}) are singular for $\varepsilon = 0$. Taking into account this singularity, we look for the components $z(t,\varepsilon)$, $\lambda_{\rm u}(t,\varepsilon)$, $\lambda_{\rm v}(t,\varepsilon)$, and $\mu(t,\varepsilon)$ of the solution to the boundary-value problem (\ref{initial-bound-value-pr}) in the following block form:
\begin{eqnarray}\label{block-form-solution}
z(t,\varepsilon) = \left(\begin{array}{c}z_{1}(t,\varepsilon)\\ z_{2}(t,\varepsilon)\end{array}\right),\  \  \  \   \lambda_{\rm u}(t,\varepsilon)  = \left(\begin{array}{c}\lambda_{\rm u1}(t,\varepsilon)\\ \varepsilon\lambda_{\rm u2}(t,\varepsilon)\end{array}\right),\nonumber\\
 \lambda_{\rm v}(t,\varepsilon)  = \left(\begin{array}{c}\lambda_{\rm v1}(t,\varepsilon)\\ \varepsilon\lambda_{\rm v2}(t,\varepsilon)\end{array}\right),\  \  \  \ \mu(t,\varepsilon)  = \left(\begin{array}{c}\mu_{1}(t,\varepsilon)\\ \mu_{2}(t,\varepsilon)\end{array}\right),\nonumber\\
\end{eqnarray}
where the vectors $z_{1}(t,\varepsilon)$, $\lambda_{\rm u1}(t,\varepsilon)$, $\lambda_{\rm v1}(t,\varepsilon)$ and $\mu_{1}(t,\varepsilon)$ are of the dimension $(n - s)$; the vectors $z_{2}(t,\varepsilon)$, $\lambda_{\rm u2}(t,\varepsilon)$, $\lambda_{\rm v2}(t,\varepsilon)$ and $\mu_{2}(t,\varepsilon)$ are of the dimension $s$.

Substitution of the block forms of the matrices
$D_{\rm v}(t)$, $S_{\rm v}(t,\varepsilon)$, $S_{\rm u,v}(t)$, $A(t)$, $D_{\rm u}(t)$, $S_{\rm u}(t)$ (see the equations (\ref{new-matr-D_v}), (\ref{matr-S_u-S_v-Suv}), (\ref{matr-A-D_u-S_u-blocks})), and the block forms of the vectors $z_{0}$, $z(t)$, $\lambda_{\rm u}(t)$, $\lambda_{\rm v}(t)$, $\mu(t)$
(see the equations (\ref{z_0-blocks}) and (\ref{block-form-solution})) into the boundary-value problem (\ref{initial-bound-value-pr}) yields, after a routine algebra, the following equivalent boundary-value problem in the time interval $[0,t_{f}]$:
\begin{eqnarray}\label{equiv-system}
\frac{dz_{1}(t,\varepsilon)}{dt} = A_{1}(t)z_{1}(t,\varepsilon) + A_{2}(t)z_{2}(t,\varepsilon) - S_{\rm u,1}(t)\lambda_{\rm u1}(t,\varepsilon) - \varepsilon S_{\rm u,2}(t)\lambda_{\rm u2}(t,\varepsilon),\nonumber\\
\varepsilon\frac{dz_{2}(t,\varepsilon)}{dt} = \varepsilon A_{3}(t)z_{1}(t,\varepsilon) + \varepsilon A_{4}(t)z_{2}(t,\varepsilon) - \varepsilon S_{\rm u,2}^{T}(t)\lambda_{\rm u1}(t,\varepsilon)\nonumber\\ - \varepsilon^{2}S_{\rm u,3}(t)\lambda_{\rm u2}(t,\varepsilon) - \lambda_{\rm v2}(t,\varepsilon),\nonumber\\
\frac{d\lambda_{\rm u1}(t,\varepsilon)}{dt} = - D_{\rm u,1}(t)z_{1}(t,\varepsilon) - D_{\rm u,2}(t)z_{2}(t,\varepsilon) - A_{1}^{T}(t)\lambda_{\rm u1}(t,\varepsilon)\nonumber\\
- \varepsilon A_{3}^{T}(t)\lambda_{\rm u2}(t,\varepsilon) + D_{\rm v,1}(t)\mu_{1}(t,\varepsilon),\nonumber\\
\varepsilon\frac{d\lambda_{\rm u2}(t,\varepsilon)}{dt} = - D_{\rm u,2}^{T}(t)z_{1}(t,\varepsilon) - D_{\rm u,3}(t)z_{2}(t,\varepsilon) - A_{2}^{T}(t)\lambda_{\rm u1}(t,\varepsilon)\nonumber\\
- \varepsilon A_{4}^{T}(t)\lambda_{\rm u2}(t,\varepsilon) + D_{\rm v,2}(t)\mu_{2}(t,\varepsilon),\nonumber\\
\frac{d\lambda_{\rm v1}(t,\varepsilon)}{dt} = - D_{\rm v,1}(t)z_{1}(t,\varepsilon) - A_{1}^{T}(t)\lambda_{\rm v1}(t,\varepsilon) - \varepsilon A_{3}^{T}(t)\lambda_{\rm v2}(t,\varepsilon),\nonumber\\
\varepsilon\frac{d\lambda_{\rm v2}(t,\varepsilon)}{dt} = - D_{\rm v,2}(t)z_{2}(t,\varepsilon) - A_{2}^{T}(t)\lambda_{\rm v1}(t,\varepsilon) - \varepsilon A_{4}^{T}(t)\lambda_{\rm v2}(t,\varepsilon),\nonumber\\
\frac{d\mu_{1}(t,\varepsilon)}{dt} = A_{1}(t)\mu_{1}(t,\varepsilon) + A_{2}(t)\mu_{2}(t,\varepsilon),\nonumber\\
\varepsilon\frac{d\mu_{2}(t,\varepsilon)}{dt} = \lambda_{\rm u2}(t,\varepsilon) - \varepsilon^{4}G_{\rm u,v}(t)\lambda_{\rm v2}(t,\varepsilon) + \varepsilon A_{3}(t)\mu_{1}(t,\varepsilon) + \varepsilon A_{4}(t)\mu_{2}(t,\varepsilon),\nonumber\\
\end{eqnarray}
\begin{eqnarray}\label{equiv-cond} z_{1}(0,\varepsilon) = z_{0 1},\  \  \  \ z_{2}(0,\varepsilon) = z_{0 2},\  \  \  \ \lambda_{\rm u1}(t_{f},\varepsilon) = 0,\  \  \  \ \lambda_{\rm u2}(t_{f},\varepsilon) = 0,\nonumber\\
\lambda_{\rm v1}(t_{f},\varepsilon) = 0,\  \  \  \ \lambda_{\rm v2}(t_{f},\varepsilon) = 0,\  \  \  \ \mu_{1}(0,\varepsilon) = 0,\  \  \  \ \mu_{2}(0,\varepsilon) = 0.
\end{eqnarray}

\begin{remark}\label{partitioning-matr-A-D_v} It is important to note that the above made partitioning the matrices $A(t)$, $D_{\rm u}(t)$, $S_{\rm u}(t)$ and the vector $z_{0}$ (see the equations (\ref{matr-A-D_u-S_u-blocks}) and (\ref{z_0-blocks})), as well as the representation of the state variables $z(t)$, $\lambda_{\rm u}(t)$, $\lambda_{\rm v}(t)$, $\mu(t)$ in the block form (see the equation (\ref{block-form-solution})), allow us to partition each equation of the differential system in (\ref{initial-bound-value-pr}) into two considerably different modes (see the equation (\ref{equiv-system})). The first of these modes is with the multiplier $1$ for the derivative, while the second mode is with the small multiplier $\varepsilon > 0$ for the derivative. This means that the state variable, described by the second mode, varies much faster than the state variable, described by the first mode. Therefore, the second mode and the corresponding state variable are called the fast mode and the fast state variable. The first mode and the corresponding state variable are called the slow mode and the slow state variable. Thus, the aforementioned matrices' partitioning and the state variables' representation allow us to convert the boundary-value problem (\ref{initial-bound-value-pr}) into the explicit form of a singularly perturbed boundary-value problem, namely, the problems
(\ref{equiv-system})-(\ref{equiv-cond}).
\end{remark}
\begin{remark}\label{comparison-cheap-leader-follower} It is important to note the following. The singularly perturbed boundary-value problem (\ref{equiv-system})-(\ref{equiv-cond}), arising in the analysis of the Stackelberg differential game with the cheap control of the follower, has four slow modes and four fast modes. This structure of the problem (\ref{equiv-system})-(\ref{equiv-cond}) differs considerably from the structure of the singularly perturbed boundary-value problem, arising in the analysis of the Stackelberg differential game with the cheap control of the leader (see \cite{Glizer-Turetsky-Axioms-2024}). The latter problem has four slow modes and only two fast modes. The difference in the structures of two singularly perturbed boundary-value problems, being the consequence of the hierarchical character of the Stackelberg differential game, yields a considerable difference in the asymptotic analysis of these problems and in the results of this analysis.
\end{remark}

\subsection{First-order formal asymptotic solution of the problem (\ref{equiv-system})-(\ref{equiv-cond})}\label{formal-first-order-sol}

Based on the Boundary Functions Method (see, e.g., \cite{Vasil'eva-Butuzov-Kalachev,Esipova-Diff-Eq}), we look for the first-order asymptotic solution $${\rm col}\big(z_{1}^{1}(t,\varepsilon),z_{2}^{1}(t,\varepsilon),\lambda_{\rm u1}^{1}(t,\varepsilon),\lambda_{\rm u2}^{1}(t,\varepsilon),
\lambda_{\rm v1}^{1}(t,\varepsilon),\lambda_{\rm v2}^{1}(t,\varepsilon),\mu_{1}^{1}(t,\varepsilon),\mu_{2}^{1}(t,\varepsilon)\big)$$ of the problem (\ref{equiv-system})-(\ref{equiv-cond}) in the form
\begin{eqnarray}\label{first-order-asympt-solution}
z_{j}^{1}(t,\varepsilon) = \sum_{k = 0}^{1}\varepsilon^{k}\big[\bar{z}_{j,k}(t) + z_{j,k}^{0}(\xi) + z_{j,k}^{f}(\rho)\big],
\nonumber\\
\lambda_{{\rm u}j}^{1}(t,\varepsilon) = \sum_{k = 0}^{1}\varepsilon^{k}\big[\bar{\lambda}_{{\rm u}j,k}(t) + \lambda_{{\rm u}j,k}^{0}(\xi) + \lambda_{{\rm u}j,k}^{f}(\rho)\big],\nonumber\\
\lambda_{{\rm v}j}^{1}(t,\varepsilon) =
\sum_{k = 0}^{1}\varepsilon^{k}\big[\bar{\lambda}_{{\rm v}j,k}(t) + \lambda_{{\rm v}j,k}^{0}(\xi) + \lambda_{{\rm v}j,k}^{f}(\rho)\big],\nonumber\\
\mu_{j}^{1}(t,\varepsilon) =
\sum_{k = 0}^{1}\varepsilon^{k}\big[\bar{\mu}_{j,k}(t) + \mu_{j,k}^{0}(\xi) + \mu_{j,k}^{f}(\rho)\big],\nonumber\\
\end{eqnarray}
where $(j = 1,2)$,
\begin{eqnarray}\label{xi-rho}\xi = \frac{t}{\varepsilon},\  \  \  \  \ \rho = \frac{t - t_{f}}{\varepsilon}.\end{eqnarray}

\begin{remark}\label{explanation-zero-order-solution} In the equation (\ref{first-order-asympt-solution}), the terms with the overbar constitute the so-called outer solution, the terms with the superscript $0$ are the boundary corrections in the right-hand neighborhood of $t = 0$, the terms with the superscript $f$ are the boundary corrections in the left-hand neighborhood of $t = t_{f}$. Equations and boundary conditions for the asymptotic solution terms are obtained by substituting $z_{j}^{1}(t,\varepsilon)$, $\lambda_{{\rm u}j}^{1}(t,\varepsilon)$,
$\lambda_{{\rm v}j}^{1}(t,\varepsilon)$ and $\mu_{j}^{1}(t,\varepsilon)$, $(j = 1,2)$ from
(\ref{first-order-asympt-solution}) into the problem (\ref{equiv-system})-(\ref{equiv-cond}) instead of $z_{j}(t)$, $\lambda_{{\rm u}j}(t)$, $\lambda_{{\rm v}j}(t)$ and
$\mu_{j}(t)$, $(j = 1,2)$, respectively, and equating the coefficients for the same power of $\varepsilon$ on both sides of the resulting equations, separately depending on $t$, on $\xi$, and on $\rho$. Additionally, it should be noted the following. If $\varepsilon \rightarrow + 0$, then, for any $t \in (0,t_{f}]$, $\xi \rightarrow + \infty$, while, for any $t \in [0,t_{f})$, $\rho \rightarrow - \infty$.
\end{remark}

\subsubsection{Obtaining the boundary corrections $z_{1,0}^{0}(\xi)$, $\lambda_{{\rm u}1,0}^{0}(\xi)$, $\lambda_{{\rm v}1,0}^{0}(\xi)$ and $\mu_{1,0}^{0}(\xi)$}\label{boundary-correction-1-0}\label{left-zero-bound-corr}

These boundary corrections satisfy the differential equations
\begin{eqnarray}\label{eq-for-bound-corr-l}
\frac{dz_{1,0}^{0}(\xi)}{d\xi} = 0,\
 \  \  \ \xi \ge 0,\nonumber\\
\frac{d\lambda_{{\rm u}1,0}^{0}(\xi)}{d\xi} = 0,\  \  \  \ \xi \ge 0,\nonumber\\
\frac{d\lambda_{{\rm v}1,0}^{0}(\xi)}{d\xi} = 0,\
 \  \  \ \xi \ge 0,\nonumber\\
\frac{d\mu_{1,0}^{0}(\xi)}{d\xi} = 0,\  \  \  \ \xi \ge 0.\nonumber\\
\end{eqnarray}

To obtain the unique solutions of these equations, we should give some additional conditions for the unknown functions. By virtue of the Boundary Functions Method, these conditions are
\begin{equation}\label{cond-for-bound-corr-l}
\lim_{\xi \rightarrow + \infty}z_{1,0}^{0}(\xi)= 0,\  \ \lim_{\xi \rightarrow + \infty}\lambda_{{\rm u}1,0}^{0}(\xi)= 0,\  \
\lim_{\xi \rightarrow + \infty}\lambda_{{\rm v}1,0}^{0}(\xi)= 0,\  \ \lim_{\xi \rightarrow + \infty}\mu_{1 ,0}^{0}(\xi)= 0.\end{equation}

The set of the equations (\ref{eq-for-bound-corr-l}) subject to the conditions (\ref{cond-for-bound-corr-l}) has the unique solution
\begin{equation}\label{bound-corr-1}z_{1,0}^{0}(\xi) \equiv 0,\  \  \ \lambda_{{\rm u}1,0}^{0}(\xi) \equiv 0,\  \  \ \lambda_{{\rm v}1,0}^{0}(\xi) \equiv 0,\  \  \ \mu_{1,0}^{0}(\xi) \equiv 0,\  \  \ \xi \ge 0.
\end{equation}

\subsubsection{Obtaining the boundary corrections $z_{1,0}^{f}(\rho)$, $\lambda_{{\rm u}1,0}^{f}(\rho)$, $\lambda_{{\rm v}1,0}^{f}(\rho)$ and $\mu_{1,0}^{f}(\eta)$}\label{boundary-corrections-f-0}

For these boundary corrections, we obtain the following differential equations:
\begin{eqnarray}\label{eq-for-bound-corr-2}
\frac{dz_{1,0}^{f}(\rho)}{d\rho} = 0,\
 \  \  \ \rho \le 0,\nonumber\\
\frac{d\lambda_{{\rm u}1,0}^{f}(\rho)}{d\rho} = 0,\  \  \  \ \rho \le 0,\nonumber\\
\frac{d\lambda_{{\rm v}1,0}^{f}(\rho)}{d\rho} = 0,\  \  \  \ \rho \le 0,\nonumber\\
\frac{d\mu_{1,0}^{f}(\rho)}{d\rho} = 0,\  \  \  \ \rho \le 0.\nonumber\\
\end{eqnarray}

Due to the Boundary Functions Method, we require that
\begin{equation}\label{cond-for-bound-corr-2}
\lim_{\rho \rightarrow - \infty}z_{1,0}^{f}(\rho)= 0,\  \ \lim_{\rho \rightarrow - \infty}\lambda_{{\rm u}1,0}^{f}(\rho)= 0,\  \
\lim_{\rho \rightarrow - \infty}\lambda_{{\rm v}1,0}^{f}(\rho)= 0,\  \ \lim_{\rho \rightarrow - \infty}\mu_{1 ,0}^{f}(\rho)= 0.\end{equation}

Thus, the set of the equations (\ref{eq-for-bound-corr-2}) subject to the conditions (\ref{cond-for-bound-corr-2}) yields the unique solution
\begin{equation}\label{bound-corr-2}z_{1,0}^{f}(\rho) \equiv 0,\  \  \ \lambda_{{\rm u}1,0}^{f}(\rho) \equiv 0,\  \  \ \lambda_{{\rm v}1,0}^{f}(\rho) \equiv 0,\  \  \ \mu_{1,0}^{f}(\rho) \equiv 0,\  \  \ \rho
 \le 0.
\end{equation}

\subsubsection{Obtaining the outer solution terms $\bar{z}_{j,0}(t)$, $\bar{\lambda}_{{\rm u}j,0}(t)$, $\bar{\lambda}_{{\rm v}j,0}(t)$, $\bar{\mu}_{j0}(t)$, $j = 1,2$}\label{outer-solution-zero-order}

Using Remark \ref{explanation-zero-order-solution} and the equations (\ref{bound-corr-1}),(\ref{bound-corr-2}), we obtain the following differential-algebraic system of equations for the outer solution terms $\bar{z}_{j,0}(t)$, $\bar{\lambda}_{{\rm u}j,0}(t)$, $\bar{\lambda}_{{\rm v}j,0}(t)$, $\bar{\mu}_{j0}(t)$, $(j = 1,2)$ in the time interval $[0,t_{f}]$:
\begin{eqnarray}\label{outer-solution-system}
\frac{d\bar{z}_{1,0}(t)}{dt} =
A_{1}(t)\bar{z}_{1,0}(t) + A_{2}(t)\bar{z}_{2,0}(t) - S_{\rm u,1}(t)\bar{\lambda}_{\rm u1,0}(t),\nonumber\\
\bar{\lambda}_{\rm v2,0}(t) = 0,\nonumber\\
\frac{d\bar{\lambda}_{\rm u1,0}(t)}{dt} = - D_{\rm u,1}(t)\bar{z}_{1,0}(t) - D_{\rm u,2}(t)\bar{z}_{2,0}(t) - A_{1}^{T}(t)\bar{\lambda}_{\rm u1,0}(t) + D_{\rm v,1}(t)\bar{\mu}_{1,0}(t),\nonumber\\
0 = - D_{\rm u,2}^{T}(t)\bar{z}_{1,0}(t) - D_{\rm u,3}(t)\bar{z}_{2,0}(t) - A_{2}^{T}(t)\bar{\lambda}_{\rm u1,0}(t)  + D_{\rm v,2}(t)\bar{\mu}_{2,0}(t),\nonumber\\
\frac{d\bar{\lambda}_{\rm v1,0}(t)}{dt} = - D_{\rm v,1}(t)\bar{z}_{1,0}(t) - A_{1}^{T}(t)\bar{\lambda}_{\rm v1,0}(t),\nonumber\\
0 = - D_{\rm v,2}(t)\bar{z}_{2,0}(t) - A_{2}^{T}(t)\bar{\lambda}_{\rm v1,0}(t),\nonumber\\
\frac{d\bar{\mu}_{1,0}(t)}{dt} = A_{1}(t)\bar{\mu}_{1,0}(t) + A_{2}(t)\bar{\mu}_{2,0}(t),\nonumber\\
\bar{\lambda}_{\rm u2,0}(t) = 0,
\end{eqnarray}
\begin{eqnarray}\label{outer-solution-cond} \bar{z}_{1,0}(0) = z_{0 1},\  \  \  \ \bar{\lambda}_{\rm u1,0}(t_{f}) = 0,\  \  \  \ \bar{\lambda}_{\rm v1,0}(t_{f}) = 0,\  \  \  \ \bar{\mu}_{1,0}(0) = 0.\end{eqnarray}

\begin{remark}\label{no-cond-for-two-variables} The system (\ref{outer-solution-system}) consists of four differential equations with respect to $\bar{z}_{1,0}(t)$, $\bar{\lambda}_{\rm u1,0}(t)$, $\bar{\lambda}_{\rm v1,0}(t)$, $\bar{\mu}_{1,0}(t)$ and four algebraic equations with respect to $\bar{\lambda}_{\rm v2,0}(t)$, $\bar{z}_{2,0}(t)$, $\bar{\lambda}_{\rm u2,0}(t)$, $\bar{\mu}_{2,0}(t)$. Therefore, in (\ref{outer-solution-cond}), boundary conditions for $\bar{\lambda}_{\rm v2,0}(t)$, $\bar{z}_{2,0}(t)$, $\bar{\lambda}_{\rm u2,0}(t)$, $\bar{\mu}_{2,0}(t)$ are absent.\end{remark}

Resolving the sixth equation in the system (\ref{outer-solution-system}) with respect to $\bar{z}_{2,0}(t)$ and taking into account the invertibility of the matrix $D_{\rm v,2}(t)$ for all $t \in [0,t_{f}]$, we obtain
\begin{equation}\label{bar-z_2}\bar{z}_{2,0}(t) = - D_{\rm v,2}^{-1}(t)A_{2}^{T}(t)\bar{\lambda}_{\rm v1,0}(t),\  \  \  \ t \in [0,t_{f}].
\end{equation}

Substituting (\ref{bar-z_2}) into the fourth equation of the system (\ref{outer-solution-system}) and solving the resulting equation with respect to $\bar{\mu}_{2,0}(t)$ yield
\begin{eqnarray}\label{bar-mu_20}\bar{\mu}_{2,0}(t) = D_{\rm v,2}^{-1}(t)\big[D_{\rm u,2}^{T}(t)\bar{z}_{1,0}(t) - D_{\rm u,3}(t)D_{\rm v,2}^{-1}(t)A_{2}^{T}(t)\bar{\lambda}_{\rm v1,0}(t)\nonumber\\ + A_{2}^{T}(t)\bar{\lambda}_{\rm u1,0}(t)\big],\  \  \  \ t \in [0,t_{f}].
\end{eqnarray}

Finally, substitution of (\ref{bar-z_2})-(\ref{bar-mu_20}) into the first, third and seventh equations of the  system (\ref{outer-solution-system}) and use of the fifth equation of this system yield after a routine algebra the following set of four differential equations with respect to $\bar{z}_{1,0}(t)$, $\bar{\lambda}_{\rm u1,0}(t)$, $\bar{\lambda}_{\rm v1,0}(t)$, $\bar{\mu}_{1,0}(t)$:
\begin{eqnarray}\label{set-four-diff-eq-outer-sol}
\frac{d\bar{z}_{1,0}(t)}{dt} = A_{1}(t)\bar{z}_{1,0}(t) - S_{\rm u,1}(t)\bar{\lambda}_{\rm u1,0}(t) - A_{2}(t)D_{\rm v,2}^{-1}(t)A_{2}^{T}(t)\bar{\lambda}_{\rm v1,0}(t),\nonumber\\
\frac{d\bar{\lambda}_{\rm u1,0}(t)}{dt} = - D_{\rm u,1}(t)\bar{z}_{1,0}(t) - A_{1}^{T}(t)\bar{\lambda}_{\rm u1,0}(t) + D_{\rm u,2}(t)D_{\rm v,2}^{-1}(t)A_{2}^{T}(t)\bar{\lambda}_{\rm v1,0}(t)\nonumber\\ + D_{\rm v,1}(t)\bar{\mu}_{\rm 1,0}(t),\nonumber\\
\frac{d\bar{\lambda}_{\rm v1,0}(t)}{dt} = - D_{\rm v,1}(t)\bar{z}_{1,0}(t) -
A_{1}^{T}(t)\bar{\lambda}_{\rm v1,0}(t),\nonumber\\
\frac{d\bar{\mu}_{1,0}(t)}{dt} = A_{2}(t)D_{\rm v,2}^{-1}(t)D_{\rm u,2}^{T}(t)\bar{z}_{1,0}(t) + A_{2}(t)D_{\rm v,2}^{-1}(t)A_{2}^{T}(t)\bar{\lambda}_{\rm u1,0}(t)\nonumber\\
- A_{2}(t)D_{\rm v,2}^{-1}(t)D_{u,3}(t)D_{\rm v,2}^{-1}(t)A_{2}^{T}(t)\bar{\lambda}_{\rm v1,0}(t) + A_{1}(t)\bar{\mu}_{1,0}(t).\nonumber\\
\end{eqnarray}

The system of the differential equations (\ref{set-four-diff-eq-outer-sol}) is subject to the boundary conditions  (\ref{outer-solution-cond}).

\begin{remark}\label{dimension-reduced-syst} Let us note the following. The dimension of each of the state variables $z(t)$, $\lambda_{\rm u}(t)$, $\lambda_{\rm v}(t)$, $\mu(t)$ of the differential system in (\ref{initial-bound-value-pr}) is $n$. The dimension of the state variables $\bar{z}_{1,0}(t)$, $\bar{\lambda}_{\rm u1,0}(t)$, $\bar{\lambda}_{\rm v1,0}(t)$, $\bar{\mu}_{1,0}$ of the differential system (\ref{set-four-diff-eq-outer-sol}) is $n - s$. Hence, the differential system (\ref{set-four-diff-eq-outer-sol}) is of a considerably lower dimension than the differential system in (\ref{initial-bound-value-pr}). Moreover, in contrast with the system in (\ref{initial-bound-value-pr}), the system (\ref{set-four-diff-eq-outer-sol}) is independent of $\varepsilon$.
\end{remark}

In what follows, we assume.\\
{\bf (A7)} The boundary-value problem (\ref{set-four-diff-eq-outer-sol}),(\ref{outer-solution-cond}) has the unique solution\\ ${\rm col}\big(\bar{z}_{1,0}(t),\bar{\lambda}_{\rm u1,0}(t),\bar{\lambda}_{\rm v1,0}(t),\bar{\mu}_{1,0}(t)\big)$,\  \ $t \in [0,t_{f}]$.

Thus, the solution of the boundary-value problem (\ref{set-four-diff-eq-outer-sol}),(\ref{outer-solution-cond}), along with $\bar{\lambda}_{\rm v 2,0}(t) = 0$, $\bar{\lambda}_{\rm u 2,0}(t) = 0$ (see the second and the eighth  equations in (\ref{outer-solution-system})) and $\bar{z}_{2,0}(t)$, $\bar{\mu}_{2,0}(t)$ (see the equations (\ref{bar-z_2})-(\ref{bar-mu_20})), constitutes the zero-order (with respect to $\varepsilon$) outer solution.

\subsubsection{Control-theoretic interpretation of the boundary-value problem (\ref{set-four-diff-eq-outer-sol}),(\ref{outer-solution-cond})}

Consider the following system of differential equations in the time interval $[0,t_{f}]$:
\begin{eqnarray}\label{diff-eqs-control-interpr}
\frac{d\bar{z}_{1}(t)}{dt} = A_{1}(t)\bar{z}_{1}(t) - A_{2}(t)D_{\rm v,2}^{-1}(t)A_{2}^{T}(t)\bar{\lambda}_{\rm v1}(t) + B_{\rm u,1}(t)\bar{u}(t),\nonumber\\
\frac{d\bar{\lambda}_{\rm v1}(t)}{dt} = - D_{\rm v,1}(t)\bar{z}_{1}(t) - A_{1}^{T}(t)\bar{\lambda}_{\rm v1}(t),\nonumber\\
\end{eqnarray}
where $\bar{z}_{1}(t) \in \mathbb{R}^{n - s}$ and $\bar{\lambda}_{\rm v1}(t) \in \mathbb{R}^{n - s}$ are state variables; $\bar{u}(t) \in \mathbb{R}^{r}$ is a control variable; the vector-valued function $\bar{u}(t)$ is assumed to be continuous in the interval $[0,t_{f}]$; $B_{\rm u,1}(t)$ is the upper block of the matrix $B_{\rm u}(t)$ of the dimension $(n - s)\times r$.

The system (\ref{diff-eqs-control-interpr}) is subject to the boundary conditions
\begin{equation}\label{bound-cond-control-interpr} \bar{z}_{1}(0) = z_{0 1},\  \  \  \ \bar{\lambda}_{\rm v1}(t_{f}) = 0.\end{equation}

The control variable $\bar{u}(t)$ in the boundary-value problem (\ref{diff-eqs-control-interpr})-(\ref{bound-cond-control-interpr}) is evaluated by the performance index
\begin{eqnarray}\label{perform-index-control-interpr}\bar{J}\big(\bar{u}(t)\big) \stackrel{\triangle}{=} \frac{1}{2}\int_{0}^{t_{f}}\big[\bar{z}_{1}^{T}(t)D_{\rm u,1}(t)\bar{z}_{1}(t) - 2\bar{z}_{1}^{T}(t)D_{\rm u,2}(t)D_{\rm v,2}^{-1}(t)A_{2}^{T}(t)\bar{\lambda}_{\rm v1}(t)\nonumber\\
+ \bar{\lambda}_{\rm v1}^{T}(t)A_{2}(t)D_{\rm v,2}^{-1}(t)D_{\rm u,3}(t)D_{\rm v,2}^{-1}(t)A_{2}^{T}(t)\bar{\lambda}_{\rm v1}(t)
+ \bar{u}^{T}(t)\bar{u}(t)\big]dt \rightarrow \min_{\bar{u}(t)}.
\end{eqnarray}

\begin{proposition}\label{solution-auxiliary-ocp} Let the assumptions (A1)-(A7) be valid. Then, the optimal control problem (\ref{diff-eqs-control-interpr})-(\ref{bound-cond-control-interpr}),(\ref{perform-index-control-interpr}) has the unique open-loop optimal control
\begin{eqnarray}\label{opt-contr-bar-w} \bar{u}^{*}(t) = - B_{\rm u,1}(t)\bar{\lambda}_{\rm u1,0}(t),\  \  \  \ t \in [0,t_{f}],
\end{eqnarray}
where $\bar{\lambda}_{\rm u1,0}(t)$ is the corresponding component of the solution to the boundary-value problem (\ref{set-four-diff-eq-outer-sol}),(\ref{outer-solution-cond}). The optimal value $\bar{J}_{\rm u}^{*}$ of the functional in the optimal control problem (\ref{diff-eqs-control-interpr})-(\ref{bound-cond-control-interpr}),(\ref{perform-index-control-interpr}) has the form
\begin{eqnarray}\label{bar-J_u*}\bar{J}_{\rm u}^{*} = \frac{1}{2}\int_{0}^{t_{f}}\big[\bar{z}_{1,0}^{T}(t)D_{\rm u,1}(t)\bar{z}_{1,0}(t) - 2\bar{z}_{1,0}^{T}(t)D_{\rm u,2}(t)D_{\rm v,2}^{-1}(t)A_{2}^{T}(t)\bar{\lambda}_{\rm v1,0}(t)\nonumber\\
+ \bar{\lambda}_{\rm v1,0}^{T}(t)A_{2}(t)D_{\rm v,2}^{-1}(t)D_{\rm u,3}(t)D_{\rm v,2}^{-1}(t)A_{2}^{T}(t)\bar{\lambda}_{\rm v1,0}(t)
+ \bar{\lambda}_{\rm u1,0}^{T}(t)S_{\rm u,1}(t)\bar{\lambda}_{\rm u1,0}(t)\big]dt,\nonumber
\end{eqnarray}
where $\bar{z}_{1,0}(t)$ and $\bar{\lambda}_{\rm v1,0}(t)$ are the corresponding components of the solution to the boundary-value problem (\ref{set-four-diff-eq-outer-sol}),(\ref{outer-solution-cond}).
\end{proposition}
\begin{proof}
First of all, let us observe the following. Due to the assumption (A2) and the equation (\ref{new-matr-D_u}), the matrix $D_{\rm u}(t)$ is positive semi-definite for all $t \in [0,t_{f}]$. Using the block form of $D_{\rm u}(t)$ (see the equation (\ref{matr-A-D_u-S_u-blocks})), the quadratic form of the state variables in the functional $\bar{J}\big(\bar{u}(t)\big)$ can be represented as:
\begin{eqnarray}\label{quadratic-form-states}\bar{z}_{1}^{T}(t)D_{\rm u,1}(t)\bar{z}_{1}(t) - 2\bar{z}_{1}^{T}(t)D_{\rm u,2}(t)D_{\rm v,2}^{-1}(t)A_{2}^{T}(t)\bar{\lambda}_{\rm v1}(t)\nonumber\\
+ \bar{\lambda}_{\rm v1}^{T}(t)A_{2}(t)D_{\rm v,2}^{-1}(t)D_{\rm u,3}(t)D_{\rm v,2}^{-1}(t)A_{2}^{T}(t)\bar{\lambda}_{\rm v1}(t)\nonumber\\
= \Big(\bar{z}_{1}^{T}(t)\ ,\ - \bar{\lambda}_{\rm v1}^{T}(t)A_{2}(t)D_{\rm v,2}^{-1}(t)\Big)D_{\rm u}(t)\left(\begin{array}{l}\  \  \ \bar{z}_{1}(t)\\ - D_{\rm v,2}^{-1}(t)A_{2}^{T}(t)\lambda_{\rm v1}(t)\end{array}\right),\  \  \  \ t \in [0,t_{f}].\nonumber
\end{eqnarray}

Hence, this quadratic form is positive semi-definite for all $t \in [0,t_{f}]$.
Moreover, the quadratic form of the control variable in the functional $\bar{J}\big(\bar{u}(t)\big)$ is positive definite for all $t \in [0,t_{f}]$. Using this observation and
based on the results of \cite{Ioffe-Tikhomirov} (Section 9.2, the existence of a solution to an optimal control problem) and \cite{Bryson-Ho} (Section 2.5, the control optimality conditions), we directly obtain the statements of the proposition.
\end{proof}

\begin{remark}\label{new-form-bar-J_u^*} Using Proposition \ref{solution-auxiliary-ocp} and the equations (\ref{matr-A-D_u-S_u-blocks}),(\ref{bar-z_2}), we can rewrite the expression for the optimal value $\bar{J}_{\rm u}^{*}$ of the functional in the optimal control problem (\ref{diff-eqs-control-interpr})-(\ref{bound-cond-control-interpr}),(\ref{perform-index-control-interpr}) as:
\begin{eqnarray}\label{new-form-optim-value-funct}\bar{J}_{\rm u}^{*} = \frac{1}{2}\int_{0}^{t_{f}}\big[\bar{z}_{0}^{T}(t)D_{\rm u}(t)\bar{z}_{0}(t) + \bar{\lambda}_{\rm u1,0}^{T}(t)S_{\rm u,1}(t)\bar{\lambda}_{\rm u1,0}(t)\big]dt,\nonumber
\end{eqnarray}
\end{remark}
where $\bar{z}_{0}(t) = {\rm col}\big(\bar{z}_{1,0}(t) , \bar{z}_{2,0}(t)\big)$, $t \in [0,t_{f}]$.

\subsubsection{Obtaining the boundary corrections $z_{2,0}^{0}(\xi)$, $\lambda_{\rm u 2,0}^{0}(\xi)$, $\lambda_{\rm v 2,0}^{0}(\xi)$, $\mu_{2,0}^{0}(\xi)$ and $z_{2,0}^{f}(\rho)$, $\lambda_{\rm u 2,0}^{f}(\rho)$, $\lambda_{\rm v 2,0}^{f}(\rho)$, $\mu_{2,0}^{f}(\rho)$}\label{obtaining-bound-corr-3-4}

Due to Remark \ref{explanation-zero-order-solution} and the equations (\ref{bound-corr-1}),(\ref{bound-corr-2}), we have the following differential equations for these boundary corrections:
\begin{eqnarray}\label{eq-bound-corr-3} \frac{dz_{2,0}^{0}(\xi)}{d\xi} = - \lambda_{\rm v 2,0}^{0}(\xi),\  \  \  \ \xi \ge 0,\nonumber\\
\frac{d\lambda_{\rm u 2,0}^{0}(\xi)}{d\xi} = - D_{\rm u,3}(0)z_{2,0}^{0}(\xi) + D_{\rm v,2}(0)\mu_{2,0}^{0}(\xi),\  \  \  \ \xi \ge 0,\nonumber\\
\frac{d\lambda_{\rm v2,0}(\xi)}{d\xi} = -D_{\rm v,2}(0)z_{2,0}^{0}(\xi),\  \  \  \ \xi \ge 0,\nonumber\\
\frac{d\mu_{2,0}^{0}(\xi)}{d\xi} = \lambda_{\rm u2,0}^{0}(\xi),\  \  \  \ \xi \ge 0.\nonumber\\
\end{eqnarray}
\begin{eqnarray}\label{eq-bound-corr-4} \frac{dz_{2,0}^{f}(\rho)}{d\rho} = - \lambda_{\rm v 2,0}^{f}(\rho),\  \  \  \ \rho \le 0,\nonumber\\
\frac{d\lambda_{\rm u 2,0}^{f}(\rho)}{d\rho} = - D_{\rm u,3}(t_{f})z_{2,0}^{f}(\rho) + D_{\rm v,2}(t_{f})\mu_{2,0}^{f}(\rho),\  \  \  \ \rho \le 0,\nonumber\\
\frac{d\lambda_{\rm v2,0}^{f}(\rho)}{d\rho} = - D_{\rm v,2}(t_{f})z_{2,0}^{f}(\rho),\  \  \  \ \rho \le 0,\nonumber\\
\frac{d\mu_{2,0}^{f}(\rho)}{d\rho} = \lambda_{\rm u2,0}^{f}(\rho),\  \  \  \ \rho \le 0.\nonumber\\
\end{eqnarray}

By virtue of the Boundary Functions Method \cite{Vasil'eva-Butuzov-Kalachev,Esipova-Diff-Eq}, we require (similarly to (\ref{cond-for-bound-corr-l}),(\ref{cond-for-bound-corr-2})) that the solutions of the systems (\ref{eq-bound-corr-3}),(\ref{eq-bound-corr-4}) satisfy the conditions
\begin{equation}\label
{additional-cond-for-corr-3} \lim_{\xi \rightarrow + \infty}z_{2,0}^{0}(\xi) = 0,\  \  \lim_{\xi \rightarrow + \infty}\lambda_{\rm u 2,0}^{0}(\xi) = 0,\  \ \lim_{\xi \rightarrow + \infty}\lambda_{\rm v 2,0}^{0}(\xi) = 0,\ \ \lim_{\xi \rightarrow + \infty}\mu_{2,0}^{0}(\xi) = 0,\end{equation}
\begin{equation}\label
{additional-cond-for-corr-4} \lim_{\rho \rightarrow - \infty}z_{2,0}^{f}(\rho) = 0,\  \  \lim_{\rho \rightarrow - \infty}\lambda_{\rm u 2,0}^{f}(\rho) = 0,\  \ \lim_{\rho \rightarrow - \infty}\lambda_{\rm v 2,0}^{f}(\rho) = 0,\ \ \lim_{\rho \rightarrow - \infty}\mu_{2,0}^{f}(\rho) = 0.\end{equation}

Substituting the expressions for $z_{2}^{1}(t,\varepsilon)$ and $\mu_{2}^{1}(t,\varepsilon)$ (see the equation (\ref{first-order-asympt-solution})) into the initial conditions for $z_{2}(t,\varepsilon)$ and $\mu_{2}(t,\varepsilon)$, respectively, (see the equation (\ref{equiv-cond})), equating the coefficients for $\varepsilon$ in the power of $0$ on both sides of the resulting equations and using the first and the fourth limit equalities in (\ref{additional-cond-for-corr-4}), we obtain additional conditions for $z_{2,0}^{0}(\xi)$ and $\mu_{2,0}^{0}(\xi)$, namely,
\begin{equation}\label{init-cond-for-z_2^0-mu_2^0} z_{2,0}^{0}(0) = z_{0 2} - \bar{z}_{2,0}(0),\  \  \ \ \mu_{2,0}^{0}(0) = - \bar{\mu}_{2,0}(0).\end{equation}

Similarly, substituting the expressions for $\lambda_{\rm u2}^{1}(t,\varepsilon)$ and $\lambda_{\rm v2}^{1}(t,\varepsilon)$ (see the equation (\ref{first-order-asympt-solution})) into the terminal conditions for $\lambda_{\rm u2}(t,\varepsilon)$ and $\lambda_{\rm v2}(t,\varepsilon)$, respectively, (see the equation (\ref{equiv-cond})), equating the coefficients for $\varepsilon$ in the power of $0$ on both sides of the resulting equations, using the second and the third limit equalities in (\ref{additional-cond-for-corr-3}) and that $\bar{\lambda}_{\rm u 2,0}(t) \equiv 0$, $\bar{\lambda}_{\rm v 2,0}(t) \equiv 0$, we obtain additional conditions for $\lambda_{\rm u 2,0}^{f}(\rho)$, $\lambda_{\rm v 2,0}^{f}(\rho)$, namely,
\begin{equation}\label{termin-cond-for-lambda_u2^f} \lambda_{\rm u 2,0}^{f}(0) = 0,\  \  \  \ \lambda_{\rm v 2,0}^{f}(0) = 0.\end{equation}

Proceed to solution of the system (\ref{eq-bound-corr-3}) subject to the conditions (\ref{additional-cond-for-corr-3}) and (\ref{init-cond-for-z_2^0-mu_2^0}).

By $D_{\rm v,2}^{1/2}(t)$, $t \in [0,t_{f}]$, we denote the unique symmetric positive definite square root of the symmetric positive definite matrix $D_{\rm v,2}(t)$. By $D_{\rm v,2}^{- 1/2}(t)$, $t \in [0,t_{f}]$, we denote the inverse matrix of this square root. Using these matrices, we consider the following block-form matrix:
\begin{equation}\label{matr-Theta_0} \Theta(t) = \left(\begin{array}{l}\  \
 \ I_{s}\  \  \  \  \  \  \  \  \  \  \  0.5D_{\rm v,2}^{-1/2}(t)\\ - D_{\rm v,2}^{1/2}(t)\  \  \  \  \ 0.5I_{s} \end{array}\right),\  \  \  \ t \in [0,t_{f}].\end{equation}
This matrix is invertible for all $t \in [0,t_{f}]$ and its inverse matrix is
\begin{equation}\label{matr-inverse-Theta_0} \Theta^{-1}(t) = \left(\begin{array}
{l}0.5I_{s}\  \  \  \  \  \  \ -0.5D_{\rm v,2}^{-1/2}(t)\\  D_{\rm v,2}^{1/2}(t)\  \  \  \  \  \  \  \ I_{s} \end{array}\right),\  \  \  \ t \in [0,t_{f}].\end{equation}

Using the matrices $\Theta(t)$ and $\Theta^{-1}(t)$, we make the following transformation of the state variables in the system (\ref{eq-bound-corr-3}) and the conditions (\ref{additional-cond-for-corr-3}),(\ref{init-cond-for-z_2^0-mu_2^0}):
\begin{equation}\label{transform-corr-3} \left(\begin{array}{l} z_{2,0}^{0}(\xi)\\ \\ \lambda_{\rm v 2,0}^{0}(\xi)\end{array}\right) =\Theta(0)\left(\begin{array}{l}x^{0}_{0}(\xi)\\ \\ y^{0}_{0}(\xi)\end{array}\right),\ \  \  \ \left(\begin{array}{l} \mu_{2,0}^{0}(\xi)\\ \\ \lambda_{\rm u 2,0}^{0}(\xi)\end{array}\right) =\Theta(0)\left(\begin{array}{l}\eta^{0}_{0}(\xi)\\ \\ \zeta^{0}_{0}(\xi)\end{array}\right),\end{equation}
where $x^{0}_{0}(\xi) \in \mathbb{R}^{s}$, $y^{0}_{0}(\xi) \in \mathbb{R}^{s}$, $\eta_{0}^{0}(\xi) \in \mathbb{R}^{s}$ and $\zeta_{0}^{0}(\xi) \in \mathbb{R}^{s}$ are new state variables.

The transformation (\ref{transform-corr-3}) converts the system (\ref{eq-bound-corr-3})  and the conditions (\ref{additional-cond-for-corr-3}),(\ref{init-cond-for-z_2^0-mu_2^0}) to the system
\begin{eqnarray}\label{new-eq-bound-corr-3} \frac{dx^{0}_{0}(\xi)}{d\xi} = D_{\rm v,2}^{1/2}(0)x^{0}_{0}(\xi),\  \  \  \ \xi \ge 0,\nonumber\\
\frac{dy^{0}_{0}(\xi)}{d\xi} = - D_{\rm v,2}^{1/2}(0)y^{0}_{0}(\xi),\  \  \  \ \xi \ge 0,\nonumber\\
\frac{d\eta_{0}^{0}(\xi)}{d\xi} = 0.5D_{\rm v,2}^{-1/2}(0)D_{\rm u,3}(0)x_{0}^{0}(\xi) + 0.25D_{\rm v,2}^{-1/2}(0)D_{\rm u,3}(0)D_{\rm v,2}^{-1/2}(0)y_{0}^{0}(\xi)\nonumber\\
- D_{\rm v,2}^{1/2}(0)\eta_{0}^{0}(\xi),\  \  \ \xi \ge 0,\nonumber\\
\frac{d\zeta_{0}^{0}(\xi)}{d\xi} = - D_{\rm u,3}(0)x_{0}^{0}(\xi) - 0.5D_{\rm u,3}(0)D_{\rm v,2}^{-1/2}(0)y_{0}^{0}(\xi) + D_{\rm v,2}^{1/2}(0)\zeta_{0}^{0}(\xi),\  \  \ \xi \ge 0,\nonumber\\
\end{eqnarray}
and the conditions
\begin{equation}\label{transf-add-cond-for-corr-3}\lim_{\xi \rightarrow + \infty}x^{0}_{0}(\xi) = 0,\  \  \  \lim_{\xi \rightarrow + \infty}y^{0}_{0}(\xi) = 0,\  \  \ \lim_{\xi \rightarrow + \infty}\eta_{0}^{0}(\xi) = 0,\  \  \ \lim_{\xi \rightarrow + \infty}\zeta_{0}^{0}(\xi) = 0,\end{equation}
\begin{eqnarray}\label{init-cond-corr-3-transf} x^{0}_{0}(0) = 0.5\big(z_{0 2} - \bar{z}_{2,0}(0)\big) - 0.5D_{\rm v,2}^{-1/2}(0)\lambda_{\rm v 2,0}^{0}(0),\nonumber\\
y^{0}_{0}(0) = D_{\rm v,2}^{1/2}(0)\big(z_{0 2} - \bar{z}_{2,0}(0)\big) + \lambda_{\rm v 2,0}^{0}(0),\nonumber\\
\eta_{0}^{0}(0) = - 0.5\bar{\mu}_{2,0}(0) - 0.5D_{\rm v,2}^{-1/2}(0)\lambda_{\rm u2,0}^{0}(0),\nonumber\\
\zeta_{0}^{0}(0) = - D_{\rm v,2}^{1/2}(0)\bar{\mu}_{2,0}(0) + \lambda_{\rm u2,0}^{0}(0).\nonumber\\
\end{eqnarray}

The first equation in (\ref{new-eq-bound-corr-3}) subject to the first condition in (\ref{init-cond-corr-3-transf}) yields the solution
\begin{equation}\label{x^0} x^{0}_{0}(\xi) = \exp\big(D_{\rm v,2}^{1/2}(0)\xi\big)x^{0}_{0}(0),\  \  \  \ \xi \ge 0.\end{equation}

Since the matrix $D_{\rm v,2}^{1/2}(0)$ is positive definite, then the vector-valued function $x^{0}_{0}(\xi)$, obtained in (\ref{x^0}), satisfies the first limit condition in (\ref{transf-add-cond-for-corr-3}) if and only if $x^{0}_{0}(0) = 0$. Hence, the solution of the first equation in (\ref{new-eq-bound-corr-3}) subject to the first condition in (\ref{transf-add-cond-for-corr-3}) is identically zero, i.e.,
\begin{equation}\label{x^0=0} x^{0}_{0}(\xi) \equiv 0,\  \  \  \ \xi \ge 0.\end{equation}
The latter, along with the first and the second conditions in (\ref{init-cond-corr-3-transf}), yields
\begin{eqnarray}\label{lambda_v2^0(0)} \lambda_{\rm v 2,0}^{0}(0) = D_{\rm v,2}^{1/2}(0)\big(z_{0 2} - \bar{z}_{2,0}(0)\big),\nonumber\\
y^{0}_{0}(0) = 2D_{\rm v,2}^{1/2}(0)\big(z_{0 2} - \bar{z}_{2,0}(0)\big).\nonumber\\
\end{eqnarray}

Solving the second equation in (\ref{new-eq-bound-corr-3}) subject to the second initial condition in (\ref{lambda_v2^0(0)}), we have
\begin{equation}\label{y^0} y^{0}_{0}(\xi) = 2\exp\big(- D_{\rm v,2}^{1/2}(0)\xi\big)D_{\rm v,2}^{1/2}(0)\big(z_{0 2} - \bar{z}_{2,0}(0)\big),\  \  \  \ \xi \ge 0.
\end{equation}
Since the matrix $D_{\rm v,2}^{1/2}(0)$ is positive definite, then the vector-valued function $y^{0}_{0}(\xi)$, obtained in (\ref{y^0}), satisfies the second limit condition in (\ref{transf-add-cond-for-corr-3}). Thus, we have obtained the solution of the first and the second differential equations in (\ref{new-eq-bound-corr-3}) satisfying the first and the second conditions in (\ref{transf-add-cond-for-corr-3}) and in (\ref{init-cond-corr-3-transf}).

Now, let us treat the third and the fourth differential equations in (\ref{new-eq-bound-corr-3}) subject to the third and the fourth conditions in (\ref{transf-add-cond-for-corr-3}) and in (\ref{init-cond-corr-3-transf}). First, substituting (\ref{x^0=0}) and (\ref{y^0}) into the aforementioned differential equations, we obtain after routine matrix calculations
\begin{eqnarray}\label{new-eq-eta_0-zeta_0}\frac{d\eta_{0}^{0}(\xi)}{d\xi} = - D_{\rm v,2}^{1/2}(0)\eta_{0}^{0}(\xi) + 0.5D_{\rm v,2}^{-1/2}(0)D_{\rm u,3}(0)\exp\big(- D_{\rm v,2}^{1/2}(0)\xi\big)\big(z_{02} - \bar{z}_{2,0}(0)\big),
\nonumber\\\xi \ge 0,\nonumber\\
\frac{d\zeta_{0}^{0}(\xi)}{d\xi} = D_{\rm v,2}^{1/2}(0)\zeta_{0}^{0}(\xi) - D_{\rm u,3}(0)\exp\big(- D_{\rm v,2}^{1/2}(0)\xi\big)\big(z_{02} - \bar{z}_{2,0}(0)\big),\nonumber\\\xi \ge 0.\nonumber\\
\end{eqnarray}

The second differential equation of (\ref{new-eq-eta_0-zeta_0}) subject to the fourth initial condition in (\ref{init-cond-corr-3-transf}) yields the solution
\begin{eqnarray}\label{zeta_0^0}\zeta_{0}^{0}(\xi) = \exp\big(D_{\rm v,2}^{1/2}(0)\xi\big)\bigg[\zeta_{0}^{0}(0)\nonumber\\
- \int_{0}^{\xi}\exp\big(- D_{\rm v,2}^{1/2}(0)\sigma\big)D_{\rm u,3}(0)\exp\big(- D_{\rm v,2}^{1/2}(0)\sigma\big)d\sigma \big(z_{02} - \bar{z}_{2,0}(0)\big)\bigg],\  \  \ \xi \ge 0.\end{eqnarray}
This equation can be rewritten as:
\begin{eqnarray}\label{zeta_0^0-I}\exp\big(- D_{\rm v,2}^{1/2}(0)\xi\big)\zeta_{0}^{0}(\xi) = \zeta_{0}^{0}(0)\nonumber\\
- \int_{0}^{\xi}\exp\big(- D_{\rm v,2}^{1/2}(0)\sigma\big)D_{\rm u,3}(0)\exp\big(- D_{\rm v,2}^{1/2}(0)\sigma\big)d\sigma \big(z_{02} - \bar{z}_{2,0}(0)\big),\  \  \ \xi \ge 0.\nonumber
\end{eqnarray}

Calculating the limit of this equation for $\xi \rightarrow + \infty$ and taking into account the positive definiteness of $D_{\rm v,2}^{1/2}(0)$ and the fourth limit condition in (\ref{transf-add-cond-for-corr-3}), we obtain the new expression for $\zeta_{0}^{0}(0)$
\begin{eqnarray}\label{zeta_0^0(0)} \zeta_{0}^{0}(0) = \int_{0}^{+\infty}\exp\big(- D_{\rm v,2}^{1/2}(0)\sigma\big)D_{\rm u,3}(0)\exp\big(- D_{\rm v,2}^{1/2}(0)\sigma\big)d\sigma \big(z_{02} - \bar{z}_{2,0}(0)\big).\end{eqnarray}

Note that, due to the positive definiteness of the matrix $D_{\rm v,2}^{1/2}(0)$, the integral in the right-hand side of (\ref{zeta_0^0(0)}) converges.

Using (\ref{zeta_0^0(0)}), we can rewrite the expression for $\zeta_{0}^{0}(\xi)$ in (\ref{zeta_0^0}) as:
\begin{eqnarray}\label{zeta_0^0-new-form}\zeta_{0}^{0}(\xi) =  \int_{\xi}^{+ \infty}\exp\big(D_{\rm v,2}^{1/2}(0)(\xi - \sigma)\big)D_{\rm u,3}(0)\exp\big(- D_{\rm v,2}^{1/2}(0)\sigma\big)d\sigma \big(z_{02} - \bar{z}_{2,0}(0)\big),\nonumber\\\xi \ge 0.\end{eqnarray}

Due to the positive definiteness of the matrix $D_{\rm v,2}^{1/2}(0)$, the integral in the right-hand side of (\ref{zeta_0^0-new-form}) converges for any $\xi \ge 0$ and $\zeta_{0}^{0}(\xi)$, given by this equation, satisfies the fourth limit condition in  (\ref{transf-add-cond-for-corr-3}).

Furthermore, equating the expressions for $\zeta_{0}^{0}(0)$, given in (\ref{init-cond-corr-3-transf}) and in (\ref{zeta_0^0(0)}), we directly have
\begin{eqnarray}\label{lambda_u20^0(0)} \lambda_{\rm u2,0}^{0}(0) = D_{v,2}^{1/2}(0)\bar{\mu}_{2,0}(0)\nonumber\\
+ \int_{0}^{+\infty}\exp\big(- D_{\rm v,2}^{1/2}(0)\sigma\big)D_{\rm u,3}(0)\exp\big(- D_{\rm v,2}^{1/2}(0)\sigma\big)d\sigma \big(z_{02} - \bar{z}_{2,0}(0)\big).\end{eqnarray}

Substitution of (\ref{lambda_u20^0(0)}) into the third condition in (\ref{init-cond-corr-3-transf}) yields
\begin{eqnarray}\label{eta_0^0(0)}\eta_{0}^{0}(0) = - \bar{\mu}_{2,0}(0)\nonumber\\
- 0.5D_{\rm v,2}^{-1/2}(0)\int_{0}^{+\infty}\exp\big(- D_{\rm v,2}^{1/2}(0)\sigma\big)D_{\rm u,3}(0)\exp\big(- D_{\rm v,2}^{1/2}(0)\sigma\big)d\sigma\times \nonumber\\
\big(z_{02} - \bar{z}_{2,0}(0)\big).\end{eqnarray}

Then, solving the first differential equation in (\ref{new-eq-eta_0-zeta_0}) subject to the initial condition (\ref{eta_0^0(0)}), we directly have
\begin{eqnarray}\label{eta_0^0}\eta_{0}^{0}(\xi) = \exp\big(- D_{\rm v,2}^{1/2}(0)\xi\big)\eta_{0}^{0}(0)\nonumber\\
+ \frac{1}{2}\int_{0}^{\xi}\exp\big(- D_{\rm v,2}^{1/2}(0)(\xi - \sigma)\big)D_{\rm v,2}^{-1/2}(0)D_{u,3}(0)\exp\big(- D_{\rm v,2}^{1/2}(0)\sigma\big)d\sigma\times \nonumber\\
\big(z_{0 2} - \bar{z}_{2,0}(0)\big),\  \  \  \ \xi \ge 0.\nonumber\\
\end{eqnarray}

Thus, the vector-valued functions $x^{0}_{0}(\xi)$, $y^{0}_{0}(\xi)$, $\zeta_{0}^{0}(\xi)$  and $\eta_{0}^{0}(\xi)$ (see the equations (\ref{x^0=0}),(\ref{y^0}),(\ref{zeta_0^0-new-form}) and (\ref{eta_0^0})), respectively) constitute the solution of the system (\ref{new-eq-bound-corr-3}) satisfying the conditions (\ref{transf-add-cond-for-corr-3}),(\ref{init-cond-corr-3-transf}).
Substitution of this solution into (\ref{transform-corr-3}) and use of (\ref{matr-Theta_0}) yield, after a routine rearrangement, the solution of the system (\ref{eq-bound-corr-3}) satisfying the conditions (\ref{additional-cond-for-corr-3}),(\ref{init-cond-for-z_2^0-mu_2^0})
\begin{eqnarray}\label{boundary-correction-3}
z_{2,0}^{0}(\xi) = \exp\big(- D_{\rm v,2}^{1/2}(0)\xi\big)\big(z_{0 2} - \bar{z}_{2,0}(0)\big),\  \  \  \ \xi \ge 0,\nonumber\\
\lambda_{\rm v 2,0}^{0}(\xi) = D_{\rm v, 2}^{1/2}(0)\exp\big(- D_{\rm v,2}^{1/2}(0)\xi\big)\big(z_{0 2} - \bar{z}_{2,0}(0)\big),\  \  \  \ \xi \ge 0,\nonumber\\
\lambda_{\rm u2,0}^{0}(\xi) = - D_{\rm v,2}^{1/2}(0)\eta_{0}^{0}(\xi) + 0.5\zeta_{0}^{0}(\xi),\  \  \  \ \xi \ge 0,\nonumber\\
\mu_{2,0}^{0}(\xi) = \eta_{0}^{0}(\xi) + 0.5D_{\rm v,2}^{-1/2}(0)\zeta_{0}^{0}(\xi),\  \  \  \ \xi \ge 0.\nonumber\\
\end{eqnarray}
This solution satisfies the inequalities
\begin{eqnarray}\label{ineq-for-bound-correc-3} \big\|z_{2,0}^{0}(\xi)\big\| \le a_{\rm z}^{0}\exp(- \beta^{0}\xi),\  \  \  \  \big\|\lambda_{\rm v 2,0}^{0}(\xi)\big\| \le a_{\rm v}^{0}\exp(- \beta^{0}\xi),\  \  \  \ \xi \ge 0,\nonumber\\
\big\|\lambda_{\rm u 2,0}^{0}(\xi)\big\| \le a_{\rm u}^{0}\exp(- 0.5\beta^{0}\xi),\  \  \  \ \big\|\mu_{2,0}^{0}(\xi)\big\| \le a_{\mu}^{0}\exp(- 0.5\beta^{0}\xi),\  \  \  \ \xi \ge 0,\nonumber\\
\end{eqnarray}
where $a_{\rm z}^{0} > 0$, $a_{\rm v}^{0} > 0$, $a_{\rm v}^{0} > 0$, $a_{\mu}^{0} > 0$ are some constants; $\beta^{0} > 0$ is the minimal eigenvalue of the matrix $D_{\rm v,2}^{1/2}(0)$.

For obtaining the solution of the system (\ref{eq-bound-corr-4}) satisfying the conditions (\ref{additional-cond-for-corr-4}),(\ref{termin-cond-for-lambda_u2^f}), we use the procedure similar to that for the derivation of the solution to the problem (\ref{eq-bound-corr-3}),(\ref{additional-cond-for-corr-3}),(\ref{init-cond-for-z_2^0-mu_2^0}). In this procedure, we replace $\Theta(0)$ with $\Theta(t_{f})$. Thus, taking into account that, in contrast with the system (\ref{eq-bound-corr-3}), the independent variable in the system (\ref{eq-bound-corr-4}) is non-positive ($\rho \le 0$), we have
\begin{equation}\label{bound-corrections-4} z_{2,0}^{f}(\rho) \equiv 0,\  \  \  \ \lambda_{\rm v 2,0}^{f}(\rho) \equiv 0,\  \  \  \  \lambda_{\rm u 2,0}^{f}(\rho) \equiv 0,\  \  \  \ \mu_{2,0}^{f}(\rho) \equiv 0,\  \  \  \  \eta \le 0.\end{equation}

\subsubsection{Obtaining the boundary corrections $z_{1,1}^{0}(\xi)$, $\lambda_{\rm u1,1}^{0}(\xi)$, $\lambda_{\rm v1,1}^{0}(\xi)$ and $\mu_{1,1}^{0}(\xi)$}

Using Remark \ref{explanation-zero-order-solution} and taking into account equation (\ref{bound-corr-1}), we obtain that  these boundary corrections satisfy the following equations:
\begin{eqnarray}\label{eq-for-bound-corr-l-1}
\frac{dz_{1,1}^{0}(\xi)}{d\xi} = A_{2}(0)z_{2,0}^{0}(\xi),\
 \  \  \ \xi \ge 0,\nonumber\\
\frac{d\lambda_{\rm u1,1}^{0}(\xi)}{d\xi} = - D_{\rm u,2}(0)z_{2,0}^{0}(\xi),\  \  \  \ \xi \ge 0,\nonumber\\
\frac{d\lambda_{\rm v1,1}^{0}(\xi)}{d\xi} = 0,\
 \  \  \ \xi \ge 0,\nonumber\\
\frac{d\mu_{1,1}^{0}(\xi)}{d\xi} = A_{2}(0)\mu_{2,0}^{0}(\xi),\  \  \  \ \xi \ge 0.\nonumber\\
\end{eqnarray}

Similarly to the results of Subsection \ref{boundary-correction-1-0} (see the equation (\ref{cond-for-bound-corr-l})),
we look for the solution of the system (\ref{eq-for-bound-corr-l-1}) subject to the conditions
\begin{equation}\label{cond-for-bound-corr-l-1}
\lim_{\xi \rightarrow + \infty}z_{1,1}^{0}(\xi)= 0,\  \ \lim_{\xi \rightarrow + \infty}\lambda_{\rm u1,1}^{0}(\xi)= 0,\  \
\lim_{\xi \rightarrow + \infty}\lambda_{\rm v1,1}^{0}(\xi)= 0,\  \ \lim_{\xi \rightarrow + \infty}\mu_{\rm 1,1}^{0}(\xi)= 0.\end{equation}

Now, using the expressions for $z_{2,0}^{0}(\xi)$ and $\mu_{2,0}^{0}(\xi)$ (see the equation (\ref{boundary-correction-3})), we directly obtain the unique solution of the system (\ref{eq-for-bound-corr-l-1}) subject to the conditions (\ref{cond-for-bound-corr-l-1})
\begin{eqnarray}\label{bound-corr-1-1}z_{1,1}^{0}(\xi) = - A_{2}(0)D_{\rm v,2}^{-1/2}(0)\exp\big(- D_{\rm v,2}^{1/2}(0)\xi\big)\big(z_{0 2} - \bar{z}_{2,0}(0)\big),\  \  \  \ \xi \ge 0,\nonumber\\
\lambda_{\rm u1,1}^{0}(\xi) = D_{\rm u,2}(0)D_{\rm v,2}^{-1/2}(0)\exp\big(- D_{\rm v,2}^{1/2}(0)\xi\big)\big(z_{0 2} - \bar{z}_{2,0}(0)\big),\  \  \  \ \xi \ge 0,\nonumber\\
\lambda_{\rm v1,1}^{0}(\xi) \equiv 0,\  \  \ \xi \ge 0,\nonumber\\
\mu_{1,1}^{0}(\xi) = - A_{2}(0)\int_{\xi}^{+\infty}\mu_{2,0}^{0}(\sigma)d\sigma,\  \  \ \xi \ge 0.
\nonumber\\ \end{eqnarray}

Due to the first and the fourth inequalities in (\ref{ineq-for-bound-correc-3}), we have the following inequalities for the vector-valued functions $z_{1,1}^{0}(\xi)$, $\lambda_{\rm u1,1}^{0}(\xi)$ and $\mu_{1,1}^{0}(\xi)$:
\begin{eqnarray}\label{ineq-bound-terms-1-1}
\big\|z_{1,1}^{0}(\xi)\big\| \le a_{z}^{0}\big\| A_{2}(0)D_{\rm v,2}^{-1/2}(0)\big\|\exp(- \beta^{0}\xi),\  \  \  \ \xi \ge 0,\nonumber\\
\big\|\lambda_{\rm u1,1}^{0}(\xi)\big\| \le a_{z}^{0}\big\|D_{\rm u,2}(0)D_{\rm v,2}^{-1/2}(0)\big\|\exp(- \beta^{0}\xi),\  \  \  \ \xi \ge 0,\nonumber\\
\big\|\mu_{1,1}^{0}(\xi)\big\| \le \frac{2a_{\mu}^{0}}{\beta^{0}}\big\|A_{2}(0)\big\|\exp(-0.5\beta^{0}\xi),\  \  \  \ \xi \ge 0.\nonumber\\
\end{eqnarray}

\subsubsection{Obtaining the boundary corrections $z_{1,1}^{f}(\rho)$, $\lambda_{\rm u1,1}^{f}(\rho)$, $\lambda_{\rm v1,1}^{f}(\rho)$ and $\mu_{1,1}^{f}(\rho)$}

Using the equations (\ref{bound-corr-2}) and (\ref{bound-corrections-4}), we derive the following differential equations for these boundary corrections:
\begin{eqnarray}\label{eq-for-bound-corr-2-1}
\frac{dz_{1,1}^{f}(\rho)}{d\rho} = 0,\
 \  \  \ \rho \le 0,\nonumber\\
\frac{d\lambda_{\rm u1,1}^{f}(\rho)}{d\rho} = 0,\  \  \  \ \rho \le 0,\nonumber\\
\frac{d\lambda_{\rm v1,1}^{f}(\rho)}{d\rho} = 0,\  \  \  \ \rho \le 0,\nonumber\\
\frac{d\mu_{\rm 1,1}^{f}(\rho)}{d\rho} = 0,\  \  \  \ \rho \le 0.\nonumber\\
\end{eqnarray}

Similarly to the results of Subsection \ref{boundary-corrections-f-0} (see the equation (\ref{cond-for-bound-corr-2})), we require that
\begin{equation}\label{cond-for-bound-corr-2-1}
\lim_{\rho \rightarrow - \infty}z_{1,1}^{f}(\rho)= 0,\  \ \lim_{\rho \rightarrow - \infty}\lambda_{\rm u1,1}^{f}(\rho)= 0,\  \
\lim_{\rho \rightarrow - \infty}\lambda_{\rm v1,1}^{f}(\rho)= 0,\  \ \lim_{\rho \rightarrow - \infty}\mu_{ 1,1}^{f}(\rho)= 0.\end{equation}

The system (\ref{eq-for-bound-corr-2-1}) has the unique solution satisfying the conditions (\ref{cond-for-bound-corr-2-1})
\begin{equation}\label{bound-corr-2-1}z_{1,1}^{f}(\rho) \equiv 0,\  \  \ \lambda_{\rm u1,1}^{f}(\rho) \equiv 0,\  \  \ \lambda_{\rm v1,1}^{f}(\rho) \equiv 0,\  \  \ \mu_{1,1}^{f}(\rho) \equiv 0,\  \  \ \rho \le 0.
\end{equation}

\subsubsection{Obtaining the outer solution terms $\bar{z}_{j,1}(t)$, $\bar{\lambda}_{\rm uj,1}(t)$, $\bar{\lambda}_{\rm vj,1}(t)$, $\bar{\mu}_{\rm j,1}(t)$, $j = 1,2$}\label{outer-solution-first-order}

Using the second and the eighth equations in (\ref{outer-solution-system}) and the equations (\ref{bound-corr-1-1}),(\ref{bound-corr-2-1}), we derive (similarly to the equations (\ref{outer-solution-system}),(\ref{outer-solution-cond})), the following differential-algebraic system of equation for the outer solution's terms $\bar{z}_{j,1}(t)$, $\bar{\lambda}_{\rm uj,1}(t)$, $\bar{\lambda}_{\rm vj,1}(t)$, $\bar{\mu}_{\rm j,1}(t)$, $(j = 1,2)$ in the time interval $[0,t_{f}]$:
\begin{eqnarray}\label{outer-solution-system-1}
\frac{d\bar{z}_{1,1}(t)}{dt} =
A_{1}(t)\bar{z}_{1,1}(t) + A_{2}(t)\bar{z}_{2,1}(t) - S_{\rm u,1}(t)\bar{\lambda}_{\rm u1,1}(t),\nonumber\\
\frac{d\bar{z}_{2,0}(t)}{dt} = A_{3}(t)\bar{z}_{1,0}(t) + A_{4}(t)\bar{z}_{2,0}(t) - S_{\rm u,2}^{T}(t)\bar{\lambda}_{\rm u1,0}(t) - \bar{\lambda}_{\rm v2,1}(t),\nonumber\\
\frac{d\bar{\lambda}_{\rm u1,1}(t)}{dt} = - D_{\rm u,1}(t)\bar{z}_{1,1}(t) - D_{\rm u,2}(t)\bar{z}_{2,1}(t) - A_{1}^{T}(t)\bar{\lambda}_{\rm u1,1}(t) + D_{\rm v,1}(t)\bar{\mu}_{1,1}(t),\nonumber\\
0 = - D_{\rm u,2}^{T}(t)\bar{z}_{1,1}(t) - D_{\rm u,3}(t)\bar{z}_{2,1}(t) - A_{2}^{T}(t)\bar{\lambda}_{\rm u1,1}(t)  + D_{\rm v,2}(t)\bar{\mu}_{2,1}(t),\nonumber\\
\frac{d\bar{\lambda}_{\rm v1,1}(t)}{dt} = - D_{\rm v,1}(t)\bar{z}_{1,1}(t) - A_{1}^{T}(t)\bar{\lambda}_{\rm v1,1}(t),\nonumber\\
0 = - D_{\rm v,2}(t)\bar{z}_{2,1}(t) - A_{2}^{T}(t)\bar{\lambda}_{\rm v1,1}(t),\nonumber\\
\frac{d\bar{\mu}_{1,1}(t)}{dt} = A_{1}(t)\bar{\mu}_{1,1}(t) + A_{2}(t)\bar{\mu}_{2,1}(t),\nonumber\\
\frac{d\bar{\mu}_{2,0}(t)}{dt} = \bar{\lambda}_{\rm u2,1}(t) + A_{3}(t)\bar{\mu}_{1,0}(t) + A_{4}(t)\bar{\mu}_{2,0}(t),\nonumber\\
\end{eqnarray}
\begin{eqnarray}\label{outer-solution-cond-1} \bar{z}_{1,1}(0) = - z_{1,1}^{0}(0) = A_{2}(0)D_{\rm v,2}^{-1/2}(0)\big(z_{0 2} - \bar{z}_{2,0}(0)\big),\nonumber\\   \bar{\lambda}_{\rm u1,1}(t_{f}) = 0,\  \  \  \  \bar{\lambda}_{\rm v1,1}(t_{f}) = 0,\  \  \  \  \bar{\mu}_{1,1}(0) = - \mu_{1,1}^{0}(0) = A_{2}(0)\int_{0}^{+\infty}\mu_{2,0}^{0}(\sigma)d\sigma.\nonumber\\
\end{eqnarray}

Note that, for the system (\ref{outer-solution-system-1})-(\ref{outer-solution-cond-1}), the assertion similar to Remark \ref{no-cond-for-two-variables} is valid.

Resolving the second, the fourth, the sixth and the eighth equations in the system (\ref{outer-solution-system-1}) with respect to $\bar{\lambda}_{\rm v2,1}(t)$, $\bar{\mu}_{2,1}(t)$, $\bar{z}_{2,1}(t)$ and $\bar{\lambda}_{\rm u2,1}(t)$, respectively, and taking into account the invertibility of the matrix $D_{\rm v,2}(t)$ for all $t \in [0,t_{f}]$, we obtain
\begin{eqnarray}\label{bar-lambda_v21}\bar{\lambda}_{\rm v2,1}(t) = A_{3}(t)\bar{z}_{1,0}(t) + A_{4}(t)\bar{z}_{2,0}(t) - S_{\rm u,2}^{T}(t)\bar{\lambda}_{\rm u1,0}(t) - d\bar{z}_{2,0}(t)/dt,\end{eqnarray}
\begin{eqnarray}\label{bar-mu_21} \bar{\mu}_{2,1}(t) = D_{\rm v,2}^{-1}(t)\big[D_{\rm u,2}^{T}(t)\bar{z}_{1,1}(t) + D_{\rm u,3}(t)\bar{z}_{2,1}(t) + A_{2}^{T}(t)\bar{\lambda}_{\rm u1,1}(t)\big],\end{eqnarray}
\begin{eqnarray}\label{bar-z_21}\bar{z}_{2,1}(t) = - D_{\rm v,2}^{-1}(t)A_{2}^{T}(t)\bar{\lambda}_{\rm v1,1}(t),
\end{eqnarray}
\begin{eqnarray}\label{bar-lambda_u21} \bar{\lambda}_{\rm u2,1}(t) = d\bar{\mu}_{2,0}(t)/dt - A_{3}(t)\bar{\mu}_{1,0}(t) - A_{4}(t)\bar{\mu}_{2,0}(t).\end{eqnarray}

Substituting (\ref{bar-z_21}) into the equation (\ref{bar-mu_21}), we rewrite the expression for $\bar{\mu}_{2,1}(t)$ as:
\begin{eqnarray}\label{bar-mu_21-new-form}\bar{\mu}_{2,1}(t) = D_{\rm v,2}^{-1}(t)\big[D_{\rm u,2}^{T}(t)\bar{z}_{1,1}(t) - D_{\rm u,3}(t)D_{\rm v,2}^{-1}(t)A_{2}^{T}(t)\bar{\lambda}_{\rm v1,1}(t) + A_{2}^{T}(t)\bar{\lambda}_{\rm u1,1}(t)\big].\  \  \  \
\end{eqnarray}

Finally, substitution of (\ref{bar-z_21}),(\ref{bar-mu_21-new-form}) into the first, third and seventh equations of the  system (\ref{outer-solution-system-1}) and use of the fifth equation of this system yield after a routine algebra the following set of four differential equations with respect to $\bar{z}_{1,1}(t)$, $\bar{\lambda}_{\rm u1,1}(t)$, $\bar{\lambda}_{\rm v1,1}(t)$, $\bar{\mu}_{1,1}(t)$:
\begin{eqnarray}\label{set-four-diff-eq-outer-sol-1}
\frac{d\bar{z}_{1,1}(t)}{dt} = A_{1}(t)\bar{z}_{1,1}(t) - S_{\rm u,1}(t)\bar{\lambda}_{\rm u1,1}(t) - A_{2}(t)D_{\rm v,2}^{-1}(t)A_{2}^{T}(t)\bar{\lambda}_{\rm v1,1}(t),\nonumber\\
\frac{d\bar{\lambda}_{\rm u1,1}(t)}{dt} = - D_{\rm u,1}(t)\bar{z}_{1,1}(t) - A_{1}^{T}(t)\bar{\lambda}_{\rm u1,1}(t) + D_{\rm u,2}(t)D_{\rm v,2}^{-1}(t)A_{2}^{T}(t)\bar{\lambda}_{\rm v1,1}(t)\nonumber\\ + D_{\rm v,1}(t)\bar{\mu}_{\rm 1,1}(t),\nonumber\\
\frac{d\bar{\lambda}_{\rm v1,1}(t)}{dt} = - D_{\rm v,1}(t)\bar{z}_{1,1}(t) -
A_{1}^{T}(t)\bar{\lambda}_{\rm v1,1}(t),\nonumber\\
\frac{d\bar{\mu}_{1,1}(t)}{dt} = A_{2}(t)D_{\rm v,2}^{-1}(t)D_{\rm u,2}^{T}(t)\bar{z}_{1,1}(t) + A_{2}(t)D_{\rm v,2}^{-1}(t)A_{2}^{T}(t)\bar{\lambda}_{\rm u1,1}(t)\nonumber\\
- A_{2}(t)D_{\rm v,2}^{-1}(t)D_{u,3}(t)D_{\rm v,2}^{-1}(t)A_{2}^{T}(t)\bar{\lambda}_{\rm v1,1}(t) + A_{1}(t)\bar{\mu}_{1,1}(t).\nonumber\\
\end{eqnarray}

The system of the differential equations (\ref{set-four-diff-eq-outer-sol-1}) is subject to the boundary conditions  (\ref{outer-solution-cond-1}).

\begin{remark}\label{existence-uniqueness-solution-1} It is important to note the following. The boundary-value problem (\ref{set-four-diff-eq-outer-sol-1}),(\ref{outer-solution-cond-1}) is of the same form as the boundary-value problem (\ref{set-four-diff-eq-outer-sol}),(\ref{outer-solution-cond}). Since, due to the assumption (A7), the problem (\ref{set-four-diff-eq-outer-sol}),(\ref{outer-solution-cond}) has the unique solution, then the problem (\ref{set-four-diff-eq-outer-sol-1}),(\ref{outer-solution-cond-1}) also has the unique solution ${\rm col}\big(\bar{z}_{1,1}(t),\bar{\lambda}_{\rm u1,1}(t),\bar{\lambda}_{\rm v1,1}(t),\bar{\mu}_{1,1}(t)\big)$,\  \ $t \in [0,t_{f}]$.
\end{remark}

Hence, the solution of the boundary-value problem (\ref{set-four-diff-eq-outer-sol-1}),(\ref{outer-solution-cond-1}), along with $\bar{\lambda}_{\rm v2,1}(t)$, $\bar{\mu}_{2,1}(t)$, $\bar{z}_{2,1}(t)$ and $\bar{\lambda}_{\rm u2,1}(t)$ (see the equations (\ref{bar-lambda_v21}),(\ref{bar-mu_21}),(\ref{bar-z_21}) and (\ref{bar-lambda_u21})) constitutes the coefficient of the first-order (with respect to $\varepsilon$) addend in the outer solution.

\subsubsection{Obtaining the boundary corrections $z_{2,1}^{0}(\xi)$, $\lambda_{\rm u 2,1}^{0}(\xi)$, $\lambda_{\rm v 2,1}^{0}(\xi)$, $\mu_{2,1}^{0}(\xi)$ and $z_{2,1}^{f}(\rho)$, $\lambda_{\rm u 2,1}^{f}(\rho)$, $\lambda_{\rm v 2,1}^{f}(\rho)$, $\mu_{2,1}^{f}(\rho)$}\label{obtaining-bound-corr-3,4-1}

Using Remark \ref{explanation-zero-order-solution} and the equations (\ref{bound-corr-1}), (\ref{bound-corr-2}), (\ref{boundary-correction-3}), (\ref{bound-corrections-4}), (\ref{bound-corr-1-1}), (\ref{bound-corr-2-1}), we derive the following differential equations for these boundary corrections:
\begin{eqnarray}\label{eq-bound-corr-3-1} \frac{dz_{2,1}^{0}(\xi)}{d\xi} = A_{4}(0)z_{2, 0}^{0}(\xi) -
\lambda_{\rm v 2,1}^{0}(\xi),\  \  \  \ \xi \ge 0,\nonumber\\
\frac{d\lambda_{\rm u 2,1}^{0}(\xi)}{d\xi} = - D_{\rm u,2}^{T}(0)z_{1,1}^{0}(\xi) - D_{\rm u,3}(0)z_{2,1}^{0}(\xi) - \frac{dD_{\rm u,3}(0)}{dt}\xi z_{2, 0}^{0}(\xi)\nonumber\\ - A_{2}^{T}(0)\lambda_{\rm u1,1}^{0}(\xi)
- A_{4}^{T}(0)\lambda_{\rm u2,0}^{0}(\xi) + D_{\rm v,2}(0)\mu_{2,1}^{0}(\xi) + \frac{dD_{\rm v,2}(0)}{dt}\xi\mu_{2,0}^{0}(\xi),\  \  \  \ \xi \ge 0,\nonumber\\
\frac{d\lambda_{\rm v 2,1}^{0}(\xi)}{d\xi} = -D_{\rm v,2}(0)z_{2,1}^{0}(\xi) - \frac{dD_{\rm v,2}(0)}{dt}\xi z_{2, 0}^{0}(\xi) - A_{4}^{T}(0)\lambda_{\rm v2,0}^{0}(\xi),\  \  \  \ \xi \ge 0,\nonumber\\
\frac{d\mu_{2,1}^{0}(\xi)}{d\xi} = \lambda_{\rm u2,1}^{0}(\xi) + A_{4}(0)\mu_{2,0}^{0}(\xi),\  \  \  \ \xi \ge 0,\nonumber\\
\end{eqnarray}
\begin{eqnarray}\label{eq-bound-corr-4-1} \frac{dz_{2,1}^{f}(\rho)}{d\rho} = - \lambda_{\rm v 2,1}^{f}(\rho),\  \  \  \ \rho \le 0,\nonumber\\
\frac{d\lambda_{\rm u 2,1}^{f}(\rho)}{d\rho} = - D_{\rm u,3}(t_{f})z_{2,1}^{f}(\rho) + D_{\rm v,2}(t_{f})\mu_{2,1}^{f}(\rho),\  \  \  \ \rho \le 0,\nonumber\\
\frac{d\lambda_{\rm v 2,1}^{f}(\rho)}{d\rho} = - D_{\rm v,2}(t_{f})z_{2,1}^{f}(\rho),\  \  \  \ \rho \le 0,\nonumber\\
\frac{d\mu_{2,1}^{f}(\rho)}{d\rho} = \lambda_{\rm u2,1}^{f}(\rho),\  \  \  \ \rho \le 0.\nonumber\\
\end{eqnarray}

Similarly to (\ref{additional-cond-for-corr-3}),(\ref{additional-cond-for-corr-4}), we consider the systems (\ref{eq-bound-corr-3-1}) and (\ref{eq-bound-corr-4-1}) subject to the conditions
\begin{equation}\label
{additional-cond-for-corr-3-1} \lim_{\xi \rightarrow + \infty}z_{2,1}^{0}(\xi) = 0,\  \  \ \lim_{\xi \rightarrow + \infty}\lambda_{\rm u 2,1}^{0}(\xi) = 0,\  \  \ \lim_{\xi \rightarrow + \infty}\lambda_{\rm v 2,1}^{0}(\xi) = 0,\ \ \ \lim_{\xi \rightarrow + \infty}\mu_{2,1}^{0}(\xi) = 0,\end{equation}
\begin{equation}\label
{additional-cond-for-corr-4-1} \lim_{\rho \rightarrow - \infty}z_{2,1}^{f}(\rho) = 0,\  \  \ \lim_{\rho \rightarrow - \infty}\lambda_{\rm u 2,1}^{f}(\rho) = 0,\  \  \ \lim_{\rho \rightarrow - \infty}\lambda_{\rm v 2,1}^{f}(\rho) = 0,\ \ \ \lim_{\rho \rightarrow - \infty}\mu_{2,1}^{0}(\rho) = 0.\end{equation}

Moreover, similarly to (\ref{init-cond-for-z_2^0-mu_2^0}),(\ref{termin-cond-for-lambda_u2^f}), we obtain two additional conditions for the solutions of the systems (\ref{eq-bound-corr-3-1}) and (\ref{eq-bound-corr-4-1})
\begin{eqnarray}\label{init-cond-for-z_2^0-1} z_{2,1}^{0}(0) = - \bar{z}_{2,1}(0),\  \  \  \ \mu_{2,1}^{0}(0) = - \bar{\mu}_{2,1}(0),\end{eqnarray}
\begin{eqnarray}\label{termin-cond-for-lambda_u2^f-1} \lambda_{\rm u 2,1}^{f}(0) = - \bar{\lambda}_{\rm u2,1}(t_{f}),\ \  \  \ \lambda_{\rm v 2,1}^{f}(0) = - \bar{\lambda}_{\rm v2,1}(t_{f}).\end{eqnarray}

We start with the system (\ref{eq-bound-corr-3-1}) subject to the conditions (\ref{additional-cond-for-corr-3-1}) and (\ref{init-cond-for-z_2^0-1}).

Similarly to (\ref{transform-corr-3}), we make the following transformation of the state variables in the system (\ref{eq-bound-corr-3-1}) and the conditions (\ref{additional-cond-for-corr-3-1}),(\ref{init-cond-for-z_2^0-1}):
\begin{equation}\label{transform-corr-3-1} \left(\begin{array}{l} z_{2,1}^{0}(\xi)\\ \\ \lambda_{\rm v 2,1}^{0}(\xi)\end{array}\right) =\Theta(0)\left(\begin{array}{l}x^{0}_{1}(\xi)\\ \\ y^{0}_{1}(\xi)\end{array}\right),\ \  \  \ \left(\begin{array}{l} \mu_{2,1}^{0}(\xi)\\ \\ \lambda_{\rm u 2,1}^{0}(\xi)\end{array}\right) = \Theta(0)\left(\begin{array}{l}\eta^{0}_{1}(\xi)\\ \\ \zeta^{0}_{1}(\xi)\end{array}\right),\end{equation}
where the invertible matrix $\Theta(t)$, $t \in [0,t_{f}]$ is given by (\ref{matr-Theta_0}) and its inverse matrix $\Theta^{-1}(t)$ is given by (\ref{matr-inverse-Theta_0}); $x^{0}_{1}(\xi) \in \mathbb{R}^{s}$, $y^{0}_{1}(\xi) \in \mathbb{R}^{s}$, $\eta_{1}^{0}(\xi) \in \mathbb{R}^{s}$ and $\zeta_{1}^{0}(\xi) \in \mathbb{R}^{s}$ are new state variables.

The transformation (\ref{transform-corr-3-1}) converts the system (\ref{eq-bound-corr-3-1}) and the conditions (\ref{additional-cond-for-corr-3-1}),(\ref{init-cond-for-z_2^0-1}) to the system
\begin{eqnarray}\label{new-eq-bound-corr-3-1} \frac{dx^{0}_{1}(\xi)}{d\xi} = D_{\rm v,2}^{1/2}(0)x^{0}_{1}(\xi) + f_{x}(\xi),\  \  \  \ \xi \ge 0,\nonumber\\
\frac{dy^{0}_{1}(\xi)}{d\xi} = - D_{\rm v,2}^{1/2}(0)y^{0}_{1}(\xi) + f_{y}(\xi),\  \  \  \ \xi \ge 0,\nonumber\\
\frac{d\eta_{1}^{0}(\xi)}{d\xi} = 0.5D_{\rm v,2}^{-1/2}(0)D_{\rm u,3}(0)x_{1}^{0}(\xi) + 0.25D_{\rm v,2}^{-1/2}(0)D_{\rm u,3}(0)D_{\rm v,2}^{-1/2}(0)y_{1}^{0}(\xi)\nonumber\\
- D_{\rm v,2}^{1/2}(0)\eta_{1}^{0}(\xi) + f_{\eta}(\xi),\  \  \ \xi \ge 0,\nonumber\\
\frac{d\zeta_{1}^{0}(\xi)}{d\xi} = - D_{\rm u,3}(0)x_{1}^{0}(\xi) - 0.5D_{\rm u,3}(0)D_{\rm v,2}^{-1/2}(0)y_{1}^{0}(\xi)\nonumber\\
+ D_{\rm v,2}^{1/2}(0)\zeta_{1}^{0}(\xi) + f_{\zeta}(\xi),\  \  \ \xi \ge 0,\nonumber\\
\end{eqnarray}
and the conditions
\begin{equation}\label{transf-add-cond-for-corr-3-1}\lim_{\xi \rightarrow + \infty}x^{0}_{1}(\xi) = 0,\  \  \  \lim_{\xi \rightarrow + \infty}y^{0}_{1}(\xi) = 0,\  \  \ \lim_{\xi \rightarrow + \infty}\eta_{1}^{0}(\xi) = 0,\  \  \ \lim_{\xi \rightarrow + \infty}\zeta_{1}^{0}(\xi) = 0,\end{equation}
\begin{eqnarray}\label{init-cond-corr-3-transf-1} x^{0}_{1}(0) = - 0.5\bar{z}_{2,1}(0) - 0.5D_{\rm v,2}^{-1/2}(0)\lambda_{\rm v 2,1}^{0}(0),\nonumber\\
y^{0}_{1}(0) = - D_{\rm v,2}^{1/2}(0)\bar{z}_{2,1}(0) + \lambda_{\rm v 2,1}^{0}(0),\nonumber\\
\eta_{1}^{0}(0) = - 0.5\bar{\mu}_{2,1}(0) - 0.5D_{\rm v,2}^{-1/2}(0)\lambda_{\rm u2,1}^{0}(0),\nonumber\\
\zeta_{1}^{0}(0) = - D_{\rm v,2}^{1/2}(0)\bar{\mu}_{2,1}(0) + \lambda_{\rm u2,1}^{0}(0),\nonumber\\
\end{eqnarray}
where
\begin{eqnarray}\label{f_x,y-1}f_{x}(\xi) = 0.5\bigg(A_{4}(0)z_{2,0}^{0}(\xi) + D_{\rm v,2}^{-1/2}(0)\frac{dD_{\rm v,2}(0)}{dt}\xi z_{2,0}^{0}(\xi)\nonumber\\
+ D_{\rm v,2}^{-1/2}(0)A_{4}^{T}(0)\lambda_{\rm v2,0}^{0}(\xi)\bigg),\nonumber\\
f_{y}(\xi) = D_{\rm v,2}^{1/2}(0)A_{4}(0)z_{2,0}^{0}(\xi) - \frac{dD_{\rm v,2}(0)}{dt}\xi z_{2,0}^{0}(\xi) - A_{4}^{T}(0)\lambda_{\rm v2,0}^{0}(\xi),\nonumber\\
f_{\eta}(\xi) = 0.5\bigg(A_{4}(0)\mu_{2,0}^{0}(\xi) + D_{\rm v,2}^{-1/2}(0)D_{\rm u,2}^{T}(0)z_{1,1}^{0}(\xi)\nonumber\\ + D_{\rm v,2}^{-1/2}(0)\frac{dD_{\rm u,3}(0)}{dt}\xi z_{2,0}^{0}(\xi) + D_{\rm v,2}^{-1/2}(0)A_{2}^{T}(0)\lambda_{\rm u1,1}^{0}(\xi)\nonumber\\ + D_{\rm v,2}^{-1/2}(0)A_{4}^{T}(0)\lambda_{\rm u2,0}^{0}(\xi) - D_{\rm v,2}^{-1/2}(0)\frac{dD_{\rm v,2}(0)}{dt}\xi \mu_{2,0}^{0}(\xi)\bigg),\nonumber\\
f_{\zeta}(\xi) = D_{\rm v,2}^{1/2}(0)A_{4}(0)\mu_{2,0}^{0}(\xi) - D_{\rm u,2}^{T}(0)z_{1,1}^{0}(\xi) - \frac{dD_{\rm u,3}(0)}{dt}\xi z_{2,0}^{0}(\xi)\nonumber\\
- A_{2}^{T}(0)\lambda_{\rm u1,1}^{0}(\xi) - A_{4}^{T}(0)\lambda_{\rm u2,0}^{0}(\xi) + \frac{dD_{\rm v,2}(0)}{dt}\xi \mu_{2,0}^{0}(\xi).\nonumber\\
\end{eqnarray}

By virtue of the inequalities in (\ref{ineq-for-bound-correc-3}),(\ref{ineq-bound-terms-1-1}), we immediately have the estimates for $f_{x}(\xi)$, $f_{y}(\xi)$, $f_{\eta}(\xi)$ and $f_{\zeta}(\xi)$
\begin{eqnarray}\label{estimates-f_x,y}\big\|f_{x}(\xi)\big\| \le \alpha_{x}\exp(- 0.5\beta^{0}\xi),\  \  \  \ \big\|f_{y}(\xi)\big\| \le \alpha_{y}\exp(- 0.5\beta^{0}\xi),\  \  \  \ \xi \ge 0,\nonumber\\
\big\|f_{\eta}(\xi)\big\| \le \alpha_{\eta}\exp(- 0.5\beta^{0}\xi),\  \  \  \ \big\|f_{\zeta}(\xi)\big\| \le \alpha_{\zeta}\exp(- 0.5\beta^{0}\xi),\  \  \  \ \xi \ge 0,\nonumber\\
\end{eqnarray}
where $\alpha_{x}>0$, $\alpha_{y}>0$, $\alpha_{\eta}>0$ and $\alpha_{\zeta}>0$ are some constants; the positive constant $\beta^{0}$ was introduced in (\ref{ineq-for-bound-correc-3}).

Solving the first equation in (\ref{new-eq-bound-corr-3-1}) subject to the first condition in (\ref{transf-add-cond-for-corr-3-1}) and using the positive definiteness of $D_{\rm v,2}^{1/2}(0)$ and the estimate for $f_{x}(\xi)$ in (\ref{estimates-f_x,y}), we obtain after a routine algebra
\begin{eqnarray}\label{x_0^1}x_{1}^{0}(\xi) = - \int_{\xi}^{+ \infty}\exp\big(- D_{\rm v,2}^{1/2}(0)(\sigma - \xi)\big)f_{x}(\sigma)d\sigma,\  \  \  \ \xi \ge 0.\end{eqnarray}
The latter, along with the estimate for $f_{x}(\sigma)$, yields the inequality
\begin{eqnarray}\label{ineq-x_0^1}\big\|x_{1}^{0}(\xi)\big\| \le \omega_{x}\exp(- 0.5\beta^{0}\xi),\  \  \  \ \xi \ge 0,\end{eqnarray}
where $\omega_{x} > 0$ is some constant.

Elimination of $\lambda_{\rm v2,1}^{0}(0)$ from the first and the second equations in (\ref{init-cond-corr-3-transf-1}) and use of (\ref{x_0^1}) yield the initial condition for $y_{1}^{0}(\xi)$
\begin{eqnarray}\label{y_1^0(0)} y_{1}^{0}(0) = 2D_{\rm v,2}^{1/2}(0)\left(- \bar{z}_{2,1}(0) + \int_{0}^{+ \infty}\exp\big(- D_{\rm v,2}^{1/2}(0)\sigma\big)f_{x}(\sigma)d\sigma\right).
\end{eqnarray}

Solving the second equation in (\ref{new-eq-bound-corr-3-1}) subject to this initial condition, we have
\begin{eqnarray}\label{y_1^0}y_{1}^{0}(\xi) = \exp\big(- D_{\rm v,2}^{1/2}(0)\xi\big)y_{1}^{0}(0) + \int_{0}^{\xi}\exp\big(- D_{\rm v,2}^{1/2}(0)(\xi - \sigma)\big)f_{y}(\sigma)d\sigma,\  \  \ \xi \ge 0.\end{eqnarray}
Using the positive definiteness of the matrix $D_{\rm v,2}^{1/2}(0)$ and the estimate for $f_{y}(\xi)$ (see the equation (\ref{estimates-f_x,y})), we directly obtain that the vector-valued function $y_{1}^{0}(\xi)$ satisfies the limit condition in (\ref{transf-add-cond-for-corr-3-1}). Moreover, this function satisfies the inequality
\begin{eqnarray}\label{ineq-y_0^1}\big\|y_{1}^{0}(\xi)\big\| \le \omega_{y}\exp(- 0.5\beta^{0}\xi),\  \  \  \ \xi \ge 0,\end{eqnarray}
where $\omega_{y} > 0$ is some constant.

Now, proceed to the solution of the third and the fourth differential equations in (\ref{new-eq-bound-corr-3-1}) subject to the third and the fourth conditions in (\ref{transf-add-cond-for-corr-3-1}) and in (\ref{init-cond-corr-3-transf-1}). Let us denote
\begin{eqnarray}\label{g_eta-g_zeta}g_{\eta}(\xi) \stackrel{\triangle}{=} 0.5D_{\rm v,2}^{-1/2}(0)D_{\rm u,3}(0)x_{1}^{0}(\xi) + 0.25D_{\rm v,2}^{-1/2}(0)D_{\rm u,3}(0)D_{\rm v,2}^{-1/2}(0)y_{1}^{0}(\xi)\nonumber\\
+ f_{\eta}(\xi),\  \  \  \ \xi \ge 0,\nonumber\\
g_{\zeta}(\xi) \stackrel{\triangle}{=} - D_{\rm u,3}(0)x_{1}^{0}(\xi) - 0.5D_{\rm u,3}(0)D_{\rm v,2}^{-1/2}(0)y_{1}^{0}(\xi) + f_{\zeta}(\xi),\nonumber  \\\xi \ge 0.\nonumber\\
\end{eqnarray}

Using these notations, we can rewrite the third and the second equations in (\ref{new-eq-bound-corr-3-1}) as:
\begin{eqnarray}\label{eqs-eta_1^0-zeta_1^0}\frac{d\eta_{1}^{0}(\xi)}{d\xi} = - D_{\rm v,2}^{1/2}(0)\eta_{1}^{0}(\xi) + g_{\eta}(\xi),\  \  \ \xi \ge 0,\nonumber\\
\frac{d\zeta_{1}^{0}(\xi)}{d\xi} = D_{\rm v,2}^{1/2}(0)\zeta_{1}^{0}(\xi) + g_{\zeta}(\xi),\  \  \ \xi \ge 0.\nonumber\\
\end{eqnarray}

Using the positive definiteness of the matrix $D_{\rm v,2}^{1/2}(0)$ and the inequalities for $f_{\eta}(\xi)$, $f_{\zeta}(\xi)$, $x_{1}^{0}(\xi)$, $y_{1}^{0}(\xi)$ (see the equations (\ref{estimates-f_x,y}),(\ref{ineq-x_0^1}),(\ref{ineq-y_0^1})), we obtain (similarly to (\ref{x_0^1}),(\ref{y_1^0(0)})-(\ref{y_1^0})) the unique solution of the system (\ref{eqs-eta_1^0-zeta_1^0}) subject to the third and the fourth conditions in (\ref{transf-add-cond-for-corr-3-1}) and in (\ref{init-cond-corr-3-transf-1})
\begin{eqnarray}\label{eta_1^0-zeta_1^0}\eta_{1}^{0}(\xi) = \exp\big(- D_{\rm v,2}^{1/2}(0)\xi\big)\eta_{1}^{0}(0) + \int_{0}^{\xi}\exp\big(- D_{\rm v,2}^{1/2}(0)(\xi - \sigma)\big)g_{\eta}(\sigma)d\sigma,\  \ \xi \ge 0,\nonumber\\
\zeta_{1}^{0}(\xi) = - \int_{\xi}^{+ \infty}\exp\big(- D_{\rm v,2}^{1/2}(0)(\sigma - \xi)\big)g_{\zeta}(\sigma)d\sigma,\  \  \  \ \xi \ge 0,\nonumber\\
\end{eqnarray}
where
\begin{eqnarray}\label{eta_1^0(0)} \eta_{1}^{0}(0) = - \bar{\mu}_{2,1}(0) + 0.5D_{\rm v,2}^{-1/2}(0)\int_{0}^{+\infty}\exp\big(- D_{\rm v,2}^{1/2}(0)\sigma\big)g_{\zeta}(\sigma)d\sigma.\end{eqnarray}

The vector-valued functions $\eta_{1}^{0}(\xi)$ and $\zeta_{1}^{0}(\xi)$ satisfy the inequalities
\begin{eqnarray}\label{ineq-eta_1^0-zeta_1^0}\big\|\eta_{1}^{0}(\xi)\big\| \le \omega_{\eta}\exp(- 0.5\beta^{0}\xi),\  \  \  \ \big\|\zeta_{1}^{0}(\xi)\big\| \le \omega_{\zeta}\exp(- 0.5\beta^{0}\xi\big),\  \  \  \ \xi \ge 0,\end{eqnarray}
where $\omega_{\eta} > 0$ and $\omega_{\zeta} > 0$ are some constants.

Thus, the vector-valued functions $x^{0}_{1}(\xi)$, $y^{0}_{1}(\xi)$, $\eta_{1}^{0}(\xi)$, $\zeta_{1}^{0}(\xi)$ (see the equations (\ref{x_0^1}),(\ref{y_1^0}),(\ref{eta_1^0-zeta_1^0})) constitute the solution of the system (\ref{new-eq-bound-corr-3-1}) satisfying the conditions (\ref{transf-add-cond-for-corr-3-1}),(\ref{init-cond-corr-3-transf-1}).
Substitution of this solution into (\ref{transform-corr-3-1}) and use of (\ref{matr-Theta_0}) yield, after a routine algebra, the solution of the system (\ref{eq-bound-corr-3-1}) satisfying the conditions (\ref{additional-cond-for-corr-3-1}),(\ref{init-cond-for-z_2^0-1})
\begin{eqnarray}\label{boundary-correction-3-1}
z_{2,1}^{0}(\xi) = x_{1}^{0}(\xi) + 0.5D_{\rm v,2}^{-1/2}(0)y_{1}^{0}(\xi),\  \  \  \ \xi \ge 0,\nonumber\\
\lambda_{\rm v2,1}^{0}(\xi) = - D_{\rm v,2}^{1/2}(0)x_{1}^{0}(\xi) + 0.5y_{1}^{0}(\xi),\  \  \  \ \xi \ge 0.\nonumber\\
\lambda_{\rm u2,1}^{0}(\xi) = - D_{\rm v,2}^{1/2}(0)\eta_{1}^{0}(\xi) + 0.5\zeta_{1}^{0}(\xi),\  \  \  \ \xi \ge 0,\nonumber\\
\mu_{2,1}^{0}(\xi) = \eta_{1}^{0}(\xi) + 0.5D_{\rm v,2}^{-1/2}(0)\zeta_{1}^{0}(\xi),\  \  \  \ \xi \ge 0.\nonumber\\
\end{eqnarray}
Due to the estimates (\ref{ineq-x_0^1}),(\ref{ineq-y_0^1}),(\ref{ineq-eta_1^0-zeta_1^0}), this solution satisfies the inequalities
\begin{eqnarray}\label{ineq-for-bound-correc-3-1} \big\|z_{2,1}^{0}(\xi)\big\| \le b_{\rm z}^{0}\exp(- 0.5\beta^{0}\xi),\  \  \  \  \big\|\lambda_{\rm v 2,1}^{0}(\xi)\big\| \le b_{\rm v}^{0}\exp(- 0.5\beta^{0}\xi),\  \  \  \ \xi \ge 0,\nonumber\\
\big\|\lambda_{\rm u 2,1}^{0}(\xi)\big\| \le b_{\rm u}^{0}\exp(- 0.5\beta^{0}\xi),\  \  \  \ \big\|\mu_{2,1}^{0}(\xi)\big\| \le b_{\mu}^{0}\exp(- 0.5\beta^{0}\xi),\  \  \  \ \xi \ge 0,\nonumber\\
\end{eqnarray}
where $b_{\rm z}^{0} > 0$, $b_{\rm v}^{0} > 0$, $b_{\rm u}^{0} > 0$, $b_{\mu}^{0} > 0$ are some constants; $\beta^{0} > 0$ is the minimal eigenvalue of the matrix $D_{\rm u,2}^{1/2}(0)$.

Proceed to the solution of the system (\ref{eq-bound-corr-4-1}) subject to the conditions (\ref{additional-cond-for-corr-4-1}) and (\ref{termin-cond-for-lambda_u2^f-1}).
Similarly to (\ref{transform-corr-3-1}), we make the following transformation of the state variables in the problem  (\ref{eq-bound-corr-4-1}),(\ref{additional-cond-for-corr-4-1}),(\ref{termin-cond-for-lambda_u2^f-1}):
\begin{equation}\label{transform-corr-4-1} \left(\begin{array}{l} z_{2,1}^{f}(\rho)\\ \\ \lambda_{\rm v 2,1}^{f}(\rho)\end{array}\right) =\Theta(t_{f})\left(\begin{array}{l}x^{f}_{1}(\rho)\\ \\ y^{f}_{1}(\rho)\end{array}\right),\ \  \  \ \left(\begin{array}{l} \mu_{2,1}^{f}(\rho)\\ \\ \lambda_{\rm u 2,1}^{f}(\rho)\end{array}\right) = \Theta(t_{f})\left(\begin{array}{l}\eta^{f}_{1}(\rho)\\ \\ \zeta^{f}_{1}(\rho)\end{array}\right),\end{equation}
where $x^{f}_{1}(\rho) \in \mathbb{R}^{s}$, $y^{f}_{1}(\rho) \in \mathbb{R}^{s}$, $\eta^{f}_{1}(\rho) \in \mathbb{R}^{s}$, $\zeta^{f}_{1}(\rho) \in \mathbb{R}^{s}$ are new state variables.

The transformation (\ref{transform-corr-4-1}) converts the problem  (\ref{eq-bound-corr-4-1}),(\ref{additional-cond-for-corr-4-1}),(\ref{termin-cond-for-lambda_u2^f-1}) to the problem
\begin{eqnarray}\label{new-eq-bound-corr-4-1} \frac{dx^{f}_{1}(\rho)}{d\rho} = D_{\rm v,2}^{1/2}(t_{f})x^{f}_{1}(\rho),\  \  \ \rho \le 0,\nonumber\\
\frac{dy^{f}_{1}(\rho)}{d\rho} = - D_{\rm v,2}^{1/2}(t_{f})y^{f}_{1}(\rho),\  \  \ \rho \le 0,\nonumber\\
\frac{d\eta_{1}^{f}(\rho)}{d\rho} = 0.5D_{\rm v,2}^{-1/2}(t_{f})D_{\rm u,3}(t_{f})x_{1}^{f}(\rho) + 0.25D_{\rm v,2}^{-1/2}(t_{f})D_{\rm u,3}(t_{f})D_{\rm v,2}^{-1/2}(t_{f})y_{1}^{f}(\rho)\nonumber\\
- D_{\rm v,2}^{1/2}(t_{f})\eta_{1}^{f}(\rho),\  \ \rho \le 0,\nonumber\\
\frac{d\zeta_{1}^{f}(\rho)}{d\rho} = - D_{\rm u,3}(t_{f})x_{1}^{f}(\rho) - 0.5D_{\rm u,3}(t_{f})D_{\rm v,2}^{-1/2}(t_{f})y_{1}^{f}(\rho)\nonumber\\
+ D_{\rm v,2}^{1/2}(t_{f})\zeta_{1}^{f}(\rho),\  \ \rho \le 0,\nonumber\\
\end{eqnarray}
\begin{equation}\label{transf-add-cond-for-corr-4-1}\lim_{\rho \rightarrow - \infty}x^{f}_{1}(\rho) = 0,\  \  \ \lim_{\rho \rightarrow - \infty}y^{f}_{1}(\rho) = 0\  \  \ \lim_{\rho \rightarrow - \infty}\eta_{1}^{f}(\rho) = 0,\  \ \  \lim_{\rho \rightarrow - \infty}\zeta_{1}^{f}(\rho) = 0, \end{equation}
\begin{eqnarray}\label{term-cond-corr-4-transf-1} x^{f}_{1}(0) = 0.5z_{2,1}^{f}(0) + 0.5D_{\rm v,2}^{-1/2}(t_{f})\bar{\lambda}_{\rm v 2,1}(t_{f}),\nonumber\\
y^{f}_{1}(0) = D_{\rm v,2}^{1/2}(t_{f})z_{2,1}^{f}(0) - \bar{\lambda}_{\rm v 2,1}(t_{f}),\nonumber\\
\eta_{1}^{f}(0) = 0.5\mu_{2,1}^{f}(0) + 0.5D_{\rm v,2}^{-1/2}(t_{f})\bar{\lambda}_{\rm u2,1}(t_{f}),\nonumber\\
\zeta_{1}^{f}(0) = D_{\rm v,2}^{1/2}(t_{f})\mu_{2,1}^{f}(0) - \bar{\lambda}_{\rm u2,1}(t_{f}).\nonumber\\
\end{eqnarray}

Solving the system of differential equations (\ref{new-eq-bound-corr-4-1}) subject to the conditions (\ref{transf-add-cond-for-corr-4-1})-(\ref{term-cond-corr-4-transf-1}), we obtain (similarly to (\ref{x_0^1}),(\ref{y_1^0}),(\ref{eta_1^0-zeta_1^0}))
\begin{eqnarray}\label{x_1^f-y_1^f-eta_1^f-zeta_1^f} x^{f}_{1}(\rho) = \exp\big(D_{\rm v,2}^{1/2}(t_{f})\rho\big)D_{\rm v,2}^{- 1/2}(t_{f})\bar{\lambda}_{\rm v2,1}(t_{f}),\  \  \  \ y^{f}_{1}(\rho) \equiv 0,\nonumber\\
\eta_{1}^{f}(\rho) =
0.5\int_{- \infty}^{\rho}\exp\big(- D_{\rm v,2}^{1/2}(t_{f})(\rho - \sigma)\big)D_{\rm v,2}^{-1/2}(t_{f})D_{\rm u,3}(t_{f})D_{\rm v,2}^{- 1/2}(t_{f})\times \nonumber\\
\exp\big(D_{\rm v,2}^{1/2}(t_{f})\sigma\big)\bar{\lambda}_{\rm v2,1}(t_{f})d\sigma,\nonumber\\
\zeta_{1}^{f}(\rho) = \exp\big(D_{\rm v,2}^{1/2}(t_{f})\rho\big)\zeta_{1}^{f}(0)\nonumber\\
- \int_{0}^{\rho}\exp\big(D_{\rm v,2}^{1/2}(t_{f})(\rho - \sigma)\big)D_{\rm u,3}(t_{f})D_{\rm v,2}^{- 1/2}(t_{f})\exp\big(D_{\rm v,2}^{1/2}(t_{f})\sigma\big)\bar{\lambda}_{\rm v2,1}(t_{f})d\sigma,\nonumber\\
\end{eqnarray}
where $\rho \le 0$,
\begin{eqnarray}\label{zeta_1^f(0)} \zeta_{1}^{f}(0) = - 2\bar{\lambda}_{\rm u2,1}(t_{f})\nonumber\\
+ \int_{- \infty}^{0}\exp\big(D_{\rm v,2}^{1/2}(t_{f})\sigma\big)D_{\rm u,3}(t_{f})D_{\rm v,2}^{- 1/2}(t_{f})\exp\big(D_{\rm v,2}^{1/2}(t_{f})\sigma\big)\bar{\lambda}_{\rm v2,1}(t_{f})d\sigma.\end{eqnarray}

Substitution of this solution into (\ref{transform-corr-4-1}) and use of the equation (\ref{matr-Theta_0}) yield, after a routine rearrangement, the solution of the system (\ref{eq-bound-corr-4-1}) satisfying the conditions (\ref{additional-cond-for-corr-4-1}),(\ref{termin-cond-for-lambda_u2^f-1})
\begin{eqnarray}\label{boundary-correction-4-1}
z_{2,1}^{f}(\eta) = \exp\big(D_{\rm v,2}^{1/2}(t_{f})\rho\big)D_{\rm v,2}^{- 1/2}(t_{f})\bar{\lambda}_{\rm v2,1}(t_{f}),\  \  \  \ \rho \le 0,\nonumber\\
\lambda_{\rm v 2,1}^{f}(\rho) = - \exp\big( D_{\rm v,2}^{1/2}(t_{f})\rho\big)\bar{\lambda}_{\rm v2,1}(t_{f}),\  \  \  \ \rho \le 0,\nonumber\\
\lambda_{\rm u2,1}^{f}(\rho) = - D_{\rm v,2}^{1/2}(t_{f})\eta_{1}^{f}(\rho) + 0.5\zeta_{1}^{f}(\rho),\  \  \  \ \rho \le 0,\nonumber\\
\mu_{2,1}^{f}(\rho) = \eta_{1}^{f}(\rho) + 0.5D_{\rm v,2}^{-1/2}(t_{f})\zeta_{1}^{f}(\rho),\  \  \  \ \rho \le 0.\nonumber\\
\end{eqnarray}

This solution satisfies the inequalities
\begin{eqnarray}\label{ineq-for-bound-correc-4-1} \big\|z_{2,1}^{f}(\rho)\big\| \le b^{f}_{z}\exp(\beta^{f}\rho),\  \  \  \ \big\|\lambda_{\rm v 2,1}^{f}(\rho)\big\| \le b^{f}_{\rm v}\exp(\beta^{f}\rho),\  \  \  \ \rho \le 0,\nonumber\\
\big\|\lambda_{\rm u2,1}^{f}(\rho)\big\| \le b^{f}_{\rm u}\exp(0.5\beta^{f}\rho),\  \  \  \ \big\|\mu_{ 2,1}^{f}(\rho)\big\| \le b^{f}_{\mu}\exp(0.5\beta^{f}\rho),\  \  \  \ \rho \le 0,\nonumber\\
\end{eqnarray}
where $b^{f}_{z} > 0$, $b^{f}_{\rm v}>0$, $b^{f}_{\rm u}>0$, $b^{f}_{\mu}>0$  are some constants; $\beta^{f} > 0$ is the minimal eigenvalue of the matrix $D_{\rm v,2}^{1/2}(t_{f})$.

\subsection{Justification of the first-order asymptotic solution to the problem (\ref{equiv-system})-(\ref{equiv-cond})}

Using the results of the previous subsection,
we can rewrite the components of the first-order asymptotic solution to the problem (\ref{equiv-system})-(\ref{equiv-cond}) (see equation (\ref{first-order-asympt-solution})) as follows:
\begin{eqnarray}\label{zero-order-asympt-solution-new-form}
z_{1}^{1}(t,\varepsilon) = \bar{z}_{1,0}(t) + \varepsilon\big[\bar{z}_{1,1}(t) + z_{1,1}^{0}(t/\varepsilon)\big],\nonumber\\
z_{2}^{1}(t,\varepsilon) = \bar{z}_{2,0}(t) + z_{2,0}^{0}(t/\varepsilon) + \varepsilon\big[\bar{z}_{2,1}(t) + z_{2,1}^{0}(t/\varepsilon) + z_{2,1}^{f}\big((t - t_{f})/\varepsilon\big)\big],\nonumber\\
\lambda_{\rm u1}^{1}(t,\varepsilon) = \bar{\lambda}_{\rm u1,0}(t) + \varepsilon\big[\bar{\lambda}_{\rm u1,1}(t) + \lambda_{\rm u1,1}^{0}(t/\varepsilon)\big],\nonumber\\
\lambda_{\rm u2}^{1}(t,\varepsilon) = \lambda_{\rm u2,0}^{0}(t/\varepsilon) + \varepsilon\big[\bar{\lambda}_{\rm u2,1}(t) + \lambda_{\rm u2,1}^{0}(t/\varepsilon) + \lambda_{\rm u2,1}^{f}\big((t - t_{f})/\varepsilon\big)\big],\nonumber\\
\lambda_{\rm v1}^{1}(t,\varepsilon) = \bar{\lambda}_{\rm v1,0}(t) + \varepsilon\bar{\lambda}_{\rm v1,1}(t),\nonumber\\
\lambda_{\rm v2}^{1}(t,\varepsilon) = \lambda_{\rm v2,0}^{0}(t/\varepsilon) + \varepsilon\big[\bar{\lambda}_{\rm v2,1}(t) + \lambda_{\rm v2,1}^{0}(t/\varepsilon) + \lambda_{\rm v2,1}^{f}\big((t - t_{f})/\varepsilon\big)\big],\nonumber\\
\mu_{1}^{1}(t,\varepsilon) = \bar{\mu}_{1,0}(t) + \varepsilon\big[\bar{\mu}_{1,1}(t) + \mu_{1,1}^{0}(t/\varepsilon)\big],\nonumber\\
\mu_{2}^{1}(t,\varepsilon) = \bar{\mu}_{2,0}(t) + \mu_{2,0}^{0}(t/\varepsilon) + \varepsilon\big[\bar{\mu}_{2,1}(t) + \mu_{2,1}^{0}(t/\varepsilon) + \mu_{2,1}^{f}\big((t - t_{f})/\varepsilon\big)\big].\nonumber\\
\end{eqnarray}
\begin{lemma}\label{justific-asympt-solution} Let the assumptions (A1)-(A7) be satisfied. Then, there exists a number $\varepsilon_{0} > 0$ such that, for all $t \in [0,t_{f}]$ and $\varepsilon \in (0,\varepsilon_{0}]$,
the following inequalities are valid:
\begin{eqnarray}\label{ineq-zero-order-solution}\big\|z_{j}(t,\varepsilon) - z_{j}^{1}(t,\varepsilon)\big\| \le c_{0}\varepsilon^{2},\  \  \  \ \big\|\lambda_{\rm uj}(t,\varepsilon) - \lambda_{\rm uj}^{1}(t,\varepsilon)\big\|\le c_{0}\varepsilon^{2},\  \  \  \ j = 1,2, \nonumber\\
\big\|\lambda_{\rm vj}(t,\varepsilon) - \lambda_{\rm vj}^{1}(t,\varepsilon)\big\|\le c_{0}\varepsilon^{2},\  \  \  \
\big\|\mu_{j}(t,\varepsilon) - \mu_{j}^{1}(t,\varepsilon)\big\| \le c_{0}\varepsilon^{2},\  \  \  \ j = 1,2,\nonumber\\
\end{eqnarray}
where 
$${\rm col}\big(z_{1}(t,\varepsilon),z_{2}(t,\varepsilon),\lambda_{\rm u 1}(t,\varepsilon),\lambda_{\rm u 2}(t,\varepsilon),\lambda_{\rm v1}(t,\varepsilon),\lambda_{\rm v2}(t,\varepsilon),\mu_{1}(t,\varepsilon),\mu_{2}(t,\varepsilon)\big)$$ is the unique solution of the singularly perturbed boundary-value problem (\ref{equiv-system}),(\ref{equiv-cond});
$c_{0} > 0$ is some constant independent of $\varepsilon$.
\end{lemma}
\begin{proof} First of all, let us observe the following. \\
{\bf O1.} The differential system (\ref{equiv-system}) is a linear homogeneous one. Furthermore, due to Proposition \ref{init-problem-transf} and the equations (\ref{matr-S_u-S_v-Suv}),(\ref{matr-A-D_u-S_u-blocks}), all the matrix-valued coefficients in this system are continuously differentiable in the interval $[0,t_{f}]$. Moreover, the matrix-valued coefficients $D_{\rm v,1}(t)$ and $D_{\rm v,2}(t)$ are twice continuously differentiable in this interval.\\ 
{\bf O2.} The boundary conditions (\ref{equiv-cond}) are explicit and linear. Moreover, the boundary conditions for the fast state variables $z_{2}(t)$ and $\lambda_{\rm v 2}(t)$, and $\mu_{2}(t)$ and $\lambda_{\rm u2}(t)$ are given at the opposite ends of the interval $[0,t_{f}]$. Namely, the conditions for $z_{2}(t)$ and $\mu_{2}(t)$ are given at $t = 0$, while the conditions for $\lambda_{\rm v 2}(t)$ and $\lambda_{\rm u2}(t)$ are given at $t = t_{f}$.\\

Let us consider the matrix of the coefficients for the fast state vector ${\rm col}\big(z_{2}(t),\allowbreak\lambda_{\rm u 2}(t),\lambda_{\rm v2}(t),\mu_{2}(t)\big)$ in the set of the fast modes of the system (\ref{equiv-system}), i.e., the matrix
\begin{equation}\label{matr-Phi} \Phi(t,\varepsilon) \stackrel{\triangle}{=}\left(\begin{array}{l}\  \  \ \varepsilon A_{4}(t)\  \  \ - \varepsilon^{2}S_{\rm u,3}(t)\  \  \ - I_{s}\  \  \  \  \  \  \  \  \  \ 0\\
- D_{\rm u,3}(t)\  \  \  - \varepsilon A_{4}^{T}(t)\  \  \  \  \  \  \  \  \ 0\  \  \  \  \  \  \  \  \ D_{\rm v,2}(t)\\
- D_{\rm v,2}(t)\  \  \  \  \  \  \  \  \ 0\  \  \  \  \  \  \ - \varepsilon A_{4}^{T}(t)\  \  \  \  \  \  \ 0\\
\  \  \  \ 0\  \  \  \  \  \  \  \  \  \  \  \  \  \
I_{s}\  \  \  \  \  \ - \varepsilon^{4}G_{\rm u,v}(t)\  \  \ \varepsilon A_{4}(t)
\end{array}\right),\  \  \  t \in [0,t_{f}],\  \  \  \varepsilon > 0.\end{equation}

The matrix-valued function $\Phi(t,\varepsilon)$ is of the dimension $4s \times 4s$ and it is continuous with respect to $(t,\varepsilon) \in [0,t_{f}]\times(- \infty,+ \infty)$. By $\nu_{i}(t,\varepsilon)$, $(i = 1,...,2s,2s + 1,...,4s)$, we denote the eigenvalues of the matrix $\Phi(t,\varepsilon)$.
We are going to show that, for all $t \in [0,t_{f}]$ and all sufficiently small $\varepsilon > 0$, $2s$ eigenvalues $\nu_{i}(t,\varepsilon)$, $(i = 1,...,2s)$ of this matrix satisfy the inequality
\begin{equation}\label{Re-mu_i>0} {\rm Re}\big(\nu_{i}(t,\varepsilon)\big) \ge \alpha,\  \  \  \ i = 1,...,2s,\end{equation}
while the other eigenvalues satisfy the inequality
\begin{equation}\label{Re-mu_j<0} {\rm Re}\big(\nu_{j}(t,\varepsilon)\big) \le - \alpha,\  \  \  \ j = 2s+1,...,4s,\end{equation}
where $\alpha > 0$ is some constant independent of $\varepsilon$.

Setting $\varepsilon = 0$ in (\ref{matr-Phi}), we obtain
\begin{eqnarray}\label{Phi-eps=0} \Phi(t,0) = \left(\begin{array}{l}\  \  \  \ 0\  \  \  \  \  \  \  \  \ 0\  \  \ - I_{s}\  \  \  \  \  \  \  \  \  \  0\\
- D_{\rm u,3}(t)\  \  \ 0\  \  \  \  \  \  \ 0 \  \  \  \  \  \  \  \ D_{\rm v,2}(t)\\
- D_{\rm v,2}(t)\  \  \ 0\  \  \  \  \  \  \ 0\  \  \  \  \  \  \  \  \  \  \ 0\\
\  \  \  \ 0\  \  \  \  \  \  \  \ I_{s}\  \  \  \  \  \  \ 0\  \  \  \  \  \  \  \  \  \  \ 0
\end{array}\right),\  \  \  \ t \in [0,t_{f}].\end{eqnarray}

Permutation of the first and third rows, as well as the second and fourth rows of the matrix $\Phi(t,0)$, yield the matrix
\begin{eqnarray}\label{matr-Psi}\Psi(t) = \left(\begin{array}{l}- D_{\rm v,2}(t)\  \  \ 0\  \  \  \  \  \  \ 0\  \  \  \  \  \  \ 0\\
\  \  \  \ 0\  \  \  \  \  \  \  \ I_{s}\  \  \  \  \  \  \ 0\  \  \  \  \  \  \ 0\\
\  \  \  \ 0\  \  \  \  \  \  \  \  \ 0\  \  \   - I_{s}\  \  \  \  \  \  0\\
- D_{\rm u,3}(t)\  \  \ 0\  \  \  \  \  \  \ 0\  \  \  \  \ D_{\rm v,2}(t)\end{array}\right),\  \  \  \ t \in [0,t_{f}].\nonumber
\end{eqnarray}

The determinants of the matrices $\Phi(t,0)$ and $\Psi(t)$ are equal to each other and $\det\big(\Phi(t,0)\big) = \det\big(\Psi(t)\big) = \big[\det\big(D_{\rm v,2}(t)\big)\big]^{2}$, $t \in [0,t_{f}]$. Since $D_{\rm v,2}(t)$ is a positive definite matrix for all $t \in [0,t_{f}]$, then $\det\big(\Phi(t,0)\big) \neq 0$ for all $t \in [0,t_{f}]$. Therefore, all the eigenvalues of the matrix $\Phi(t,0)$ are non-zero for all $t \in [0,t_{f}]$. Using this observation and the Frobenius formula for the determinant of a block-form matrix \cite{Frobenius-formula}, we obtain that any eigenvalue $\kappa(t)$, $t \in [0,t_{f}]$ of the matrix $\Phi(t,0)$ satisfies the equation
\begin{eqnarray}\label{eq-for-kappa}\Big[\det\Big(D_{\rm v,2}(t) - \big(\kappa(t)\big)^{2}I_{s}\Big)\Big]^{2} = 0,\  \  \  \ t \in [0,t_{f}].\nonumber
\end{eqnarray}

Solving this equation, we obtain $2s$ positive and $2s$ negative eigenvalues of the matrix $\Phi(t,0)$. Namely, the positive eigenvalues are $\kappa_{1}(t)$,...,$\kappa_{s}(t)$,$\kappa_{1}(t)$,...,$\kappa_{s}(t)$, while the negative eigenvalues are $-\kappa_{1}(t)$,...,$-\kappa_{s}(t)$,$-\kappa_{1}(t)$,...,$-\kappa_{s}(t)$, where $\kappa_{k}(t) > 0$, $(k = 1,...,s)$, $t \in [0,t_{f}]$ are the eigenvalues of the symmetric positive definite matrix $D_{\rm v,2}^{1/2}(t)$.

Since the matrix-valued function $D_{\rm v,2}^{1/2}(t)$ is continuous in the interval $[0,t_{f}]$, then, due to the results of \cite{Coddington-Levinson,Sibuya}, the functions $\kappa_{k}(t)$, $(k = 1,...,s)$ also are continuous for $t \in [0,t_{f}]$. This feature, along with the aforementioned positiveness of these functions, directly yields the existence of a positive number $\alpha$ such that the following inequality is satisfied:
\begin{equation}\label{ineq-kappa_k} \kappa_{k}(t) \ge 2\alpha,\  \  \  \ t\in [0,t_{f}],\  \  \  \ k = 1,...,s.\end{equation}

Furthermore, since, for any $t \in [0,t_{f}]$, the matrix-valued function $\Phi(t,\varepsilon)$ is continuous with respect to $\varepsilon \in (- \infty,+ \infty)$, then
\begin{eqnarray}\label{nu(t,0)=kappa(t)}  \nu_{i_{1}}(t,0) = \kappa_{i_{1}}(t),\  \  \ i_{1} = 1,...,s,\  \  \  \ \nu_{i_{2}}(t,0) = \kappa_{i_{2}-s}(t),\  \  \ i_{2} = s+1,...,2s,\nonumber\\ t \in [0,t_{f}],\nonumber\\
\nu_{i_{1}+2s}(t,0) = -\kappa_{i_{1}}(t),\  \  \ i_{1} = 1,...,s,\  \  \  \ \nu_{i_{2}+2s}(t,0) = - \kappa_{i_{2}-s}(t),\  \  \ i_{2} = s+1,...,2s,\nonumber\\ t \in [0,t_{f}],\nonumber
\end{eqnarray}
meaning, along with (\ref{ineq-kappa_k}),
\begin{equation}\label{ineq-nu_i} \nu_{i}(t,0) \ge 2\alpha,\  \  \  \ \nu_{i+2s}(t,0) \le -  2\alpha,\  \  \ t \in [0,t_{f}],\  \  \ i = 1,...,2s.\end{equation}

The matrix-valued function $\Phi(t,\varepsilon)$, being continuous with respect to $(t,\varepsilon) \in [0,t_{f}]\times(-\infty,+\infty)$, is continuous in $\varepsilon \in(-\infty,+\infty)$ uniformly with respect to $t \in [0,t_{f}]$. Therefore, by virtue of \cite{Coddington-Levinson,Sibuya}, the eigenvalues $\nu_{i}(t,\varepsilon)$, $(i = 1,...,2s)$ and $\nu_{j}(t,\varepsilon)$, $(j = 2s + 1,...,4s)$ are continuous in $\varepsilon \in(-\infty,+\infty)$ uniformly with respect to $t \in [0,t_{f}]$. This observation, along with the inequalities in (\ref{ineq-nu_i}), yields the existence of a number $\bar{\varepsilon} > 0$ such that, for all $t \in [0,t_{f}]$ and all $\varepsilon \in (0,\bar{\varepsilon}]$, the inequalities (\ref{Re-mu_i>0}) and (\ref{Re-mu_j<0}) are satisfied. The validity of these inequalities means that the singularly perturbed boundary-value problem (\ref{equiv-system})-(\ref{equiv-cond}) is of the conditionally stable type (\cite{Vasil'eva-Butuzov-Kalachev,Esipova-Diff-Eq}). Taking into account this feature of the problem (\ref{equiv-system})-(\ref{equiv-cond}), as well as the aforementioned observations O1 and O2, and using the results of the work \cite{Esipova-Diff-Eq}, we immediately obtain the statements of the lemma.
\end{proof}

\section{Main results}

\subsection{Asymptotic approximation of the Stackelberg solution to the game (\ref{new-eq})-(\ref{new-perf-ind-2})}

For the sake of further analysis, we partition the matrix $B_{\rm u}(t)$ into blocks as:
\begin{eqnarray}\label{B_u-blocks} B_{\rm u}(t) = \left(\begin{array}{c}B_{\rm u,1}(t)\\ B_{\rm u,2}(t)\end{array}\right),\  \  \  \ t \in [0,t_{f}],\end{eqnarray}
where the matrices $B_{\rm u,1}(t)$ and $B_{\rm u,2}(t)$ are of the dimensions $(n - s)\times r$ and $s\times r$, respectively.

Using the equations (\ref{new-matr-B_uv}),(\ref{u*-v*}),(\ref{block-form-solution}),(\ref{B_u-blocks}), we can rewrite the components of the solution to the Stackelberg game (\ref{new-eq})--(\ref{new-perf-ind-2}) in the form
\begin{eqnarray}\label{u*-v*-new-form} u^{*}(t,\varepsilon) = - B_{\rm u,1}^{T}(t)\lambda_{\rm u1}(t,\varepsilon) - \varepsilon B_{\rm u,2}^{T}(t)\lambda_{\rm u2}(t,\varepsilon),\  \  \ v^{*}(t,\varepsilon) = - \frac{1}{\varepsilon}\lambda_{\rm v2}(t,\varepsilon),\nonumber\\ t \in [0,t_{f}],\ \ \varepsilon > 0.\nonumber\\
\end{eqnarray}

Along with the optimal controls (\ref{u*-v*-new-form}), we consider the following controls of the leader and the follower, respectively:
\begin{eqnarray}\label{hat-u-v} \widehat{u}(t,\varepsilon) = - B_{\rm u,1}^{T}(t)\lambda_{\rm u1}^{1}(t,\varepsilon) - \varepsilon B_{\rm u,2}^{T}(t)\lambda_{\rm u2}^{1}(t,\varepsilon),\  \  \ \widehat{v}(t,\varepsilon) = - \frac{1}{\varepsilon}\lambda_{\rm v2}^{1}(t,\varepsilon),\nonumber\\ t \in [0,t_{f}],\ \ \varepsilon > 0.\nonumber\\
\end{eqnarray}

\begin{remark}\label{obtaining-approximate-hat-u-v} The control $\widehat{u}(t,\varepsilon)$ is obtained from the control $u^{*}(t,\varepsilon)$ by replacing in the later $\lambda_{\rm u1}(t,\varepsilon)$ and $\lambda_{\rm u2}(t,\varepsilon)$ with their asymptotic approximations $\lambda_{\rm u1}^{1}(t,\varepsilon)$ and $\lambda_{\rm u2}^{1}(t,\varepsilon)$, respectively (see the equation (\ref{zero-order-asympt-solution-new-form})). Similarly, the control $\widehat{v}(t,\varepsilon)$ is obtained from the control $v^{*}(t,\varepsilon)$ by replacing in the later $\lambda_{\rm v2}(t,\varepsilon)$ with its asymptotic approximation $\lambda_{\rm v2}^{1}(t,\varepsilon)$.
\end{remark}

\begin{theorem}\label{asymp-u*-v*}
Let the assumptions (A1)-(A7) be valid. Then, for all $t \in [0,t_{f}]$ and all $\varepsilon \in (0,\varepsilon_{0}]$, the open-loop Stackelberg solution $\big(u^{*}(t,\varepsilon),v^{*}(t,\varepsilon)\big)$ to the game (\ref{new-eq})-(\ref{new-perf-ind-2}) satisfies the inequalities
\begin{eqnarray}\label{ineq-u*-v*}
\big\|u^{*}(t,\varepsilon) - \widehat{u}(t,\varepsilon)\big\| \le c_{1}\varepsilon^{2} + c_{2}\varepsilon^{3},\nonumber\\
\big\|v^{*}(t,\varepsilon) -  \widehat{v}(t,\varepsilon)\big\| \le c_{0}\varepsilon,\nonumber\\
\end{eqnarray}
where the numbers $\varepsilon_{0} > 0$ and $c_{0} > 0$ are introduced in Lemma \ref{justific-asympt-solution};
\begin{eqnarray}c_{1} = c_{0}\Big(\max_{t \in [0,t_{f}]}\big\|B_{\rm u,1}(t)\big\|\Big),\  \  \  \ c_{2} = c_{0}\Big(\max_{t \in [0,t_{f}]}\big\|B_{\rm u,2}(t)\big\|\Big). \nonumber\\
\end{eqnarray}
\end{theorem}
\begin{proof} The statement of the theorem directly follows from the equations (\ref{u*-v*-new-form}),(\ref{hat-u-v}) and Lemma
\ref{justific-asympt-solution}.
\end{proof}

Let $\widehat{J}_{\rm u}(\varepsilon)$ and $\widehat{J}_{\rm v}(\varepsilon)$ be the values of the functionals  $J_{\rm u}(u,v)$ and $J_{\rm v}(u,v)$, respectively, generated by the pair of the controls $\big(\widehat{u}(t,\varepsilon),\widehat{v}(t,\varepsilon)\big)$ in the game (\ref{new-eq})-(\ref{new-perf-ind-2}).

\begin{theorem}\label{suboptimal-solution-hat} Let the assumptions (A1)-(A7) be valid. Then, for all $\varepsilon \in (0,\varepsilon_{0}]$, the following inequalities are satisfied:
\begin{eqnarray}\label{ineq-hat-for-J_u*-J_v*} \big|J_{\rm u}^{*}(\varepsilon) - \widehat{J}_{\rm u}(\varepsilon)\big| \le \widehat{b}\varepsilon,\nonumber\\
\big|J_{\rm v}^{*}(\varepsilon) - \widehat{J}_{\rm v}(\varepsilon)\big| \le \widehat{b}\varepsilon,\nonumber\\
\end{eqnarray}
where $\widehat{b} > 0$ is some constant independent of $\varepsilon$.
\end{theorem}
\begin{proof}Let us start with the proof of the first inequality in (\ref{ineq-hat-for-J_u*-J_v*}).

Substituting the open-loop Stackelberg solution $\big(u^{*}(t,\varepsilon),v^{*}(t,\varepsilon)\big)$ into the system (\ref{new-eq}) instead of $\big(u(t),v(t)\big)$, we obtain for all $\varepsilon \in (0,\varepsilon_{0}]$
\begin{eqnarray}\label{syst-for-u^*-v^*}
\frac{dz(t)}{dt}=A(t)z(t)+B_{\rm u}(t)u^{*}(t,\varepsilon)+B_{\rm v}(t)v^{*}(t,\varepsilon),\  \  \ t\in [0,t_{f}],\  \  \ z(0)=z_{0}.
\end{eqnarray}
The solution $z^{*}(t,\varepsilon)$ of this system is the Stackelberg optimal trajectory of the game (\ref{new-eq})-(\ref{new-perf-ind-2}). Due to Proposition \ref{Stackelberg-solution} and the equivalence of the boundary-value problems (\ref{initial-bound-value-pr}) and (\ref{equiv-system})-(\ref{equiv-cond}), we directly have
\begin{eqnarray}\label{z^*=z}z^{*}(t,\varepsilon) = z(t,\varepsilon),\  \  \  \ t \in [0,t_{f}],\  \  \  \ \varepsilon \in (0,\varepsilon_{0}],\end{eqnarray}
where $z(t,\varepsilon) = {\rm col}\big(z_{1}(t,\varepsilon), z_{2}(t,\varepsilon)\big)$, and $z_{1}(t,\varepsilon)$ and $z_{2}(t,\varepsilon)$ are the corresponding components of the solution to the boundary-value problem (\ref{equiv-system})-(\ref{equiv-cond}).

Thus,
\begin{eqnarray}\label{J_u^*-hat}J_{\rm u}^{*}(\varepsilon) = J_{\rm u}\big(u^{*}(t,\varepsilon),v^{*}(t,\varepsilon)\big) = \frac{1}{2}\int_{0}^{t_{f}}\big[\big(z^{*}(t,\varepsilon)\big)^{T}D_{\rm u}(t)z^{*}(t,\varepsilon)\nonumber\\
+ \big(u^{*}(t,\varepsilon)\big)^{T}u^{*}(t,\varepsilon) +  \varepsilon^{2}\big(v^{*}(t,\varepsilon)\big)^{T}G_{\rm u,v}(t)v^{*}(t,\varepsilon)\big]dt,\  \  \
 \  \varepsilon \in (0,\varepsilon_{0}].\nonumber\\
\end{eqnarray}
Now, let us substitute the pair of the controls $\big(\widehat{u}(t,\varepsilon),\widehat{v}(t,\varepsilon)\big)$ into the system (\ref{new-eq}) instead of $\big(u(t),v(t)\big)$. This substitution yields for all $\varepsilon \in (0,\varepsilon_{0}]$
\begin{eqnarray}\label{syst-for-hat-u-v}
\frac{dz(t)}{dt}=A(t)z(t)+B_{\rm u}(t)\widehat{u}(t,\varepsilon)+B_{\rm v}(t)\widehat{v}(t,\varepsilon),\  \  \ t\in [0,t_{f}],\  \  \ z(0)=z_{0}.
\end{eqnarray}
Let $z = \widehat{z}(t,\varepsilon)$, $t \in [0,t_{f}]$ denote the solution of this system.

Substituting  $\widehat{z}(t,\varepsilon)$, $\widehat{u}(t,\varepsilon)$ and $\widehat{v}(t,\varepsilon)$ into the functional (\ref{new-perf-ind-1}) instead of $z(t)$, $u(t)$ and $v(t)$, respectively, yields the following expression for $\widehat{J}_{\rm u}(\varepsilon)$:
\begin{eqnarray}\label{widehat-J_u}\widehat{J}_{\rm u}(\varepsilon) = J_{\rm u}\big(\widehat{u}(t,\varepsilon),\widehat{v}(t,\varepsilon)\big) = \frac{1}{2}\int_{0}^{t_{f}}\big[\widehat{z}^{T}(t,\varepsilon)D_{\rm u}(t)\widehat{z}(t,\varepsilon)\nonumber\\
+ \widehat{u}^{T}(t,\varepsilon)\widehat{u}(t,\varepsilon) +  \varepsilon^{2}\widehat {v}^{T}(t,\varepsilon)G_{\rm u,v}(t)\widehat{v}(t,\varepsilon)\big]dt,\  \  \
 \  \varepsilon \in (0,\varepsilon_{0}].\nonumber\\
\end{eqnarray}

Let us denote
\begin{eqnarray}\label{hat-delta}\widehat{\delta}_{z}(t,\varepsilon) \stackrel{\triangle}{=} z^{*}(t,\varepsilon) - \widehat{z}(t,\varepsilon),\  \  \  \ t \in [0,t_{f}],\  \  \  \ \varepsilon \in (0,\varepsilon_{0}].\end{eqnarray}

Using the equations (\ref{syst-for-u^*-v^*}) and (\ref{syst-for-hat-u-v}), we directly obtain that $\widehat{\delta}_{z}(t,\varepsilon)$ is the unique solution of the system
\begin{eqnarray}\label{system-hat-delta_z}\frac{d\widehat{\delta}_{z}(t,\varepsilon)}{dt} = A(t)\widehat{\delta}_{z}(t,\varepsilon) + B_{\rm u}(t)\big(u^{*}(t,\varepsilon) - \widehat{u}(t,\varepsilon)\big) + B_{\rm v}(t)\big(v^{*}(t,\varepsilon) - \widehat{v}(t,\varepsilon)\big),\nonumber\\
t \in [0,t_{f}],\  \  \  \  \  \ \widehat{\delta}_{z}(0,\varepsilon) = 0.\nonumber\\
\end{eqnarray}

Let $\Gamma(t,\sigma)$, $0 \le \sigma \le t \le t_{f}$ be the fundamental matrix solution of the homogeneous system corresponding to the differential equation in (\ref{system-hat-delta_z}). Hence, we can express $\widehat{\delta}_{z}(t,\varepsilon)$ as follows:
\begin{eqnarray}\label{hat-delta_z} \widehat{\delta}_{z}(t,\varepsilon) = \int_{0}^{t}\Gamma(t,\sigma)\big[B_{\rm u}(\sigma)\big(u^{*}(\sigma,\varepsilon) - \widehat{u}(\sigma,\varepsilon)\big) + B_{\rm v}(\sigma)\big(v^{*}(\sigma,\varepsilon) - \widehat{v}(\sigma,\varepsilon)\big)\big]d\sigma,\nonumber\\
t \in [0,t_{f}],\  \  \  \  \ \varepsilon \in (0,\varepsilon_{0}].\nonumber
\end{eqnarray}
This equation, along with the inequalities in (\ref{ineq-u*-v*}), yields immediately
\begin{eqnarray}\label{ineq-hat-delta_z} \big\|\widehat{\delta}_{z}(t,\varepsilon)\big\| \le c_{3}\varepsilon,\  \  \  \ t \in [0,t_{f}],\  \  \  \ \varepsilon \in (0,\varepsilon_{0}],\end{eqnarray}
where $c_{3} > 0$ is some constant independent of $\varepsilon$.

Now, the equations (\ref{J_u^*-hat}),(\ref{widehat-J_u}),(\ref{hat-delta}) and the inequalities (\ref{ineq-u*-v*}),(\ref{ineq-hat-delta_z}) yield the first inequality in (\ref{ineq-hat-for-J_u*-J_v*}). The second inequality in (\ref{ineq-hat-for-J_u*-J_v*}) is proven similarly. This completes the proof of the theorem.
\end{proof}

\begin{remark}\label{suboptim-solution-hat} Due to Theorem \ref{suboptimal-solution-hat}, the
pair of the controls $\big(\widehat{u}(t),\widehat{v}(t,\varepsilon)\big)$ is an asymptotically suboptimal Stackelberg solution in the game (\ref{new-eq})-(\ref{new-perf-ind-2}). To obtain this solution, one has to obtain all the addends appearing in the expressions for $\lambda_{\rm u 1}^{1}(t,\varepsilon)$, $\lambda_{\rm u 2}^{1}(t,\varepsilon)$, $\lambda_{\rm v 2}^{1}(t,\varepsilon)$ (see the equation (\ref{zero-order-asympt-solution-new-form})).
\end{remark}

Consider the values
\begin{eqnarray}\label{hat-J} \bar{J}_{\rm u,0} = \frac{1}{2}\int_{0}^{t_{f}}\big[\bar{z}_{0}^{T}(t)D_{\rm u}(t)\bar{z}_{0}(t) + \bar{\lambda}_{\rm u1,0}^{T}(t)S_{\rm u,1}(t)\bar{\lambda}_{\rm u1,0}(t)\big]dt,\nonumber\\
\bar{J}_{\rm v,0} = \frac{1}{2}\int_{0}^{t_{f}}\big[\bar{z}_{0}^{T}(t)D_{\rm v}(t)\bar{z}_{0}(t) + \bar{\lambda}_{\rm u1,0}^{T}(t)B_{\rm u,1}(t)G_{\rm v,u}(t)B_{\rm u,1}^{T}(t)\bar{\lambda}_{\rm u1,0}(t)]dt,\nonumber\\
\end{eqnarray}
where $\bar{z}_{0}(t) = \left(\begin{array}{c}\bar{z}_{1,0}(t)\\ \bar{z}_{2,0}(t)\end{array}\right)$,
$\bar{z}_{1,0}(t)$ and $\bar{\lambda}_{\rm u1,0}(t)$ are the corresponding components of the solution to the boundary-value problem (\ref{set-four-diff-eq-outer-sol}),(\ref{outer-solution-cond}); $\bar{z}_{2,0}(t)$ is given by the equation (\ref{bar-z_2}).

It is seen that $\bar{J}_{\rm u,0}$ and $\bar{J}_{\rm v,0}$ are independent of $\varepsilon$.

\begin{corollary}\label{asympt-J_u-J_v} Let the assumptions (A1)-(A7) be valid. Then, for all $\varepsilon \in (0,\varepsilon_{0}]$, the Stackelberg optimal values of the cost functionals $J_{\rm u}^{*}(\varepsilon)$ and $J_{\rm v}^{*}(\varepsilon)$ in the game (\ref{new-eq})-(\ref{new-perf-ind-2}) satisfy the inequalities
\begin{eqnarray}\label{ineq-for-J_u*-J_v*} \big|J_{\rm u}^{*}(\varepsilon) - \bar{J}_{\rm u,0}\big| \le \bar{b}\varepsilon,\nonumber\\
\big|J_{\rm v}^{*}(\varepsilon) - \bar{J}_{\rm v,0}\big| \le \bar{b}\varepsilon,\nonumber\\
\end{eqnarray}
where $\bar{b} > 0$ is some constant independent of $\varepsilon$.
\end{corollary}
\begin{proof} Let us start with the proof of the first inequality in (\ref{ineq-for-J_u*-J_v*}).

Substituting the expressions for the $u^{*}(t,\varepsilon)$, $v^{*}(t,\varepsilon)$ (see the equation (\ref{u*-v*-new-form})) into the equation (\ref{J_u^*-hat}) and using the expression for the matrix $S_{\rm u}(t)$, the block form of this matrix and the block form of the matrix $B_{\rm u}(t)$ (see the equations (\ref{matr-S_u-S_v-Suv}),(\ref{matr-A-D_u-S_u-blocks}) and (\ref{B_u-blocks})), we obtain after a routine rearrangement the following expression for $J_{\rm u}^{*}(\varepsilon)$:
\begin{eqnarray}\label{expression-J_u*} J_{\rm u}^{*}(\varepsilon)
= \frac{1}{2}\int_{0}^{t_{f}}\big[\big(z^{*}(t,\varepsilon)\big)^{T}D_{\rm u}(t)z^{*}(t,\varepsilon) + \lambda_{\rm u 1}^{T}(t,\varepsilon)S_{\rm u,1}(t)\lambda_{\rm u 1}(t,\varepsilon)\nonumber\\ + 2\varepsilon\lambda_{\rm u1}^{T}(t,\varepsilon)S_{\rm u,2}(t)\lambda_{\rm u2}(t,\varepsilon)
+ \varepsilon^{2}\lambda_{\rm u2}^{T}(t)S_{\rm u,2}(t)\lambda_{\rm u2}(t)\nonumber\\ + \lambda_{\rm v2}^{T}(t)G_{\rm u,v}(t)\lambda_{\rm v2}(t)\big]dt,\
 \  \ \varepsilon \in (0,\varepsilon_{0}],\nonumber\\
 \end{eqnarray}
where, due to (\ref{z^*=z}), the vector $z^{*}(t,\varepsilon) = {\rm col}\big(z_{1}(t,\varepsilon),z_{2}(t,\varepsilon)\big)$,  and its entries $z_{1}(t,\varepsilon)$ and $z_{2}(t,\varepsilon)$, as well as the vectors $\lambda_{\rm u 1}(t,\varepsilon)$, $\lambda_{\rm u 2}(t,\varepsilon)$, $\lambda_{\rm v2}(t,\varepsilon)$, are the corresponding components of the solution to boundary-value problem (\ref{equiv-system})-(\ref{equiv-cond}).

Now, using the equation (\ref{expression-J_u*}), as well as Lemma \ref{justific-asympt-solution}, the inequalities in (\ref{ineq-for-bound-correc-3}), (\ref{ineq-bound-terms-1-1}), (\ref{ineq-for-bound-correc-3-1}), (\ref{ineq-for-bound-correc-4-1}) and the first equation in (\ref{hat-J}), yields immediately the first inequality in (\ref{ineq-for-J_u*-J_v*}).

The second inequality in (\ref{ineq-for-J_u*-J_v*}) is proven quite similarly to the proof of the first inequality. This completes the proof of the theorem.
\end{proof}

\begin{remark}\label{hat-J_u=bar-J_u} Using Proposition \ref{solution-auxiliary-ocp}, Remark \ref{new-form-bar-J_u^*} and the first equation in (\ref{hat-J}), we can conclude that $\bar{J}_{\rm u,0} = \bar{J}_{\rm u}^{*}$.
\end{remark}

\subsection{Simplified asymptotically suboptimal Stackelberg solution to the game (\ref{new-eq})-(\ref{new-perf-ind-2})}

Consider the following controls for the leader and the follower, respectively:
\begin{eqnarray}\label{suboptim-controls}\widetilde{u}(t) = - B_{\rm u,1}^{T}(t)\bar{\lambda}_{\rm u1,0}(t),\  \  \  \ t \in [0,t_{f}],\nonumber\\
{\rm and}\nonumber\\
\widetilde{v}(t,\varepsilon) = - \frac{1}{\varepsilon}\lambda_{\rm v2,0}^{0}(t/\varepsilon) - \bar{\lambda}_{\rm v2,1}(t),\  \  \  \ t \in [0,t_{f}],\  \  \  \ \varepsilon \in (0,\varepsilon_{0}].\nonumber\\
\end{eqnarray}

\begin{remark}\label{comparison-tilde-hat} Comparison of the equations (\ref{hat-u-v}) and (\ref{suboptim-controls}) with each other directly shows that the pair of the controls $\big(\widetilde{u}(t),\widetilde{v}(t,\varepsilon)\big)$ is much simpler than the pair of the controls $\big(\widehat{u}(t,\varepsilon),\widehat{v}(t,\varepsilon)\big)$. Namely, in contrast with $\widehat{u}(t,\varepsilon)$, the control $\widetilde{u}(t)$ is independent of $\varepsilon$. Moreover, the calculation of $\widetilde{u}(t)$ requires only the obtaining $\bar{\lambda}_{\rm u 1,0}(t)$, which is much simpler than the calculation of $\widehat{u}(t,\varepsilon)$. The calculation of $\widetilde{v}(t,\varepsilon)$ requires only the obtaining $\lambda_{\rm
 v 2,0}^{0}(t/\varepsilon)$ and $\bar{\lambda}_{\rm v 2,1}(t)$, which is much simpler than the calculation of $\widehat{v}(t,\varepsilon)$.
\end{remark}

Let $\widetilde{J}_{\rm u}(\varepsilon)$ and $\widetilde{J}_{\rm v}(\varepsilon)$ be the values of the functionals  $J_{\rm u}(u,v)$ and $J_{\rm v}(u,v)$, respectively, generated by the pair of the controls $\big(\widetilde{u}(t),\widetilde{v}(t,\varepsilon)\big)$ in the game (\ref{new-eq})-(\ref{new-perf-ind-2}).

\begin{theorem}\label{suboptimal-solution} Let the assumptions (A1)-(A7) be valid. Then, for all $\varepsilon \in (0,\varepsilon_{0}]$, the following inequalities are satisfied:
\begin{eqnarray}\label{ineq-2-for-J_u*-J_v*} \big|J_{\rm u}^{*}(\varepsilon) - \widetilde{J}_{\rm u}(\varepsilon)\big| \le \widetilde{b}\varepsilon,\nonumber\\
\big|J_{\rm v}^{*}(\varepsilon) - \widetilde{J}_{\rm v}(\varepsilon)\big| \le \widetilde{b}\varepsilon,\nonumber\\
\end{eqnarray}
where $\widetilde{b} > 0$ is some constant independent of $\varepsilon$.
\end{theorem}
\begin{proof}Let us start with the proof of the first inequality in (\ref{ineq-2-for-J_u*-J_v*}).

Let us substitute the pair of the controls $\big(\widetilde{u}(t),\widetilde{v}(t,\varepsilon)\big)$ into the system (\ref{new-eq}) instead of $\big(u(t),v(t)\big)$. This substitution yields for all $\varepsilon \in (0,\varepsilon_{0}]$
\begin{eqnarray}\label{syst-for-tilde-u-v}
\frac{dz(t)}{dt}=A(t)z(t)+B_{\rm u}(t)\widetilde{u}(t)+B_{\rm v}(t)\widetilde{v}(t,\varepsilon),\  \  \ t\in [0,t_{f}],\  \  \ z(0)=z_{0}.
\end{eqnarray}
Let $z = \widetilde{z}(t,\varepsilon)$, $t \in [0,t_{f}]$ denote the solution of this system.

Substitution of  $\widetilde{z}(t,\varepsilon)$, $\widetilde{u}(t)$ and $\widetilde{v}(t,\varepsilon)$ into the functional (\ref{new-perf-ind-1}) instead of $z(t)$, $u(t)$ and $v(t)$, respectively, yields after a routine rearrangement the following expression for $\widetilde{J}_{\rm u}(\varepsilon)$:
\begin{eqnarray}\label{widetilde-J_u}\widetilde{J}_{\rm u}(\varepsilon) = J_{\rm u}\big(\widetilde{u}(t,\varepsilon),\widetilde{v}(t)\big) = \frac{1}{2}\int_{0}^{t_{f}}\big[\widetilde{z}^{T}(t,\varepsilon)D_{\rm u}(t)\widetilde{z}(t,\varepsilon)\nonumber\\
+ \bar{\lambda}_{\rm u1,0}^{T}(t)S_{\rm u,1}(t)\bar{\lambda}_{\rm u1,0}(t) + \big(\lambda_{\rm v 2,0}^{0}(t/\varepsilon) + \varepsilon \bar{\lambda}_{\rm v2,1}(t)\big)^{T}\big(\lambda_{\rm v 2,0}^{0}(t/\varepsilon) + \varepsilon \bar{\lambda}_{\rm v2,1}(t)\big)\big]dt,\nonumber\\  \varepsilon \in (0,\varepsilon_{0}],\nonumber\\
\end{eqnarray}
where $S_{\rm u,1}(t)$ is the upper left-hand block of the matrix $S_{\rm u}(t)$ given in
(\ref{matr-A-D_u-S_u-blocks}).

Let us denote
\begin{eqnarray}\label{tilde-delta}\widetilde{\delta}_{z}(t,\varepsilon) \stackrel{\triangle}{=} z^{*}(t,\varepsilon) - \widetilde{z}(t,\varepsilon),\  \  \  \ t \in [0,t_{f}],\  \  \  \ \varepsilon \in (0,\varepsilon_{0}],\end{eqnarray}
where $z^{*}(t,\varepsilon)$ is the solution of the system (\ref{syst-for-u^*-v^*}).

Using the equations (\ref{syst-for-u^*-v^*}) and (\ref{syst-for-tilde-u-v}), we directly obtain that $\widetilde{\delta}_{z}(t,\varepsilon)$ is the unique solution of the system
\begin{eqnarray}\label{system-tilde-delta_z}\frac{d\widetilde{\delta}_{z}(t,\varepsilon)}{dt} = A(t)\widetilde{\delta}_{z}(t,\varepsilon) + B_{\rm u}(t)\big(u^{*}(t,\varepsilon) - \widetilde{u}(t)\big) + B_{\rm v}(t)\big(v^{*}(t,\varepsilon) - \widetilde{v}(t,\varepsilon)\big),\nonumber\\
t \in [0,t_{f}],\  \  \  \  \  \ \widetilde{\delta}_{z}(0,\varepsilon) = 0.\nonumber\\
\end{eqnarray}

Let us estimate the differences $\big(u^{*}(t,\varepsilon) - \widetilde{u}(t)\big)$ and $\big(v^{*}(t,\varepsilon) - \widetilde{v}(t,\varepsilon)\big)$.

Using the expressions for $\lambda_{\rm u1}^{1}(t,\varepsilon)$, $\widehat{u}(t,\varepsilon)$ and  $\widetilde{u}(t)$ (see the equations (\ref{zero-order-asympt-solution-new-form}),(\ref{hat-u-v}) and (\ref{suboptim-controls})), as well as the second inequality in (\ref{ineq-bound-terms-1-1}) and the first inequality in (\ref{ineq-u*-v*}), we have
\begin{eqnarray}\label{estim-first-difference}\big\|u^{*}(t,\varepsilon) - \widetilde{u}(t)\big\| \le c_{4}\varepsilon,\  \  \  \ t \in [0,t_{f}],\  \  \  \ \varepsilon \in (0,\varepsilon_{0}],\end{eqnarray}
where $c_{4} > 0$ is some constant independent of $\varepsilon$.

Due to Lemma \ref{justific-asympt-solution} and the equations (\ref{ineq-for-bound-correc-3-1}),(\ref{ineq-for-bound-correc-4-1}), we have
\begin{eqnarray}\label{auxiliary-est-1}\big\|\lambda_{\rm v2}(t,\varepsilon) - \lambda_{\rm v2,0}^{0}(t/\varepsilon) - \varepsilon\bar{\lambda}_{\rm v2,1}(t)\big\| \le c_{0}\varepsilon^{2}\nonumber\\
+ b_{\rm v}^{0}\varepsilon\exp\big(-0.5\beta^{0}t/\varepsilon\big) + b_{\rm v}^{f}\varepsilon\exp\big(\beta^{f}(t - t_{f})/\varepsilon\big),\  \  \  \
t \in [0,t_{f}],\  \  \ \varepsilon \in (0,\varepsilon_{0}].\nonumber\\
\end{eqnarray}

Using this inequality and the expressions for $v^{*}(t,\varepsilon)$ and $\widetilde{v}(t,\varepsilon)$ (see the equations (\ref{u*-v*-new-form}) and (\ref{suboptim-controls})) directly yields
\begin{eqnarray}\label{estim-second-difference}\big\|v^{*}(t,\varepsilon) - \widetilde{v}(t,\varepsilon)\big\| \le c_{0}\varepsilon + b_{\rm v}^{0}\exp\big(-0.5\beta^{0}t/\varepsilon\big) + b_{\rm v}^{f}\exp\big(\beta^{f}(t - t_{f})/\varepsilon\big),\nonumber\\
t \in [0,t_{f}],\  \  \  \  \ \varepsilon \in (0,\varepsilon_{0}].\nonumber\\
\end{eqnarray}

Using the fundamental matrix solution $\Gamma(t,\sigma)$, $0 \le \sigma \le t \le t_{f}$ of the homogeneous system corresponding to the differential equation in (\ref{system-tilde-delta_z}) yields the following expression for $\widetilde{\delta}_{z}(t,\varepsilon)$:
\begin{eqnarray}\label{tilde-delta_z} \widetilde{\delta}_{z}(t,\varepsilon) = \int_{0}^{t}\Gamma(t,\sigma)\big[B_{\rm u}(\sigma)\big(u^{*}(\sigma,\varepsilon) - \widetilde{u}(\sigma)\big) + B_{\rm v}(\sigma)\big(v^{*}(\sigma,\varepsilon) - \widetilde{v}(\sigma,\varepsilon)\big)]d\sigma,\nonumber\\
t \in [0,t_{f}],\  \  \  \  \ \varepsilon \in (0,\varepsilon_{0}].\nonumber
\end{eqnarray}

This equation, along with the inequalities (\ref{estim-first-difference}),(\ref{estim-second-difference}), yields immediately
\begin{eqnarray}\label{ineq-tilde-delta_z} \big\|\widetilde{\delta}_{z}(t,\varepsilon)\big\| \le c_{5}\varepsilon,\  \  \  \ t \in [0,t_{f}],\  \  \  \ \varepsilon \in (0,\varepsilon_{0}],\end{eqnarray}
where $c_{5} > 0$ is some constant independent of $\varepsilon$.

Furthermore, by virtue of the equation for $\lambda_{\rm u 1}^{1}(t,\varepsilon)$ and the inequality for $\lambda_{\rm u 1}(t,\varepsilon)$ (see the equations (\ref{zero-order-asympt-solution-new-form}) and (\ref{ineq-zero-order-solution})), we have the inequality
\begin{eqnarray}\label{ineq-lambda_v}\big\|\lambda_{\rm u1}(t,\varepsilon) - \bar{\lambda}_{\rm u1,0}(t)\big\| \le c_{6}\varepsilon,\  \  \  \ t \in [0,t_{f}],\  \  \  \ \varepsilon \in (0,\varepsilon_{0}],\end{eqnarray}
where $c_{6} > 0$ is some constant independent of $\varepsilon$.

Now, the equations (\ref{expression-J_u*}),(\ref{widetilde-J_u}),(\ref{tilde-delta}) and the inequalities (\ref{auxiliary-est-1}),(\ref{ineq-tilde-delta_z}),(\ref{ineq-lambda_v}) yield the first inequality in (\ref{ineq-2-for-J_u*-J_v*}). The second inequality in (\ref{ineq-2-for-J_u*-J_v*}) is proven similarly. This completes the proof of the theorem.
\end{proof}

\begin{remark}\label{suboptim-solution} Due to Theorem \ref{suboptimal-solution}, the
pair of the controls $\big(\widetilde{u}(t),\widetilde{v}(t,\varepsilon)\big)$ is an asymptotically suboptimal Stackelberg solution in the game (\ref{new-eq})-(\ref{new-perf-ind-2}). Moreover, due to Remark \ref{comparison-tilde-hat}, this suboptimal solution is much simpler than the previously derived in Theorem \ref{suboptimal-solution-hat} suboptimal solution $\big(\widehat{u}(t,\varepsilon),\widehat{v}(t,\varepsilon)\big)$.
\end{remark}

\section{Numerical example} Consider the supply chain model of \cite{Turetsky-Glizer-IMA-OR-2025} which structure is similar to \cite{Assarzadegan_etal}.

\subsection{Model description}

This supply chain consists of a single Manufacturer who produces the National Brand (NB) product and, \textit{additionally, invests in innovation}, and a single Retailer that sells the Manufacturer's product and, \textit{additionally, promotes her own substitute Store Brand (SB)}. Both National Brand and Store Brand have their "badwill" levels, $B_{\rm NB}$ and $B_{SB}$, respectively.  The badwill represents the opposite of the goodwill (see, e.g. \cite{Assarzadegan_etal}), i.e., the variable indicating the customer's loyalty to the brand and the brand reputation. The badwill dynamics is described by the linear differential equations for $t\in [0, t_f]$:
\begin{eqnarray}\label{badwill-dynamics}
\dot B_{\rm NB}=a_1  B_{\rm NB} - b_1 u+c_1 v,\nonumber\\
\dot B_{\rm SB}=a_2  B_{\rm SB} + b_2 u-c_2 v,\nonumber\\
\end{eqnarray}
where $a_1$, $a_2$, $b_1$, $b_2$, $c_1$, $c_2$ are constant positive coefficients, $u=u(t)\in \mathbb{R}$ is the Manufacturer's control (the NB innovation investment), $v=v(t)\in \mathbb{R}$ is the Retailer's control (the SB promotion investment).
The equations in (\ref{badwill-dynamics}) imply that by investing in innovation, the Manufacturer decreases the NB badwill and, consequently, increases the SB badwill. And, vice versa, by investing in promoting SB, the Retailer increases the NB badwill and decreases the SB badwill. In the absence of the additional investments, both badwills have the tendency to increase.

For
\begin{equation}\label{modelmatrices}Z(t)=\left(B_{\rm NB}(t), B_{\rm SB}(t)\right)^T,\ {\mathcal A}(t)\equiv\begin{bmatrix}a_1 & 0\\a_2 & 0\end{bmatrix},\ \ {\mathcal B}_{\rm u}(t)\equiv\begin{bmatrix}-b_1\\b_2\end{bmatrix},\ \ {\mathcal B}_{\rm v}(t)\equiv\begin{bmatrix}c_1\\-c_2\end{bmatrix},
\end{equation}
the system (\ref{badwill-dynamics}) admits a form (\ref{or-eq}).

Both players minimize their storage cost which is assumed to be proportional to the squared badwill, as well as their control effort, yielding the cost functionals
\begin{equation}\label{JM}J_{M}=\frac{1}{2}\int_0^{t_f}\left[k_M B_{\rm NB}^2(t)+ \alpha u^2(t)\right]dt,\end{equation}
\begin{equation}\label{JR}J_{R}=\frac{1}{2}\int_0^{t_f}\left[k_R B_{\rm SB}^2(t)+\beta v^2(t)\right]dt,\end{equation} of the Manufacturer and the Retailer, respectively, where $k_M, k_R, \alpha, \beta>0$ are constant coefficients.

The competition in this supply chain is modeled by a two-player finite horizon
linear-quadratic Stackelberg differential game with the Manufacturer as a leader and the Retailer as a follower. The game is formulated in open-loop controls.
In this example, we assume a short-horizon competition (a small enough value of $t_f$). It means that the Manufacturer as a leader is \textit{committed} to the once chosen optimal open-loop strategy and does not recalculate it dynamically.

\subsection{The case of the Retailer's cheap control}
Let $\alpha = 1$ and $\beta=\varepsilon^2\ll 1$. Then, for
\begin{eqnarray}\label{costmatrices}{\mathcal D}_{\rm u}(t)\equiv\begin{bmatrix}k_M & 0\\0 & 0\end{bmatrix},\ \  {\mathcal D}_{\rm v}(t)\equiv\begin{bmatrix}0 & 0\\0 & k_R\end{bmatrix},\nonumber\\ G_{\rm u,u}(t)\equiv 1,\  G_{\rm v,v}(t)\equiv 1,\  G_{\rm u,v}(t)=G_{\rm v,u}(t)\equiv 0,\end{eqnarray} the game (\ref{badwill-dynamics}) -- (\ref{JR}) admits the form of  (\ref{or-eq}) -- (\ref{perf-ind-2}) (see Remark \ref{mathcal-G_uu-G_vv=I}).

It is readily verified that the assumptions (A1) -- (A6) are satisfied. In order to carry out the transform (\ref{state-transform}), we chose a complement matrix to ${\mathcal B}_{\rm v}(t)$ as ${\mathcal B}_{c}(t)\equiv\begin{bmatrix}c_2,c_1\end{bmatrix}^T$. By simple algebra, this yields
\begin{equation}\label{LvRv}{\mathcal L}_{\rm v}(t)\equiv \begin{bmatrix}\dfrac{c_1^2}{c_2}+c_2\\0\end{bmatrix},\ \ {\mathcal R}_{\rm v}(t)\equiv \begin{bmatrix}\dfrac{c_1^2}{c_2}+c_2 & c_1\\0 & -c_2\end{bmatrix}.\end{equation}
In this example, the boundary-value problem (\ref{equiv-system}) -- (\ref{equiv-cond}) is written down for the scalar variables $z_i(t,\varepsilon)$, $\lambda_{\rm ui}(t,\varepsilon)$, $\lambda_{\rm vi}(t,\varepsilon)$, $\mu_i(t,\varepsilon)$, $(i=1,2)$. Due to  (\ref{LvRv}), the coefficients of (\ref{equiv-system}) are
\begin{eqnarray}\label{BVP-coefs}A_1(t)\equiv \dfrac{a_1c_2+a_2c_1}{c_2},\  \  \ A_2(t)\equiv \dfrac{c_1(a_1c_2+a_2c_1)}{c_1^2+c_2^2},\nonumber\\  A_3(t)\equiv -\dfrac{a_2(c_1^2+c_2^2)}{c_2^2},\  \  \ A_4(t)\equiv -\dfrac{a_2c_1}{c_2},  \nonumber\\
B_{\rm u,1}(t) \equiv \dfrac{b_{2}c_{1} - b_{1}c_{2}}{c_{1}^{2} + c_{2}^{2}},\  \  \ B_{u,2}(t) \equiv - \dfrac{b_{2}}{c_{2}},\nonumber\\
S_{\rm u,1}(t)\equiv\left(\dfrac{b_1c_2-b_2c_1}{c_1^2+c_2^2}\right)^2,\  \ S_{\rm u,2}(t)\equiv \dfrac{b_2(b_1c_2-b_2c_1)}{c_2(c_1^2+c_2^2)},\  \ S_{\rm u,3}(t)\equiv \dfrac{b_2^2}{c_2^2},\nonumber\\
S_{\rm v,1}(t)=S_{\rm v,2}(t)\equiv 0,\  \  \ S_{\rm v,3}(t)\equiv \dfrac{1}{\varepsilon^2},\nonumber\\
D_{\rm u,1}(t)\equiv k_M\left(\dfrac{c_1^2+c_2^2}{c_2}\right)^2,\  \ D_{\rm u,2}(t)\equiv k_M\dfrac{c_1(c_1^2+c_2^2)}{c_2},\  \ D_{\rm u,3}(t)\equiv k_M c_1^2,\nonumber\\
D_{\rm v,1}(t) \equiv 0,\  \  \ D_{\rm v,2}(t)\equiv k_R c_2^2.\nonumber\\
\end{eqnarray}
Using the equations (\ref{u*-v*-new-form}) and (\ref{BVP-coefs}), we obtain the optimal open-loop Stackelberg controls  $u^*(t,\varepsilon)$ and $v^*(t,\varepsilon)$ in this example as:
\begin{eqnarray}\label{optimalstrategies-example}u^*(t,\varepsilon)=\dfrac{b_1c_2-b_2c_1}{c_1^2+c_2^2}\lambda_{\rm u1}(t,\varepsilon)+\dfrac{\varepsilon b_2}{c_2}\lambda_{\rm u2}(t,\varepsilon),\ \
v^*(t,\varepsilon) = - \frac{1}{\varepsilon}\lambda_{\rm v2}(t,\varepsilon),\nonumber\\ t\in [0, t_f],\ \ \varepsilon>0,
\end{eqnarray}
where $\lambda_{\rm u1}(t,\varepsilon)$, $\lambda_{\rm u2}(t,\varepsilon)$, and $\lambda_{\rm v2}(t,\varepsilon)$ are the components of the solution of the boundary-value problem (\ref{equiv-system}) -- (\ref{equiv-cond}) with the coefficients (\ref{BVP-coefs}).

The respective asymptotically suboptimal controls (\ref{hat-u-v}) and simplified asymptotically suboptimal controls (\ref{suboptim-controls}) in this example are
\begin{eqnarray}\label{approximatestrategies-example}\widehat{u}(t,\varepsilon)=\dfrac{b_1c_2-b_2c_1}{c_1^2+c_2^2}\lambda^1_{\rm u1}(t,\varepsilon)+\dfrac{\varepsilon b_2}{c_2}\lambda^1_{\rm u2}(t,\varepsilon),\  \  \
\widehat{v}(t,\varepsilon) = - \frac{1}{\varepsilon}\lambda^1_{\rm v2}(t,\varepsilon),\nonumber\\ t\in [0, t_f],\ \ \varepsilon>0,
\end{eqnarray}
\begin{eqnarray}\label{suboptimalstrategies-example}\widetilde{u}(t)=\dfrac{b_1c_2-b_2c_1}{c_1^2+c_2^2}\bar{\lambda}_{\rm u1,0}(t),\  \  \
\widetilde{v}(t,\varepsilon) = - \frac{1}{\varepsilon}\lambda_{\rm v2,0}^{0}(t/\varepsilon) - \bar{\lambda}_{\rm v2,1}(t),\nonumber\\ t\in [0, t_f],\ \ \varepsilon>0,
\end{eqnarray}
where $\lambda^1_{\rm u1}(t,\varepsilon)$, $\lambda^1_{\rm u2}(t,\varepsilon)$, $\lambda^1_{\rm v1}(t,\varepsilon)$, $\bar{\lambda}_{\rm u1,0}(t)$, $\lambda_{\rm v2,0}^{0}(t/\varepsilon)$, and $\bar{\lambda}_{\rm v2,1}(t)$ are the components of the first-order asymptotic solution (\ref{first-order-asympt-solution}) of the boundary-value problem (\ref{equiv-system}) -- (\ref{equiv-cond}) with the coefficients (\ref{BVP-coefs}).
In the considered example, the components of this asymptotic solution are obtained numerically.

In Fig. \ref{deltauv-fig}, the control errors
\begin{equation}\label{Deltauv}\Delta \widehat{u}(\varepsilon)=\max_{t \in [0,t_{f}]}\|u^*(t,\varepsilon)-\widehat{u}(t,\varepsilon)\|,\  \  \  \  \Delta \widehat{v}(\varepsilon)=\max_{t \in [0,t_{f}]}\|v^*(t,\varepsilon)-\widehat{v}(t,\varepsilon)\|,
\end{equation}
are depicted for $a_1=0.1$, $a_2=0.2$, $b_1=0.5$, $b_2=0.4$, $c_1=0.2$, $c_2=0.6$, $k_M=1$, $k_R=5$, $t_f=2$. This figure illustrates the statement of Theorem \ref{asymp-u*-v*} that the asymptotically suboptimal control of a non-cheap control player (the Manufacturer) guarantees a higher order accuracy (with respect to $\varepsilon$) than that of a cheap control player (the Retailer).

In Fig. \ref{deltaJuJv-fig}, the corresponding suboptimal cost errors
\begin{equation}\label{Deltahat_JuJv}\Delta \widehat{J}_k(\varepsilon)=|J_k^*(\varepsilon)-\widehat{J}_k(\varepsilon)|,\  \  \ k={\rm R}, {\rm M},\textbf{}
\end{equation}
\begin{equation}\label{Deltatilde_JuJv}\Delta \widetilde{J}_k(\varepsilon)=|J_k^*(\varepsilon)-\widetilde{J}_k(\varepsilon)|,\  \  \ k={\rm R}, {\rm M},
\end{equation}
are depicted, illustrating the estimates established in Theorems \ref{suboptimal-solution-hat} and \ref{suboptimal-solution}. It is seen that the asymptotically suboptimal Stackelberg solution $\big(\widehat{u}(t,\varepsilon),\widehat{v}(t,\varepsilon)\big)$, requiring more calculation work, provides a better approximation of the optimal values of the cost functionals than the simplified suboptimal solution $\big(\widetilde{u}(t,\varepsilon),\widetilde{v}(t,\varepsilon)\big)$, requiring less calculation work.

\begin{figure}[hbt!]
    \centering
    \includegraphics[width=0.98\linewidth]{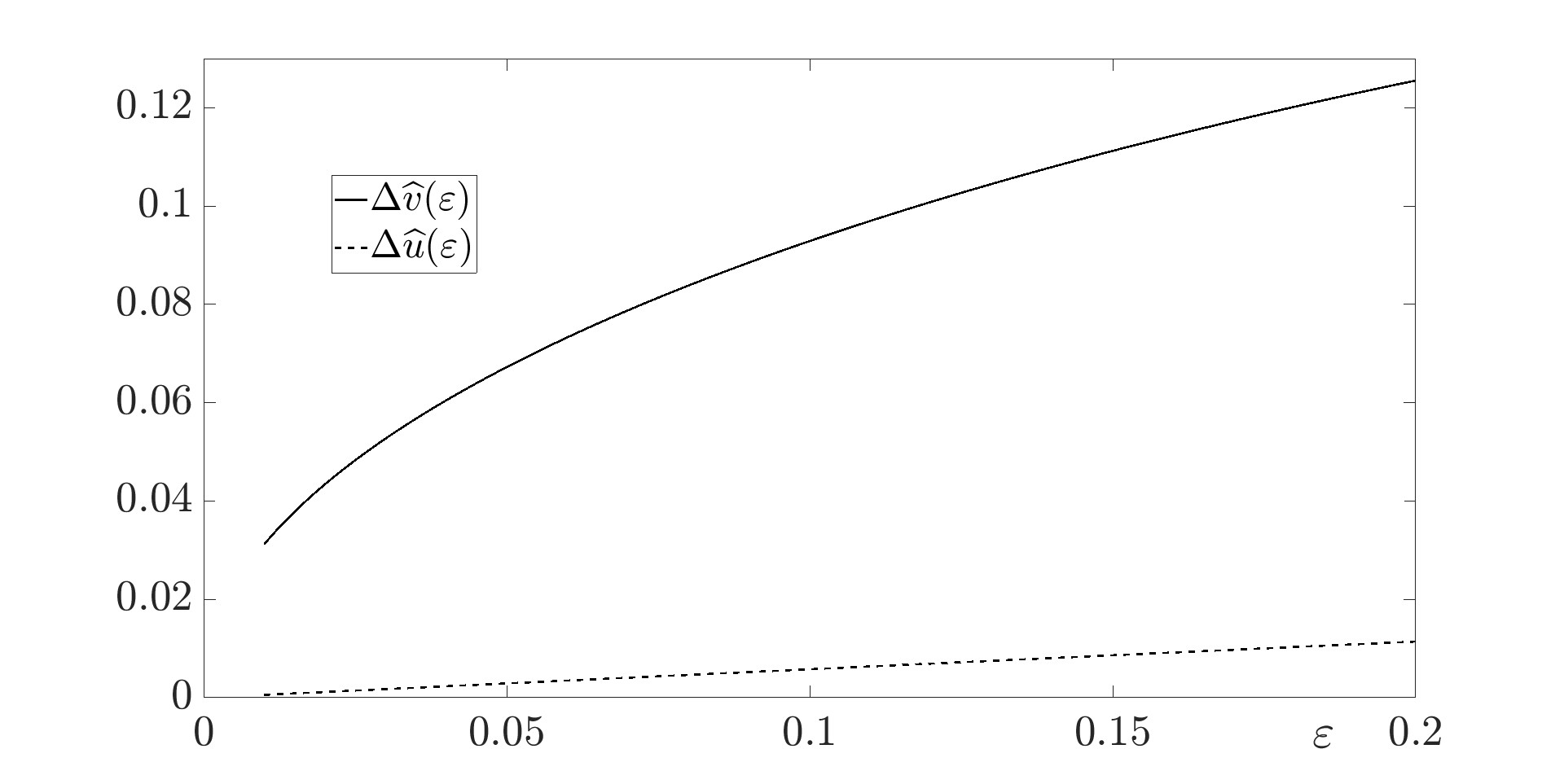}
    \caption{Asymptotically suboptimal controls accuracy}
    \label{deltauv-fig}
\end{figure}

\begin{figure}[hbt!]
    \centering
    \includegraphics[width=0.98\linewidth]{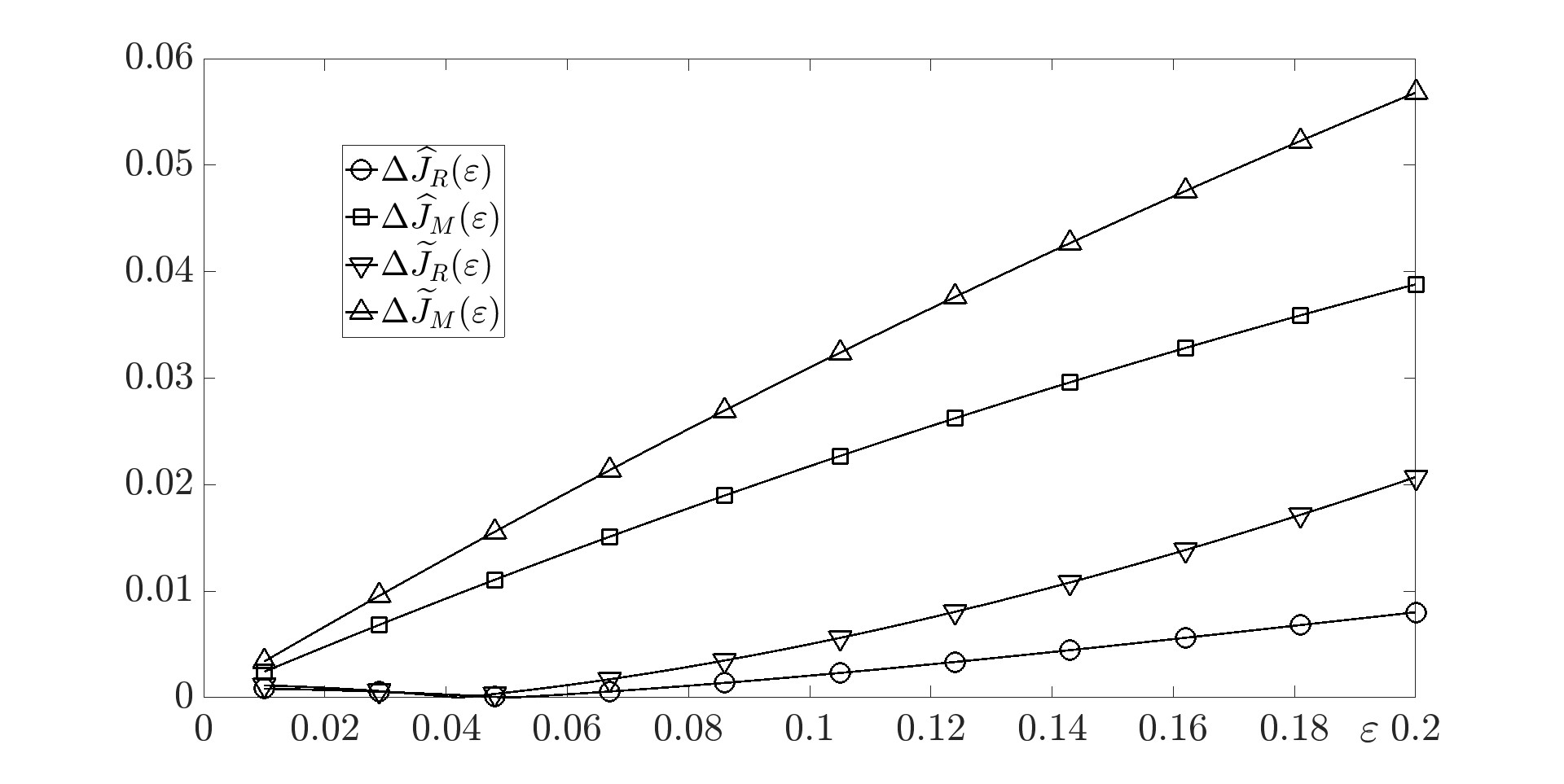}
    \caption{Cost errors by using $\left(\widehat{u}(t,\varepsilon),\widehat{v}(t,\varepsilon)\right)$ and $\left(\widetilde{u}(t,\varepsilon),\widetilde{v}(t,\varepsilon)\right)$}
    \label{deltaJuJv-fig}
\end{figure}

\begin{table}[hbt!]
\bgroup
\def\arraystretch{1.3}
\begin{tabular}{|c|c|c|c|c|}
\hline
$\varepsilon$ &  $\delta \widehat{J}_{M}$ & $\delta \widehat{J}_{R}$ & $\delta \widetilde{J}_{M}$ & $\delta \widetilde{J}_{R}$\\
\hline
 $0.2$ & $1.4368$ & $0.4867$ & $2.1032$ & $1.2561$\\
 \hline
$0.1$ &  $0.7868$ & $0.2630$ & $1.1215$ & $0.6390$\\
\hline
$0.05$ & $0.4114$ & $0.0040$ & $0.5791$ & $0.1245$\\
 \hline
 $0.01$ & $0.0857$ & $1.038$  & $0.1195$ & $1.4233$\\
\hline
\end{tabular}
\egroup
\caption{Relative cost errors}\label{relativeerrors-table}
\end{table}

In Table \ref{relativeerrors-table},
the relative errors
\begin{equation}\label{deltahatJuJv}\delta \widehat{J}_k(\varepsilon)=\dfrac{\Delta \widehat{J}_k(\varepsilon)}{J_k^*(\varepsilon)}\cdot 100\%,\  \  \ k={\rm M}, {\rm R},
\end{equation}
\begin{equation}\label{deltatildeJuJv}\delta \widetilde{J}_k(\varepsilon)=\dfrac{\Delta \widetilde{J}_k(\varepsilon)}{J_k^*(\varepsilon)}\cdot 100\%,\  \  \ k={\rm M}, {\rm R},
\end{equation} are presented.
It is seen that these errors are small and mainly decrease. The increasing values of
$\delta \widehat{J}_R(\varepsilon)$ and $\delta \widetilde{J}_R(\varepsilon)$ for $\varepsilon=0.01$ are explained by small values of the cost $J_R^*$.

Fig. \ref{BSBopt_eps-fig} shows the impact of $\varepsilon$ on the optimal SB badwill $B_{\rm SB}^*(t,\varepsilon)$. The effect is quite natural: for smaller values of $\varepsilon$ (meaning a cheaper SB promotion investment by the Retailer), the badwill faster tends to zero faster.
\begin{figure}[hbt!]
    \centering
    \includegraphics[width=0.98\linewidth]{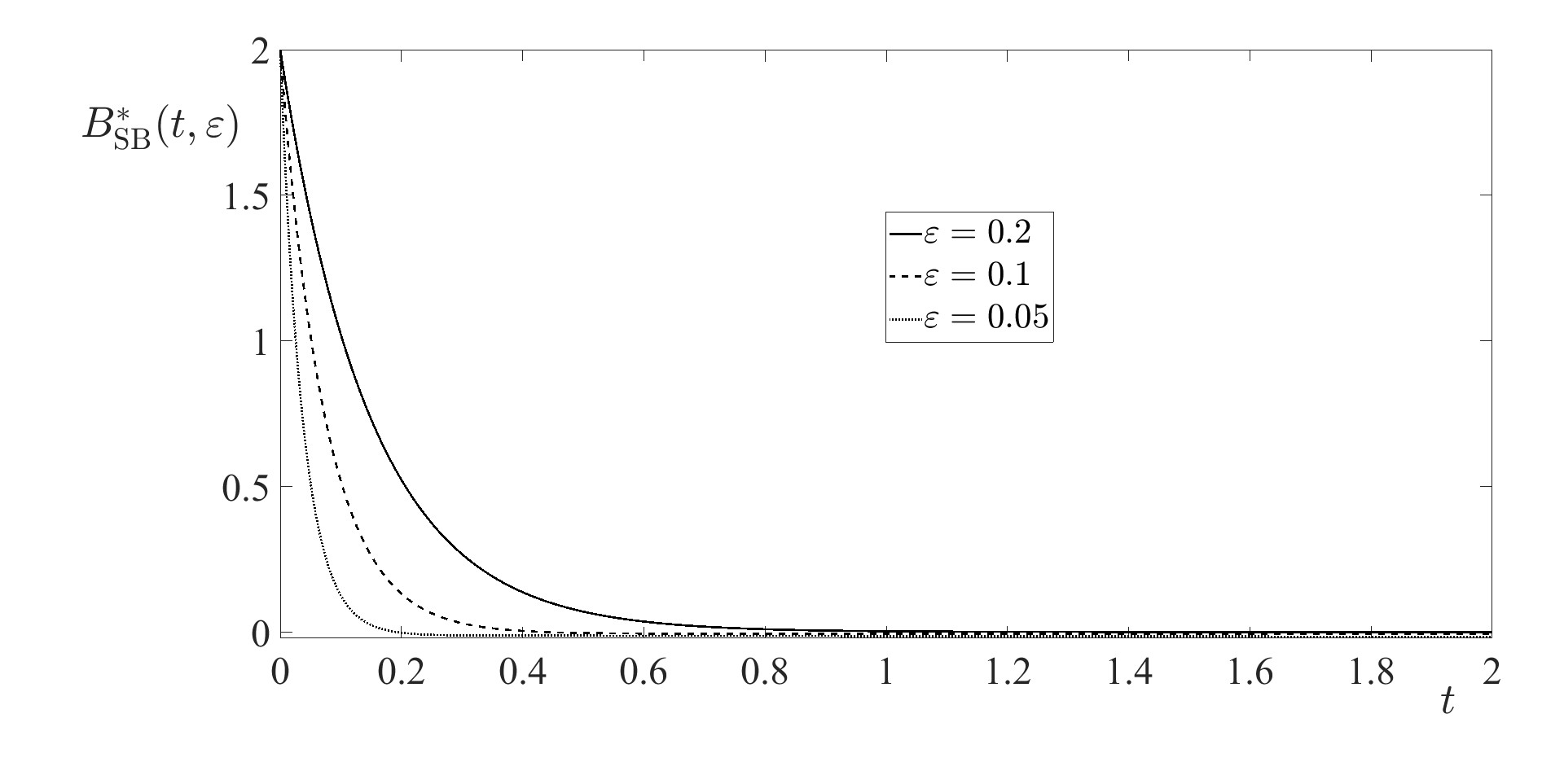}
    \caption{Impact of $\varepsilon$ on the optimal SB badwill}
    \label{BSBopt_eps-fig}
\end{figure}

\subsection{Comparison with the case of the Manufacturer's cheap control}

In this section, along with the case of the Retailer's cheap control ($\alpha=1$, $\beta=\varepsilon^2$), we consider the case of the Manufacturer's cheap control ($\alpha=\varepsilon^2$, $\beta=1$). In this case, the Manufacturer's investment in the innovation is cheap. The asymptotic game solution in the case of the leader's cheap control is established in \cite{Glizer-Turetsky-Axioms-2024}. The baseline case for the comparison is $\alpha=\beta=1$, where the controls of both players are not cheap.

Let $J_k^*(\alpha,\beta)$, ($k={\rm R}, {\rm M}$), denote the optimal costs for given values of the penalty coefficients $\alpha$ and $\beta$. Thus, $J_k^*(1,\varepsilon^2)$, $J_k^*(\varepsilon^2, 1)$, and $J_k^*(1,1)$, ($k={\rm R}, {\rm M}$),  are the optimal costs in the case of the Retailer's cheap control, in the case of the Manufacturer's cheap control, and in the case with no cheap controls, respectively.

Fig. \ref{cost_improvement-fig} demonstrates that both players benefit from using cheap control, i.e., enjoy the decrease of their cost in comparison with the no-cheap control case. It is seen that the values $J_M^*(\varepsilon^2,1)$ and $J_R^*(1,\varepsilon^2)$ are smaller than the respective baseline no-cheap control values $J_M^*(1,1)$ and $J_R^*(1,1)$ in the whole range of $\varepsilon\in [0.01,0.2]$. The opposite effect is that both players suffer from not using the cheap control, i.e., their cost of a non-cheap control increases in comparison with the case where the advisory refrains from using the cheap control. This effect is clearly demonstrated in Fig. \ref{cost_deterioration-fig}: the values  $J_M^*(1,\varepsilon^2)$ and $J_R^*(\varepsilon^2,1)$ are larger than the respective baseline no-cheap control values $J_M^*(1,1)$ and $J_R^*(1,1)$ in the whole range of $\varepsilon\in [0.01,0.2]$. In this figure, one observes that the deterioration of the Retailer's (the follower's) cost is considerably higher than that of the Manufacturer (the leader).

\begin{figure}[hbt!]
    \centering
    \includegraphics[width=0.98\linewidth]{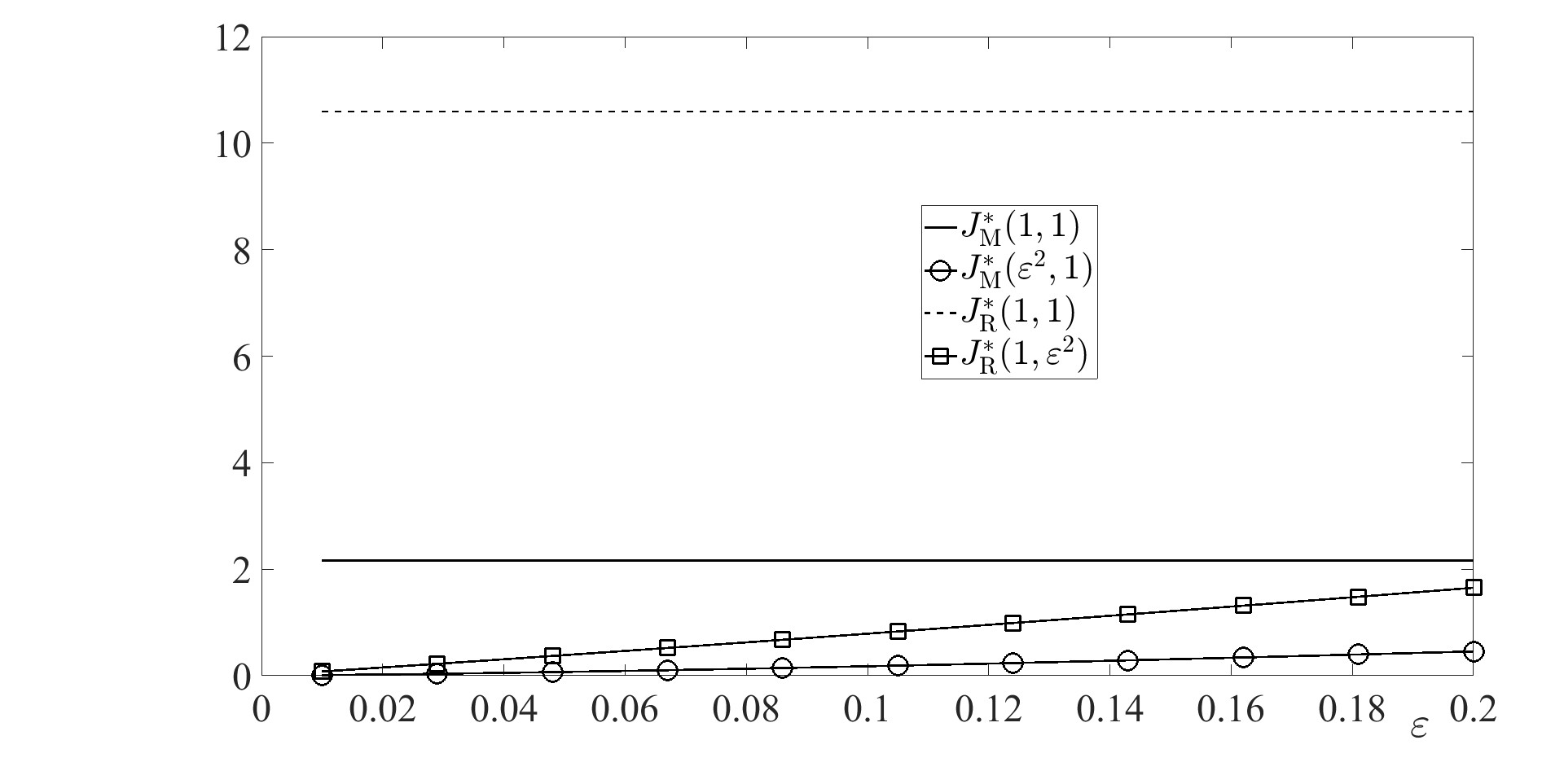}
    \caption{Cost improvement of the player with cheap control}
    \label{cost_improvement-fig}
\end{figure}

\begin{figure}[hbt!]
    \centering
    \includegraphics[width=0.98\linewidth]{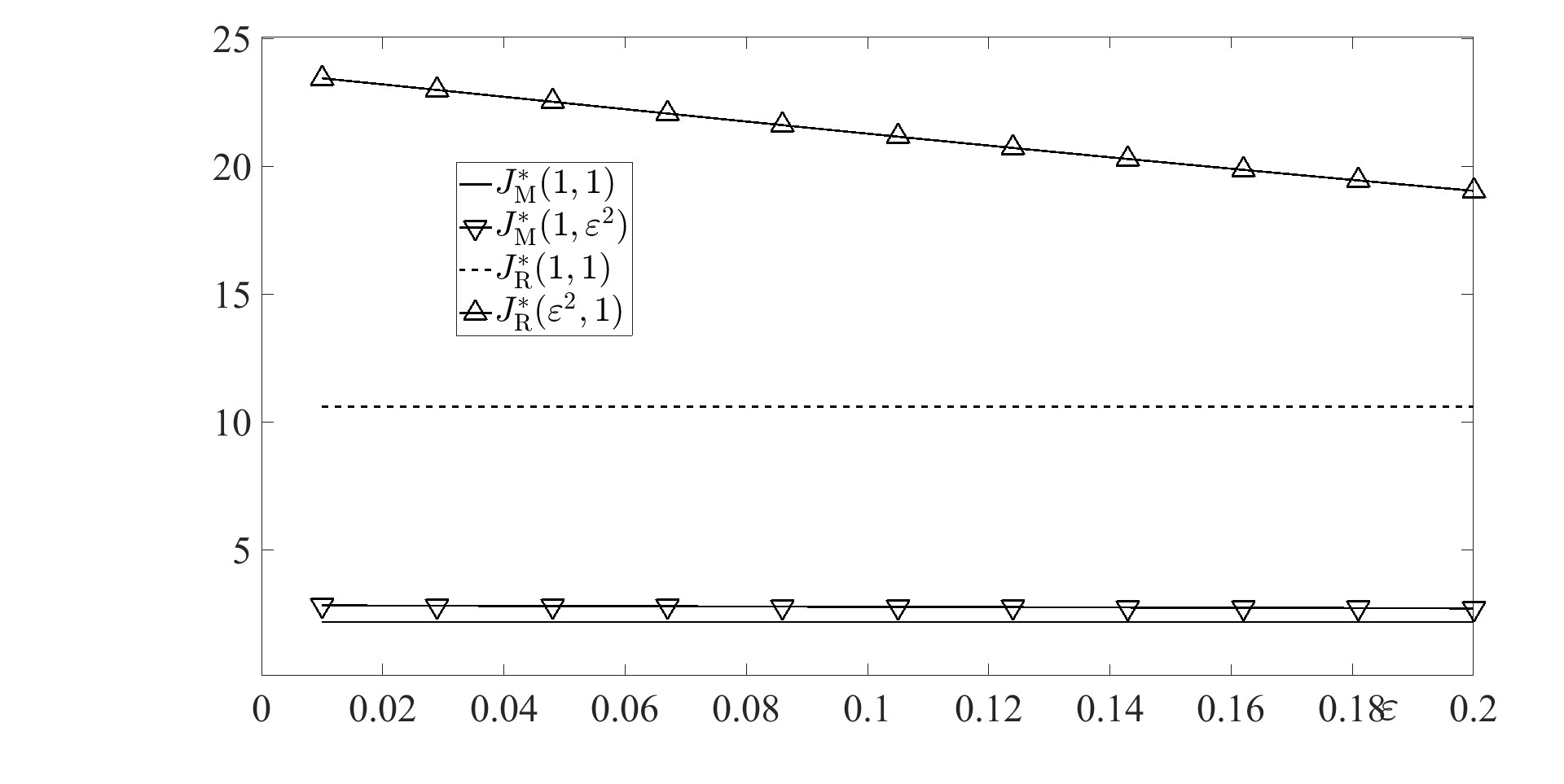}
    \caption{Cost deterioration of the player with non-cheap control}
    \label{cost_deterioration-fig}
\end{figure}
\newpage
Let
\begin{eqnarray}\label{relative_improvements}
\delta J_{\rm M}^{\rm M}(\varepsilon)=\dfrac{J_{\rm M}^*(1,1) - J_{\rm M}^*(\varepsilon^2,1)}{J_{\rm M}^*(1,1)}\cdot 100\%,\nonumber\\ \delta J_{\rm R}^{\rm R}(\varepsilon)=\dfrac{J_{\rm R}^*(1,1) - J_{\rm R}^*(1,\varepsilon^2)}{J_{\rm R}^*(1,1)}\cdot 100\% \nonumber\\
\end{eqnarray}
be the relative cost improvements of the Manufacturer and the Retailer as cheap control players. In Fig. \ref{relative_cost_improvement-fig}, the relative improvements (\ref{relative_improvements}) are depicted, showing that the costs of the cheap control players are relatively improved in a similar way.

\begin{figure}[hbt!]
    \centering
    \includegraphics[width=0.98\linewidth]{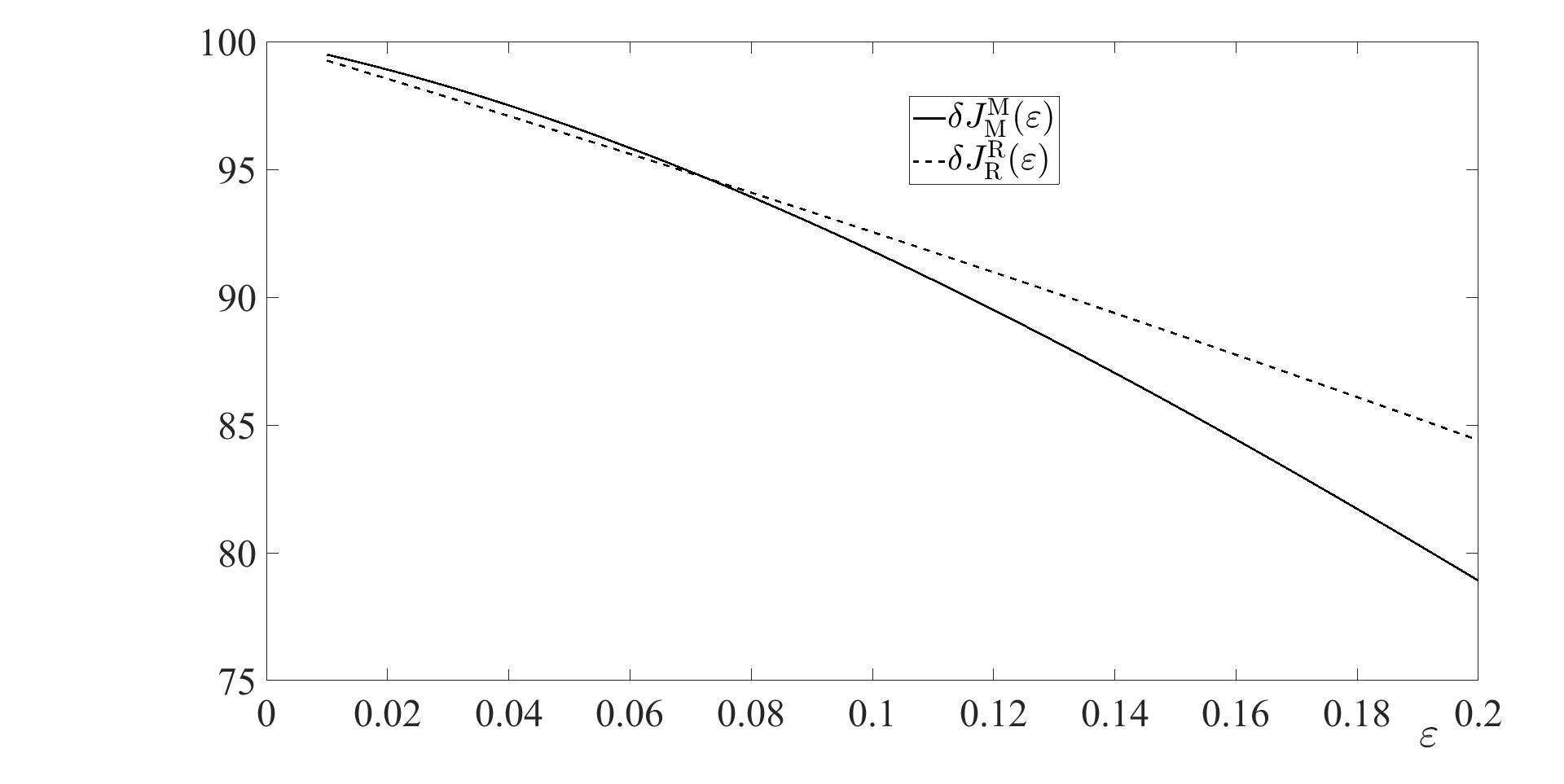}
    \caption{Relative cost improvement of the player with cheap control}
    \label{relative_cost_improvement-fig}
\end{figure}

Let
\begin{eqnarray}\label{relative_deteriorations}
\delta J_{\rm M}^{\rm R}(\varepsilon)=\dfrac{J_{\rm M}^*(1,\varepsilon^2)-{J_{\rm M}^*(1,1)}}{J_{\rm M}^*(1,1)}\cdot 100\%,\nonumber\\
\delta J_{\rm R}^{\rm M}(\varepsilon)=\dfrac{J_{\rm R}^*(\varepsilon^2,1)-{J_{\rm R}^*(1,1)}}{J_{\rm R}^*(1,1)}\cdot 100\% \nonumber\\
\end{eqnarray}
be the relative cost deteriorations of the Manufacturer and the Retailer as non-cheap control players, while the corresponding opponent is a cheap control player.

In Fig. \ref{relative_cost_deterioration-fig}, the relative deteriorations (\ref{relative_deteriorations}) are depicted. It is seen that the relative cost deterioration of the non-cheap control Retailer is considerably larger than the relative cost deterioration of the non-cheap control Manufacturer. This can be explained by the fact that the cheap-control Manufacturer enjoys two advantages - both as a leader and as a cheap control player.

\begin{figure}[hbt!]
    \centering
    \includegraphics[width=0.98\linewidth]{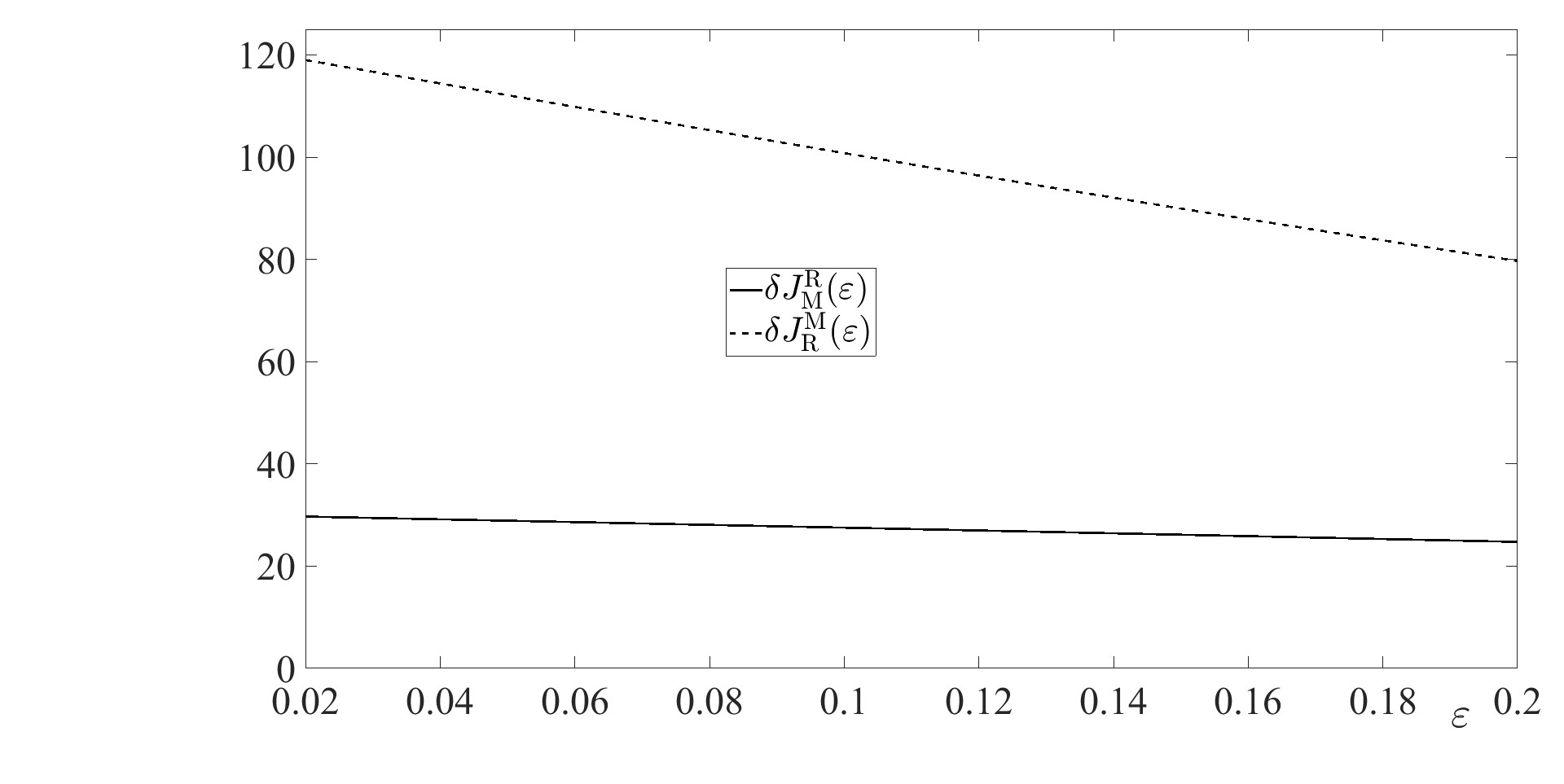}
    \caption{Relative cost deterioration of the player with non-cheap control}
    \label{relative_cost_deterioration-fig}
\end{figure}

\newpage

\section{Conclusions}

In this paper, a two-player finite horizon linear-quadratic Stackelberg differential game was considered in the case where the control cost of the follower in the cost functionals of both players is much smaller than the state cost and the cost of the leader's control. Due to this feature of the follower's control cost, the considered game is a cheap control game. For this game, an open-loop solution was sought. The proper change of the state variable transforms the initially formulated game to an equivalent, simpler, cheap control Stackelberg game. The dynamics equation of the new game consists of two modes. The first mode is controlled directly only by the leader, while the second mode is controlled directly by both players. Moreover, the dimension of the second mode equals the dimension of the follower's control, and the gain matrix for the follower's control in this mode is the identity matrix. The transformed game is also a cheap control game.
Based on the Stackelberg game's solvability conditions, the solution of the transformed game was converted to the solution of the linear singularly perturbed boundary-value problem, which dimension is four times larger than the dimension of the game's dynamics. This boundary-value problem is of the conditionally stable type. The differential system of this boundary-value problem has four slow state variables and four fast state variables. Based on the Boundary Functions Method, the first-order asymptotic solution to this singularly perturbed boundary-value problem was derived. Using this asymptotic solution, the asymptotic expansions of the optimal open-loop controls of the leader and the follower are obtained. Furthermore, suboptimal controls of the players are designed. Asymptotic approximations of the optimal values of the cost functionals are also derived. As an application of the theoretical results, an asymptotic solution of a supply chain problem with a single manufacturer and a single cheap control retailer was derived.

Completing this section, we would like to mention several issues connected with the topic of the paper, which are interesting ones for future investigations. These issues are the following: (a) asymptotic analysis of the feedback solution in the finite horizon Stackelberg linear-quadratic differential game with cheap control of a leader/follower; (b) the open-loop solution and the feedback solution of a finite horizon singular Stackelberg linear-quadratic differential game; (c) various real-life applications of the cheap control/singular Stackelberg linear-quadratic differential games.


\end{document}